\documentclass[11pt]{article}

\RequirePackage[OT1]{fontenc}
\RequirePackage{amsthm,amsmath}

\usepackage{latexsym,amsmath,amsfonts}

\usepackage{amssymb}
\usepackage{pstricks,pst-node,pst-tree}
\usepackage{thmtools,thm-restate}
\usepackage{graphicx,epsfig,epstopdf}
\usepackage{multirow}
\usepackage{natbib}
\usepackage{lastpage}
\usepackage{titletoc}
\usepackage{tikz}
\usepackage{pgfplots}
\usepackage{xcolor}

\usepackage{tikz}

\usepackage{xcolor}

\usepackage{amsmath}
\usepackage{amssymb}
\usepackage{graphicx}
\usepackage{xcolor}
\usepackage{array}
\usepackage{multirow}
\usepackage{caption}
\usepackage{enumitem}
\usepackage{natbib}
\usepackage{float}

\usepackage{fontenc}
\usepackage{multirow}
\usepackage{booktabs}
\usepackage{url}
\usetikzlibrary{shapes,arrows}

\usepackage{amsmath}
\usepackage{amssymb}
\usepackage{graphicx}
\usepackage{color}
\usepackage{hyperref}
\usepackage{url}
\usepackage{array}
\usepackage{multirow}
\usepackage{enumitem}
\usepackage{float}

\usepackage[margin=1.35truein]{geometry}

\def\style{authordate4}

\tikzstyle{block} = [rectangle, draw, fill=blue!20,
text width=5em, text centered, rounded corners, minimum width=8em,node distance=3cm,
minimum height=3em]
\tikzstyle{line} = [draw, -latex']
\tikzstyle{linena} = [draw]
\tikzstyle{cloud} = [draw, ellipse,fill=red!20, minimum height=2em]

\usetikzlibrary{arrows,calc,angles,positioning,intersections,quotes,decorations.markings,backgrounds,patterns}
\usetikzlibrary{decorations.pathreplacing}
\pgfplotsset{compat=1.11}
\tikzset{
	mynode/.style={fill,circle,inner sep=1pt,outer sep=0pt}
}

\usepackage{mathtools}  
\usepackage{tabulary}
\usepackage{booktabs}

\usepackage{tikz}
\usetikzlibrary{arrows,calc,angles,positioning,intersections,quotes,decorations.markings,backgrounds,patterns}
\usetikzlibrary{decorations.pathreplacing}
\usepackage{pgfplots}
\pgfplotsset{compat=1.11}
\tikzset{
	mynode/.style={fill,circle,inner sep=1pt,outer sep=0pt}
}

{\catcode `\@=11 \global\let\AddToReset=\@addtoreset}
\AddToReset{equation}{section}

\AddToReset{Theorem}{section}

\newtheorem{cor}{Corollary}[section]
\newtheorem{lem}{Lemma}[section]
\newtheorem{rem}{Remark}[section]
\newtheorem{thm}{Theorem}[section]

\newtheorem{Def}{Definition}[section]

\newcommand{\cA}{{\cal A}}

\newcommand{\cC}{{\cal C}}

\newcommand{\cG}{{\cal G}}

\newcommand{\cK}{{\cal K}}
\newcommand{\cL}{{\cal L}}
\newcommand{\cM}{{\cal M}}
\newcommand{\cN}{{\cal N}}

\newcommand{\cU}{{\cal U}}

\newcommand{\Np}{{\mathbb{N}}^{+}}
\newcommand{\Nn}{{\mathbb{N}}^{-}}

\def\ba{\begin{array}}
	\def\bc{\begin{center}}
		\def\bd{\begin{description}}
			\def\be{\begin{enumerate}}
				\def\ea{\end{array}}
			\def\ec{\end{center}}
		\def\ed{\end{description}}
	\def\edt{\end{document}}
\def\ee{\end{enumerate}}
\def\ben{\begin{equation}}
\def\benn{\begin{equation*}}
\def\een{\end{equation}}
\def\eenn{\end{equation*}}
\def\benr{\begin{eqnarray}}
\def\eenr{\end{eqnarray}}
\def\benrr{\begin{eqnarray*}}
\def\eenrr{\end{eqnarray*}}

\def\al{\alpha}

\def\D{\Delta}

\def\edt{\end{document}}
\def\ep{\epsilon}
\def\g{\gamma}
\def\G{\Gamma}
\def\h{\hat}
\def\ka{\kappa}

\def\iny{\infty}
\def\ka{\kappa}

\def\la{\lambda}

\def\noi{\noindent}

\def\nn{\nonumber}

\def\Om{\Omega}
\def\om{\omega}

\def\e{\mathrm e}

\def\si{\sigma}

\def\Si{\Sigma}

\def\vep{\varepsilon}

\def\vs{\vskip}

\def\R{{\mathbb R}}
\def\Z{{\mathbb Z}}
\def\z{\zeta}

\DeclareMathOperator*{\argmin}{arg\,min}

\DeclareMathOperator*{\argmax}{arg\,max}

\begin{document}

\bc
{\Large {\bf Inference on the Change Point for High Dimensional Dynamic Graphical Models}}\\[.5cm]
Abhishek Kaul$^a$\footnote{Email: abhishek.kaul@wsu.edu.}, Hongjin Zhang$^a$,\\ Konstantinos Tsampourakis$^b$, and George Michailidis$^c$\\[0.25cm]

$^a$Department of Mathematics and Statistics,\\ 
Washington State University, Pullman, WA 99164, USA.\\[0.25cm]

$^b$School of Mathematics\\ 
University of Edinburgh, Edinburgh, Scotland, EH9 3FD. \\[0.25cm]

$^c$Department of Statistics and the Informatics Institute\\
University of Florida, Gainsville, FL 32611-8545, USA.

\ec
\vs .1in
{\renewcommand{\baselinestretch}{1}
	\begin{abstract}
We develop an estimator for the change point parameter for a dynamically evolving graphical model, and also obtain its \textit{asymptotic distribution} under \textit{high dimensional scaling}. To procure the latter result, we establish that the proposed estimator exhibits an $O_p(\psi^{-2})$ rate of convergence, wherein $\psi$ represents the jump size between the graphical model parameters before and after the change point. Further, it retains sufficient adaptivity against plug-in estimates of the graphical model parameters. We characterize the forms of the asymptotic distribution under the both a vanishing and a non-vanishing regime of the magnitude of the jump size. Specifically, in the former case it corresponds to the \textit{argmax} of a negative drift asymmetric two sided Brownian motion, while in the latter case to the \textit{argmax} of a negative drift asymmetric two sided random walk, whose increments depend on the distribution of the graphical model. Easy to implement algorithms are provided for estimating the change point and their performance assessed on synthetic data. The proposed methodology is further illustrated on RNA-sequenced microbiome data and their changes between young and older individuals.		
\end{abstract} }
\noi {\it Keywords: high dimensions, dynamic graphical models, change point, inference, limiting distribution.}

\section{Introduction and problem formulation}\label{sec:intro}

Graphical models capture statistical dependencies amongst a collection of random variables. They have been extensively used in the analysis of genetics and genomics  \citep{sinoquet2014probabilistic}, metabolomics  \citep{basu2017sparse}, microbiome \citep{kaul2017structural} and neuroimaging data \citep{cribben2012dynamic}.

An undirected graphical model is a statistical model associated with a graph, whose nodes correspond to variables of interest (e.g., genes), while the edges reflect conditional dependencies amongst them. In many applications as above, the number of edges (parameters in the graphical model) to be estimated from the available data is relative small. This gave rise to a rich body of literature on recovery of \textit{sparse} graphical models. Likelihood based methods for Gaussian graphical models (GGMs) \citep{friedman2008sparse,yuan2010high} and regression based methods \citep{meinshausen2006high,cai2011constrained} leveraging $\ell_1$ penalties were developed for this task, and their theoretical properties established \citep{buhlmann2011statistics}.

Graphical models have also been used in applications where the data are collected over time. In that setting, the assumption of a fixed graphical model over an extended sampling period could be unrealistic and may lead to flawed inference on its structure. Hence, there is interest in estimating graphical models that evolve in a piecewise manner, characterized by one or more \textit{change points}. To that end, \cite{kolar2012estimating} consider fused lasso regularization together with a regression approach (neighborhood selection) to estimate a time evolving sparse GGM.  Likelihood based approaches together with suitable regularization  procedures are considered in \cite{kolar2010estimating,gibberd2017multiple,avanesov2018change}, while \cite{keshavarz2018sequential} develop an online detection problem of change in the GGM's structure. \cite{angelosante2011sparse} propose a dynamic programming algorithm together with neighborhood selection for the problem at hand, while \cite{roy2017change} provide a likelihood based approach for Markov random fields with a single change point. 

Note that the emphasis in the literature is primarily on estimating the location of the underlying change points and also the connectivity structure of the graphical models between change points. On the other hand, to the best of our knowledge, the question of uncertainty quantification through construction of confidence intervals for the change point under \textit{high dimensional scaling} for sparse graphical models has not been addressed in the literature. Further, the same question for a regularized linear regression problem under high dimensional scaling is also open. The latter problem is of independent interest, but also related to the main theme of this paper, since such regression problems constitute the main building block in the neighborhood selection method used to estimate graphical models.
Finally note that to do inference on the change point parameter, sharp convergence rates need to be established, an issue resolved for the graphical model in the sequel.

More generally, the results in the literature on inference on the change point encompassing other high dimensional models is also rather sparse. Under the simplest dynamic model which is of a mean shift, \cite{bhattacharjee2017common,bhattacharjee2019change} provide such limiting distributions for the \textit{single} change point parameter, in a regime where the dimensionality $p$ is smaller than the number of samples $T$ ($p<T)$. Under similar dimensional restrictions, \cite{bhattacharjee2018change} provide inference results for a single change point for a dynamically evolving stochastic block model. \cite{wang2019inference} provide a limiting distribution for the single change point parameter under a mean shift $p$ diverging, but at a slower rate than $T^2/\log T.$ The only article we are aware of that allows $p$ to grow exponentially, while allowing limiting distributions is that of \cite{kaul2020inference} under the same mean shift model.

\bigskip\noindent
\textbf{Change point problem formulation for the graphical model} We consider the following setting. Multivariate data  are collected for $T$ time periods and at a certain point during that time, their covariance matrix exhibits a change. Specifically, let
\vspace{-2mm}
\benr\label{model:dggm}
z_t=\begin{cases} w_t, & t=1,...,\lfloor T\tau^0\rfloor\\
	x_t, & t=\lfloor T\tau^0\rfloor+1,....,T,\end{cases}
\eenr
with $z_t\in\R^p,$ $t=1,...,T.$ The variables $w_t, x_t\in\R^p$ are independent and zero mean subgaussian random variables (r.v.'s), with unknown covariance matrices $\Si$ and $\D,$ respectively. The change point parameter $\tau^0\in (0,1)$ is unknown and needs to be estimated from the available data, together with the underlying covariance matrices.
We allow the dimension $p$ to diverge potentially at an exponential rate, i.e., $\log p=o(T^{\delta}),$ for some $0<\delta<1,$ while imposing a sparsity assumption on the inverse covariance (precision) matrices $\Si^{-1}$ and $\D^{-1}$, specified in Section \ref{sec:mainresults}.

We require additional notation to aid further discussion on the main objectives of this article. For any $p\times p$ matrix $W,$ define a $(p-1)$-dimensional vector $W_{-i,j}$ as the $j^{th}$ column of $W$ with the $i^{th}$ entry removed, and similarly define $W_{i,-j}.$ Also define a $(p-1)\times (p-1)$ matrix $W_{-i,-j}$ as the sub-matrix of $W$ with the $i^{th}$ row and the $j^{th}$ column removed. Next, define the following $\R^{p-1}$ parameter vectors 
\benr\label{def:muga}
\hspace{5mm}\mu^0_{(j)}= \Si_{-j,-j}^{-1}\Si_{-j,j},\quad{\rm and}\quad \g^0_{(j)}= \D_{-j,-j}^{-1}\D_{-j,j},\quad j=1,...,p.
\eenr

The parameters $\mu^0_{(j)},$ and $\g^0_{(j)}$'s correspond to the coefficients used in the neighborhood selection procedure. They can be directly related to the underlying graph as follows. When $\mu^{0}_{(j)k}=0$ ($k^{th}$ component of $\mu^0_{(j)}$) $\Leftrightarrow$ the $(j,k)^{th}$ entry of the corresponding precision matrix is zero, and thus indicates the absence of an edge between these nodes in the corresponding graph. These coefficients can also be interpreted through a linear regression mechanism, e.g. $\mu^0_{(j)}$ plays the role of a coefficient vector in the regression of the $j^{th}$ component of $z$ being the response, and the remaining ones as predictors. Next, we use them to characterize the magnitude of the jump size across the two graphical models. Specifically, let $\eta^0_{(j)}=\mu^0_{(j)}-\g^0_{(j)},$ $j=1,...,p,$ and define,
\benr\label{def:jumpsize}
\xi_{2,2}=\Big(\sum_{j=1}^p\|\eta^0_{(j)}\|_2^2\Big)^{\frac{1}{2}},\quad{\rm and}\quad \psi=\frac{\xi_{2,2}}{\surd p}.\footnotemark
\eenr
\footnotetext{The jump sizes $\xi_{2,2},$ and $\psi$ depend on the length of the sampling period $T;$ however, this dependence is suppressed for clarity of exposition.}
The quantities $\xi_{2,2},$ and $\psi$ reflect the magnitude of the difference between the pre- and post-change graphical models, the latter being a normalized version that plays a central role in subsequent analysis. Henceforth, we refer to $\psi$ as the \textit{jump size}. Note that $\xi_{2,2}$ or $\psi$ are non-zero, either if the conditional dependence structure (edges) exhibits changes, or the magnitude of the model parameters changes. This definition of jump size is somewhat similar to that in \cite{kolar2012estimating}, who define it as $\min_{j}\|\eta^0_{(j)}\|_2.$ The advantage of using $\psi$ over $\min_{j}\|\eta^0_{(j)}\|_2$ is that the latter requires changes in each and every row and column of the precision matrix, whereas the former allows for sub-block changes of the precision matrix pre and post the change point. Another metric of the jump size employed in the literature includes $\|\Si- \D\|_F$ \citep{gibberd2017multiple}, which is comparable to $\xi_{2,2}$.

\bigskip\noindent
\textbf{Change point estimation criterion function.}
Let $z_t\in\R^p, t=1,..,T$ and let $\mu,$ and $\g$ be the concatenation of $\mu_{(j)}'$s and $\g_{(j)}'$s. Then, consider the squared loss function
\benr
Q(z,\tau,\mu,\g)&=&\frac{1}{T}\Big[\sum_{t=1}^{\lfloor T\tau\rfloor}\sum_{j=1}^p(z_{tj}-z_{t,-j}^T\mu_{(j)})^2 
+ \sum_{t=(\lfloor T\tau\rfloor+1)}^{T}\sum_{j=1}^p(z_{tj}-z_{t,-j}^T\g_{(j)})^2\Big],\nn
\eenr
with $\tau\in (0,1)$.

Next, suppose that estimates for $\mu$ and $\g$ are available, so that the following bound holds:
\benr\label{eq:optimalmeans}
\max_{1\le j\le p}\Big(\|\h\mu_{(j)}-\mu^0_{(j)}\|_2\vee\|\h\g_{(j)}-\g^0_{(j)}\|_2\Big)\hspace{2.5cm}\\
\le c_u\surd(1+\nu^2)\frac{\si^2}{\ka}\Big\{\frac{s \log (p\vee T)}{Tl_T}\Big\}^{\frac{1}{2}}\hspace{-1.5cm},\nn
\eenr
with probability at least $1-o(1)$, with $l_T$ being a sequence separating the change point parameter from the boundary of its parametric space, i.e. $\lfloor T\tau^0\rfloor\wedge(T-\lfloor T\tau^0\rfloor)\ge Tl_T,$ (see, Condition A). The quantities $\si,\ka,\nu,$ are additional model parameters defined in Section \ref{sec:mainresults} (see, Condition B). Then, a plug-in estimator $\tilde\tau$ of the change point is given by
\benr\label{est:optimal}
\tilde\tau:=\tilde\tau(\h\mu,\h\g)=\argmin_{\tau\in(0,1)} Q(z,\tau,\h\mu,\h\g).
\eenr

\noindent
\textbf{Key contributions.} 

The main objective of this work is to establish that the change point estimator $\tilde\tau$ is sufficiently regular, so as to have a limiting distribution, thereby enabling the construction of asymptotically valid confidence intervals for $\tau^0$ under high dimensional scaling. The inference results obtained in Section \ref{sec:mainresults} are agnostic to the choice of the estimators used for $\mu, \g$, as long as the latter satisfy certain properties (see $\ell_2$ error bound (\ref{eq:optimalmeans})). A specific estimator for these parameters is presented in Section \ref{sec:nuisance}.

The first key contribution is establishing a sharp rate of convergence for the change point estimator; specifically, we obtain in Section \ref{sec:mainresults} that $\big(\lfloor T\tilde\tau\rfloor-\lfloor T\tau^0\rfloor\big)=O_p(\psi^{-2}).$ This rate is free of  auxiliary terms involving dimensional parameters $s,p$ and other logarithmic terms of the sampling period $T,$ whereas the dimension $p$  appears only through the jump size $\psi$. Further, the jump size can diminish to zero, provided that $s\log^{3/2}(p\vee T)=o\big(\surd(Tl_T)\big)$, with $s$ characterizing the sparsity of the graphical model (see, Section \ref{sec:mainresults}). Further, the obtained rate of convergence described above is sharper, and the minimum jump size assumption weaker than available results in the literature. For example, \cite{kolar2012estimating} obtain a rate of convergence $O_p(\psi^{-2}p\log T)$ under a minimum jump size assumption of order $O(p\log T\big/ T)^{1/2},$ \cite{gibberd2017multiple} a rate $O_p(\psi^{-2}p^2\log p)$ for jump size $O\{p \surd(\log p^{\beta/2}/T)\},$ \cite{li2019detection} a rate $O_p(\psi^{-2}\log^4 T),$ and finally \cite{roy2017change} provide a rate of $O_p\big(\psi^{-2}\log(pT)\big)$ for jump size of order $O(\log pT)^{1/4}$ for a Markov random field model.  The significance of this sharper convergence rate is that it leads to the existence of a \textit{limit distribution} for the change point estimator, under high dimensional scaling. Therefore, the second main contribution of this work is the derivation of this limit distribution under both a \textit{vanishing} ($\psi\rightarrow 0$) and \textit{non-vanishing} regime.

\bigskip\noindent
\textbf{Characterization of the limit distribution of $\tilde\tau$.} \\
\textit{Vanishing jump size $\psi$:} let $W_1(r),$ and $W_2(r)$ be two independent Brownian motions defined on $[0,\iny)$. Define the following process
\benr\label{def:Zr}
Z(r)=\begin{cases} 2W_1(r)-|r| &  {\rm if}\,\, r<0,\\
	0 &  {\rm if}\,\, r=0,\\
	\frac{2\si_2^*}{\si_1^*}W_2(r)-\frac{\si_2^2}{\si_1^2}|r| & {\rm if}\,\, r>0,	
\end{cases}
\eenr
where $0<\si_1,\si_2,\si_{1}^*,\si_{2}^*<\iny,$ are parameters that control both the variance and the negative drift of the process $Z(r).$ Then, for $\psi\to 0,$ we obtain,
\benr
T(\si^{*}_1)^{-2}\si_1^4\psi^{2}(\tilde\tau-\tau^0)\Rightarrow \argmax_{r\in \R} Z(r)\nn
\eenr 
The density of this limit distribution is available in closed form in \cite{bai1997estimation}.

\noindent
\textit{Non-vanishing jump size $\psi\to\psi_{\iny}$} ($0<\psi_{\iny}<\iny$): 

Let $\cL$ represent the form of the distribution of the limiting random variable of the sequence $p^{-1}\sum_{j=1}^p\Big\{2\vep_{tj}z_{t,-j}^T\eta^0_{(j)}-\eta^{0T}_{(j)}z_{t,-j}z_{t,-j}^T\eta^0_{(j)}\Big\},$ where $\vep_{tj}$ are r.v.'s measuring the orthogonal distance between $z_{tj}$ and the space of the remaining components, see, (\ref{def:epsilons}) for an explicit definition. Then define the following negative drift two sided random walk initialized at origin
\benr\label{def:cCr}
\cC_{\iny}(r)=
\begin{cases}\sum_{t=1}^{-r} z_t, & r\in \Nn=\{-1,-2,-3,...\}\\ 	
	0,				  &	r=0 \\
	\sum_{t=1}^{r}z_t^*,		  &	r\in \Np=\{1,2,3,...\}.
\end{cases}
\eenr
Further, $z_t\sim^{i.i.d}\cL\big(-\psi_{\iny}^2\si_1^2,\,\bar\si^{2}_{1}\big)$ and $z_t^*\sim^{i.i.d}\cL(-\psi_{\iny}^2\si_2^2,\,\bar\si^{2}_{2}),$ and $z_t$ and $z_t^*$ are also independent of each other over all $t.$ The notation in the arguments of $\cL(\cdotp,\cdotp)$ are representative of the mean and variance of this distribution. The quantities $0<\si_1,\si_2<\iny$ are the same as in the construction of the process $Z(r),$ and control the negative drift of the given two sided random walk. The parameters $0<\bar\si_1^2,\bar\si_2^2<\iny$ are estimable variance parameters of this limiting process which are different from those under the vanishing regime. Then, we obtain 
\benr
(\lfloor T\tilde\tau\rfloor-\lfloor T\tau^0\rfloor)\Rightarrow \argmax_{r\in \Z}\cC_{\iny}(r).
\eenr
This limit distribution does not have any explicit characterization, but its quantiles can be approximated numerically, thereby enabling the construction of asymptotically valid confidence intervals. 


\vspace{1.5mm}
\noi\textbf{\emph{Notation}}:  $\R$ shall denote the real line. For any vector $\delta,$ the norms $\|\delta\|_1,$ $\|\delta\|_2,$ $\|\delta\|_{\iny}$ represent the usual 1-norm, Euclidean norm, and sup-norm, respectively. For any set of indices $U\subseteq\{1,2,...,p\},$ let $\delta_U=(\delta_j)_{j\in U}$ represent the subvector of $\delta$ containing the components corresponding to the indices in $U.$ Let $|U|$ and $U^c$ represent the cardinality and complement of $U.$ We denote by $a\wedge b=\min\{a,b\},$ and $a\vee b=\max\{a,b\},$ for any $a,b\in\R.$ The notation $\lfloor \cdotp \rfloor$ is the usual greatest integer function. We use a generic notation $c_u>0$ to represent universal constants that do not depend on $T$ or any other model parameter. In the following this constant $c_u$ may be different from one term to the next. All limits in this article are with respect to the sample size $T\to\iny.$ We use $\Rightarrow$ to represent convergence in distribution.

\section{Theoretical analysis}\label{sec:mainresults}

Next, we state sufficient conditions required to establish the main theoretical results regarding the plugin least squares estimator $\tilde\tau$ in (\ref{est:optimal}). Specifically, an $O_p(\psi^{-2})$ rate of convergence is obtained for $\lfloor T\tilde\tau\rfloor$, together with its limiting distributions in the two \textit{regimes} discussed in Section \ref{sec:asymptotic-regime}. 


\subsection{Rate of convergence of the change point estimator}\label{sec:convergence-rate}

\hfill

\vspace{1.5mm}
{\it {{\noi{\bf Condition A (assumption on the model parameters):}} Let $S_{1j}=\{k;\,\mu^0_{{(j)}k}\ne 0\},$ and $S_{2j}=\{k;\,\g^0_{{(j)}k}\ne 0\},$ $1\le j\le p$ be sets of non-zero indices.\\~
		(i) Assume that $\max_{1\le j\le p} |S_{1j}|\vee |S_{2j}|=s\ge 1.$ \\~
		(ii) Assume a change point exists and is sufficiently separated from the boundaries of $(0,1),$ i.e., for some positive sequence $l_T\to 0,$ we have $\big(\lfloor T\tau^0\rfloor\big)\wedge \big(T-\lfloor T\tau^0\rfloor\big) \ge Tl_T\to\iny$\\~
		(iii) Let  $\psi$ be as defined in (\ref{def:jumpsize}). Then, for an appropriately chosen small enough constant $c_{u1}>0,$ the following relations hold,
		\benr
		&(a)&  c_u\surd(1+\nu^2)\frac{\si^2}{\psi\ka}\Big\{\frac{s\log^{3/2}(p\vee T)}{\surd (Tl_T)}\Big\}\le c_{u1},\,\,{\rm and}\nn\\
		&(b)& c_u\surd(1+\nu^2)\frac{\si^2}{\psi\ka}\Bigg\{\frac{s\log (p\vee T)}{T^{\big(\frac{1}{2}-b\big)}\surd l_T}\Bigg\}\le c_{u1},\nn
		\eenr
		for some $0<b<(1/2).$ The parameters $\si^2,\nu,\ka$ are defined in Condition B.}}

Condition A controls the rate at which the sparsity level of the graphical model $s$ and its dimension $p$ diverge as a function of $T$. Further, it specifies the behavior of the normalized jump size $\psi$ and the distance $l_T$ of the change point from the boundary of the observation interval, as a function of $T$. Condition A(iii) encompasses the two regimes of interest on the asymptotic behavior of the jump size. Specifically, it allows for a potentially vanishing jump size, $\psi\to 0,$ when $s\log^{3/2}(p\vee T)=o\big(\surd (Tl_T)\big).$ Alternatively, $s,p$ can diverge at an arbitrary rate provided the jump size is large enough to compensate for the increasing dimensions $s,p,$ so that Condition A(iii) holds (also see Remark \ref{rem:dimension.restriction}). 

To the best of our knowledge, this is the weakest condition assumed on the jump size in the dynamic networks literature, where the counterpart of $\psi$ is typically assumed to be diverging. The constant $b>0$ in A(iii)(b) is any arbitrary, but fixed number between $(0,1/2).$ The rate conditions (a) and (b) of A(iii) are stated in the given form to provide generality and neither (a) or (b) necessarily implies the other without additional rate restrictions; for example, (b) implies (a) if $\log p\le c_uT^{2b},$ while (a) implies (b) if $\log p \ge c_uT^{2b}.$ Condition A(iii)(b) is an assumption that arises in our analysis of the regression type estimator for the change point parameter $\tilde\tau$. This assumption can be compared to existing results on inference for change points in the classical fixed dimensional regression setting. For fixed $s,p,l_T$, the rate required for the minimum jump size $\psi$ in Part (iii) can be replaced with $ T^{\big(\frac{1}{2}-b\big)}\psi\to \iny.$ This condition is identical to Assumption A7 in \cite{bai1997estimation} and serves an analogous role in our analysis.

Sparsity on coefficient vectors $\mu^0_{(j)}$ and $\g^0_{(j)}$ is equivalent to assuming that both pre and post network structures of $\Si^{-1}$ and $\D^{-1}$ are such that each node has at most $s$ connecting edges out of a total of $(p-1)$ possible edges. This is a direct extension of the same assumption in the static setting \citep{yuan2010high}. We also note that this sparsity assumption in our setting holds column- or row-wise on the underlying precision matrices. In other settings, such as high dimensional vector autoregressive models, sparsity is often assumed on the entire $p\times p$ coefficient matrix; this distinction is important for any heuristic comparisons made on rate assumptions across settings.    

\vspace{1.5mm}
{\it {{\noi{\bf Condition B (assumption on the underlying distributions):}}\\~
		(i) The vectors $w_t=(w_{t1},...,w_{tp})^T,$ $t=1,..,\lfloor T\tau^0\rfloor,$ and $x_t=(x_{t1},...,x_{tp})^T,$ $t=\lfloor T\tau^0\rfloor+1,...T,$ are independent subgaussian r.v's with mean vector zero, and variance proxy $\si^2\le c_u.$ (see Definition \ref{def:subg})\\~
		(ii) The $p$-dimensional matrices $\Sigma:=Ew_tw_t^T$ and $\Delta:=Ex_tx_t^T$ have bounded eigenvalues, i.e., $0<\ka\le\big\{\rm{min eigen}(\Si)\wedge {\rm min eigen}(\D)\big\} \le \big\{{\rm max eigen}(\Si)\vee {\rm max eigen}(\D)\big\}\le\phi<\iny.$ Consequently, the condition numbers of $\Si$ and $\D$ are also bounded above by $\nu=\phi/\ka.$ }}

The sub-Gaussian assumption represents a significant relaxation to assuming a Gaussian distribution, since it allows asymmetric distributions, including a centered mixture of two Gaussian distributions. Our methodology allows this general setup since $\tilde\tau$ is estimated using least squares, as opposed to a likelihood based approach used for GGM's. This condition serves the following three purposes. First, it allows the residual process in the estimation of $\tau^0$ to converge weakly to the distribution (\ref{def:Zr}). Second, under a suitable choice of regularization parameters, it allows estimation of nuisance parameters at the rates of convergence presented in (\ref{eq:optimalmeans}). Finally, in addition to other technical uses, part (ii) of this condition provides an upper bound on the components of $\mu^0_{(j)}$  and $\g^0_{(j)},$ $j=1,...,p,$, which is necessary to our analysis (Lemma \ref{lem:condnumberbound}). For the remainder of the presentation in the current section, we are \textit{agnostic} regarding the choice of the estimator of the nuisance parameters and instead require the following condition.

\vspace{1.5mm}
{\it {{\noi{\bf Condition C (assumption on nuisance parameter estimates):}} Let $\pi_T\to 0$ be a positive sequence. Then, with probability $1-\pi_T,$ the following relations are assumed to hold. \\~
		(i) The vectors $\h\mu_{(j)}$ and $\h\g_{(j)},$ $1\le j\le p,$ satisfy the bound (\ref{eq:optimalmeans}). \\~
		(ii) The vectors $(\h\mu_{(j)}-\mu^0_{(j)})\in\cA_{1j},$ $(\h\g_{(j)}-\g^0_{(j)})\in\cA_{2j},$ for each $1\le j\le p.$ Here $\cA_{ij},$ $i=1,2,$ $j=1,...,p,$ is a convex subset of $\R^{p-1}$ defined as, ${\cA_{ij}}=\big\{\delta\in\R^{p-1};\,\,\|\delta_{S_{ij}^c}\|_1\le 3\|\delta_{S_{ij}}\|_1\big\},$ with $S_{ij}$ being the set of indices defined in Condition A(i) and $S_{ij}^c$ being its complement set.}}

This condition is a mild requirement and is known to hold in the static setting by common precision matrix estimation methods, including neighborhood selection \citep{meinshausen2006high,yuan2010high}. Condition C(ii) provides a restriction on the sparsity level of the estimated edge parameters and is common in the $
\ell_1$ regularization literature. In Section \ref{sec:nuisance}, the estimates of the nuisance parameters developed satisfy this condition. Further, other common regularization mechanisms, such as SCAD or the Dantzig selector are also applicable.

This condition allows estimates $\h\mu_{(j)}$ and $\h\g_{(j)}$ to be irregular, in the sense that they are only required to be in a $\{s\log (p\vee T)/T\}^{1/2}$ order neighborhood of the vectors $\mu^0_{(j)}$ and $\g^0_{(j)},$ $j=1,...,p,$ in the $\ell_2$ norm. They are not required to possess oracle properties, i.e., selection mistakes in the identification of the signs of these coefficient do not influence the eventual change point estimate $\tilde\tau$ in its rate of convergence and limiting distribution. Accordingly, we do not require irrepresentable conditions on the covariance matrices $\Si$ and $\D,$ as assumed in \cite{kolar2012estimating}, nor minimum magnitude conditions of the coefficient vectors $\mu^0_{(j)},$ $\g^0_{(j)},$ the latter again guaranteeing highly accurate selection in the components of  $\mu^0_{(j)},$ and $\g^0_{(j)},$ $j=1,...,p.$

Next, define for $\mu,\g\in\R^{p(p-1)}$ and $\tau\in (0,1),$
\benr\label{def:cU}
\cU(z,\tau,\mu,\g)=Q(z,\tau,\mu,\g)-Q(z,\tau^0,\mu,\g),\nn
\eenr
where $\tau^{0}\in(0,1)$ is the unknown change point parameter and $Q(z,\tau,\mu,\g)$ is the squared loss defined earlier. For any non-negative sequences $0\le v_T\le u_T\le 1,$ define the collection
\benr\label{def:setG}
\cG(u_T,v_T)=\Big\{\tau\in (0,1);\,\,Tv_T\le \big|\lfloor T\tau \rfloor-\lfloor T\tau^0\rfloor\big|\le Tu_T\Big\}
\eenr
The following Lemma provides a uniform lower bound on the expression $\cU(z, \tau,\h\mu,\h\g),$ over the collection $\cG(u_T,v_T)$ that is instrumental to obtain the desired rate of convergence for the proposed estimator.

\begin{lem}\label{lem:mainlowerb} Suppose Condition A, B and C hold and let $0\le v_T\le u_T$ be any non-negative sequences. For any $0<a<1,$ let $c_{a1}=4\cdotp 48c_{a2},$ with $c_{a2}\ge \surd(1/a),$ and
	\benr
	c_{a3}=c_u\Big\{\frac{c_{a1}(\si^2\vee\phi)\surd(1+\nu^2)}{\ka(1\wedge \psi)}\Big\}.\nn
	\eenr	
	Additionally, let $u_T\ge c_{a1}^2\si^4\big/(T\phi^2),$ then for $T\ge 2,$ we have,	
	\benr\label{eq:12}
	\inf_{\tau\in\cG(u_T,v_T)}\cU(z,\tau,\h\mu,\h\g)\ge \ka\xi^{2}_{2,2}\Big[v_T-c_{a3}\max\Big\{\Big(\frac{u_T}{T}\Big)^{\frac{1}{2}},\,\,\frac{u_T}{T^{b}}\Big\}\Big]
	\eenr
	with probability at least $1-3a-o(1).$
\end{lem}

Lemma \ref{lem:mainlowerb} is a tool that allows us to obtain the rate of convergence of the change point estimator $\tilde\tau.$ An observation that provides some insight into this connection and the adaptivity property of the proposed plug-in least squares estimator is as follows. Although $\cU(z,\tau,\h\mu,\h\g)$ involves the $p$-dimensional r.v.'s $z_t,$ and the estimates $\h\mu_{(j)},$ and $\h\g_{(j)}$ which approximate $(p-1)$-dimensional unknown parameters $\mu^0_{(j)}$ and $\g^0_{(j)},$ $j=1,...,p,$ up to the rate $O\big(\surd(s\log p/T)\big),$ yet, the eventual lower bound of Lemma \ref{lem:mainlowerb} is free of the dimensions $s,p$ under the assumed conditions. Intuitively, the plug-in least squares estimator of the change point behaves as if the nuisance parameters $\mu^0_{(j)}$ and $\g^0_{(j)}$ are \textit{known}. This is a key property that dictates the rate of convergence established in the next Theorem. Further insight on the inner workings of this result is provided in Remark \ref{rem:kolmogorov}.

\begin{thm}\label{thm:optimalapprox} Suppose Conditions A, B and C hold, and for any $0<a<1,$ let $c_{a1},c_{a2}$ and $c_{a3}$ be as defined in Lemma \ref{lem:mainlowerb}. Then, for $T$ sufficiently large the following hold:\\
	(i) When $\psi\to 0$,  $(1+\nu^2)^{-1}(\si^2\vee\phi)^{-2}\ka^2\psi^2\big|\lfloor T\tilde\tau\rfloor-\lfloor T\tau^0\rfloor
	\big|\le c_u^2c_{a1}^2,$ with probability at least $1-3a-o(1).$ Equivalently, in this case we get that $\psi^2\big(\lfloor T\tilde\tau\rfloor-\lfloor T\tau^0\rfloor
	\big)=O_p(1).$ \\~	
	(ii) When $\psi\not\to 0,$ we have, $\big|\lfloor T\tilde\tau\rfloor-\lfloor T\tau^0\rfloor
	\big|\le c_{a3}^2,$ with probability at least $1-3a-o(1).$ Equivalently, in this case we obtain $\big(\lfloor T\tilde\tau\rfloor-\lfloor T\tau^0\rfloor\big)=O_p(1).$
\end{thm}

Theorem \ref{thm:optimalapprox} provides a $O_p(1)$ bound, wherein the bounding constant $c_a$ depends on the probability of the bound. This is in contrast to existing localizing bounds in the literature, for e.g. an $O\big(\log (p\vee T)\big)$ bound in \cite{roy2017change} that holds with probability $1-o(1),$ namely, the bounding constant is free of the probability of the bound.

\begin{rem}\label{rem:dimension.restriction} (On dimensional rate assumptions) {\rm One may observe that Theorem \ref{thm:optimalapprox} is obtained without any explicit restriction on the rate of divergence of $s$ and $p$ with respect to the sampling period $T,$ and is based on their inter-relationship with the jump size $\psi.$ The result holds true for $s,p$ diverging at an arbitrary rate with respect to $T,$ as long as the jump size $\psi$ is large enough to compensate in order to preserve Condition A(iii). This is however not the complete picture. Effectively, this result has transferred the burden of an additional assumption controlling the divergence of $s,p$ to Condition C on the nuisance parameter estimates. In order to obtain feasible estimates of the latter, an additional assumption of the form $s\log p=o(Tl_T)$ is required \big(see, Condition A$'$(i) and Theorem \ref{thm:alg1.nearoptimal} in Section \ref{sec:nuisance}\big).  }   	
\end{rem}

The following remark provides insight on how Lemma \ref{lem:mainlowerb} and Theorem \ref{thm:optimalapprox} eliminate dimensional parameters $s,p$ and other logarithmic terms of $T$ to obtain the rate of convergence.
To aid presentation, define for each $j=1,...,p,$ the following r.v.'s,
\benr\label{def:epsilons}
\vep_{tj}=\begin{cases}z_{tj}-z_{t,-j}^T\mu^0_{(j)}, & t=1,...,\lfloor T\tau^0\rfloor\\
	z_{tj}-z_{t,-j}^T\g^0_{(j)}, & t=\lfloor T\tau^0\rfloor+1,...,T.\end{cases}
\eenr

\begin{rem}\label{rem:kolmogorov} {\rm The behavior of the estimator $\tilde\tau,$ is in part controlled by a stochastic noise term of the form,
		\benr
		\sup_{\tau;\, \tau\ge\tau^0} \xi_{2,2}^{-1}\Big|\sum_{t=\lfloor T\tau^0\rfloor}^{\lfloor T\tau\rfloor} \sum_{j=1}^p \vep_{tj}z_{t,-j}^T\h\eta_{(j)}\Big|,\quad{\rm where}\,\, \h\eta_{(j)}=\h\mu_{(j)}-\h\g_{(j)}, \nn
		\eenr
		and its mirroring counterpart, wherein $\vep_{tj}$ is as defined in (\ref{def:epsilons}). Note the need for uniformity over $\tau$ of this stochastic term. A large proportion of the literature upper bounds such uniform stochastic terms using subexponential type tail bounds and obtains uniformity over $\tau$ by means of union bounds over the at most $T$ distinct values $\lfloor T\tau\rfloor.$ Thus, logarithmic terms of $T$ end up appearing in the upper bound for this stochastic term, which transfers over to the eventual bound for the change point estimate. Additionally, dimensional parameters $s,p$ also often show up, depending upon how one chooses to control the nuisance estimates $\h\eta_{(j)}.$ This approach is insufficient for inference, since it does not yield an $O_p(\psi^{-2})$ rate of convergence; in other words, it does not establish uniform tightness of the sequence $\psi^{2}\big(\lfloor T\tilde\tau\rfloor-\lfloor T\tau^0\rfloor\big),$ which in turn is necessary for the existence of a limiting distribution. To overcome this problem, we develop a novel application of Kolmogorov's inequality (Theorem \ref{thm:kolmogorov}) on partial sums in order to control such stochastic terms with sharper upper bounds. This is achieved by first using a triangle inequality,
		\benr
		\sup_{\tau;\, \tau\ge\tau^0} \xi_{2,2}^{-1}\Big|\sum_{t=\lfloor T\tau^0\rfloor}^{\lfloor T\tau\rfloor} \sum_{j=1}^p \vep_{tj}z_{t,-j}^T\h\eta_{(j)}\Big|\le \sup_{\tau;\, \tau\ge\tau^0} \xi_{2,2}^{-1}\Big|\sum_{t=\lfloor T\tau^0\rfloor}^{\lfloor T\tau\rfloor} \sum_{j=1}^p \vep_{tj}z_{t,-j}^T\eta_{(j)}^0\Big|\nn\\
		+\sup_{\tau;\, \tau\ge\tau^0} \xi_{2,2}^{-1}\Big|\sum_{t=\lfloor T\tau^0\rfloor}^{\lfloor T\tau\rfloor} \sum_{j=1}^p \vep_{tj}z_{t,-j}^T(\h\eta_{(j)}-\eta_{(j)}^0)\Big|.\nn
		\eenr
		The first term on the rhs can now be controlled at an optimal rate $O(\surd{T}),$ (see, Lemma \ref{lem:optimalcross} and Lemma \ref{lem:optimalsqterm}) without any additional logarithmic terms of $T,$ leveraging Kolmogorov's inequality. Moreover, under Conditions A and C, the second term on the rhs of the above inequality can also be controlled with the same upper bound, despite high dimensionality and without dimensional parameters $s,p,$ being involved in the upper bound (see, Lemma \ref{lem:nearoptimalcross}, Lemma \ref{lem:term123} and the proof of Lemma \ref{lem:mainlowerb}). This provides the desired sharper control on the stochastic noise terms and consequently allows for the rate of convergence presented in Theorem \ref{thm:optimalapprox}.}
\end{rem}

\subsection{Asymptotic distribution of the change point estimator}
\label{sec:asymptotic-regime}

To obtain the asymptotic distribution the following technical condition is required. \\
\vspace{1.5mm}
{\it {{\noi{\bf Condition D:}} (i) Given covariance $\Si$ and $\D$, the following limits  exist,
		\benr\label{eq:addasm}
		\xi_{2,2}^{-2}\sum_{j=1}^p\eta^{0T}_{(j)}\Si_{-j,-j}\eta^0_{(j)}\to \si_1^2,\quad{\rm and}\quad\xi_{2,2}^{-2}\sum_{j=1}^p\eta^{0T}_{(j)}\D_{-j,-j}\eta^0_{(j)}\to \si_2^2,\quad 0<\si_1^2,\si_2^2<\iny.\nn
		\eenr
		(ii) For $\vep_{tj},$ for $t=1,...,T,$ and $j=1,...,p,$ as defined in (\ref{def:epsilons}), assume that,
		\benr
		\xi_{2,2}^{-2}p^{-1}{\rm var}\Big(\sum_{j=1}^p\vep_{tj}z_{t,-j}^T\eta^0_{(j)}\Big)\to \si_1^{*2},&&\quad{\rm for}\,\,t=1,...,\lfloor T\tau^0\rfloor\,\,{\rm and},\nn\\
		\xi_{2,2}^{-2}p^{-1}{\rm var}\Big(\sum_{j=1}^p\vep_{tj}z_{t,-j}^T\eta^0_{(j)}\Big)\to \si_2^{*2},&&\quad{\rm for}\,\,t=\lfloor T\tau^0\rfloor+1,...,T,\nn
		\eenr
		where $ 0<\si_1^{*2},\si_2^{*2}<\iny.$
}}

Recall that all limits in this article are with respect to the sampling period $T.$ The limits of Condition D are acting in $T$ via the dimension $p$ and the jump size $\xi_{2,2}.$ As briefly described earlier in the construction of limiting processes (\ref{def:Zr}) and (\ref{def:cCr}), the limits $\si_1^2$ and $\si_2^2$ control the magnitude of the negative drift of the two components of these processes. On the other hand, the limits $\si_1^{*2}$ and $\si_2^{*2}$ control the variance of the process (\ref{def:Zr}).  

Note that finiteness of the limits appearing in Condition D are already guaranteed by prior assumptions, and this condition only assumes their stability. To see this,  first consider Condition D(i) and note that the assumed convergence is on a sequence that is guaranteed to be bounded, i.e.,
\benr
\ka \xi_{2,2}^2\le \sum_{j=1}^p \eta^{0T}_{(j)}\Si_{-j,-j}\eta^{0}_{(j)}\le \phi\xi_{2,2}^2,\nn
\eenr
wherein the inequalities follow from the bounded eigenvalues assumption on the covariance matrix $\Si$ \big(Condition B(ii)\big), and analogously for the post-change covariance matrix $\D.$ An easier to interpret, but stronger sufficient condition for the finiteness for the limits in Condition D(i) is as follows. Let $\Si=\big[\si_{ij}\big]_{i,j=1,...,p},$ and $\D$ be symmetric matrices such that,
\benr
\|\Si\|_1=\max_{1\le j\le p}\sum_{ij}|\si_{ij}|<\iny,\nn
\eenr
and analogous for the matrix $\D.$ Then, we have,
\benr
\xi_{2,2}^{-2}\sum_{j=1}^p\eta^{0T}_{(j)}\Si_{-j,-j}\eta^0_{(j)}\le \|\Si\|_{\iny}\|\Si\|_1\xi_{2,2}^{-2}\sum_{j=1}^{p}\|\eta_{(j)}\|_2^2= \|\Si\|_{\iny}\|\Si\|_1<\iny,\nn
\eenr
where the inequality follows from the relation $\|\Si\|_2^2\le \|\Si\|_{\iny}\|\Si\|_1,$ with $\|\Si\|_2$ denoting the operator norm. In other words, finiteness of the assumed limits of D(i) are guaranteed by absolute summability of components of each row (or column) of the underlying covariances, which are in turn satisfied by large classes of such matrices, including Toeplitz and banded ones.

Next, finiteness of the assumed limits of D(ii) can be illustrated by using properties of subgaussian distributions assumed earlier in Condition B. Specifically, let $\z_{tj}=\vep_{tj}z_{t,-j}^T\eta^0$ and $\z_t=\sum_{j=1}^p\z_{tj},$ and note that $E(\z_{t})=0$. Further, using part (ii) of Lemma \ref{lem:subez} we get that $\z_t\sim{\rm subE}(\la),$ $\la=O(\xi_{2,1}),$ with $\xi_{2,1}=\sum_{j=1}^p\|\eta^0_{(j)}\|_2.$ Hence, $\xi_{2,2}^{-1}p^{-1}{\rm var}(\z_t)=O\big(\xi_{2,1}^2\big/p\xi_{2,2}^2\big)=O(1)<\iny,$ which follows by utilizing the elementary relation $\xi_{2,1}\le \surd{p}\xi_{2,2}$ between the $1-$norm and $2-$norm. 

Next, we state the result for the asymptotic distribution of the change point estimator for the \textit{vanishing jump size regime} $\psi\to 0.$ 

\begin{thm}\label{thm:limitingdist} (Vanishing jump size regime) Suppose Conditions A, B, C, and D hold. Further, assume that $\psi\to 0,$ while satisfying, 
	\benr\label{eq:rateextra}
	\frac{1}{\psi}\Big\{\frac{s\log^{3/2} (p\vee T)}{\surd (Tl_T)}\Big\}= o(1).
	\eenr	
	Then, the estimator $\tilde\tau$ of (\ref{est:optimal}) has the following limiting distribution.
	\benr
	T(\si_1^{*})^{-2}\si_1^4\psi^2(\tilde\tau-\tau^0)\Rightarrow \argmax_{r\in\R} Z(r).\nn
	\eenr
	where $Z(r)$ is as defined in (\ref{def:Zr}).
\end{thm}

The density function of this limiting distribution is readily available in \cite{bai1997estimation},  thereby allowing straightforward computation of its quantiles. The only difference between assumption (\ref{eq:rateextra}) and the rate restriction of Condition A(iii) is that the rhs has been tightened to $o(1)$ from $O(1).$ This slightly stronger requirement for the existence of the limiting distribution is in coherence with classical results in the literature \citep{bai1994,bai1997estimation}.

\begin{rem} (On adaptation) {\rm Note that the posited limiting distribution is the same as one would obtain when the nuisance parameters $\mu^0,$ $\g^0$ were \textit{known}. This is despite $\tilde\tau$ utilizing $2p$ estimated vectors $\h\mu_{(j)}$ and $\h\g_{(j)},$ $j=1,...,p,$ each of dimension $(p-1).$ This is effectively the adaptation property as described in \cite{bickel1982adaptive}, but in a high dimensional setting and within a change point parameter context.}
\end{rem}

A note of interest concerns the jump size scaling of Condition A and its relation to the inference properties in high dimensional dynamic models. We note that the scaling $\psi\ge c_us\log^{3/2}(p\vee T)/\surd{Tl_T}$ viewed from a sparsity ($s$) perspective assumes a more sparse regime than the scaling $\psi\ge c_u\big\{s\log(p\vee T)/Tl_T\big\}^{1/2}$ for which near optimal estimation results have been established in context of other dynamic models such as that of linear regression, see, e.g. \cite{rinaldo2020localizing}. Assuming an increased sparsity level is a key distinction that makes the inference results feasible. While we only prove sufficiency of this assumption and not its necessity, however, some evidence pointing to the sharpness of this assumption follows. In a linear regression framework, Lemma 4 of \cite{rinaldo2020localizing} shows that the minimax optimal rate of estimation under a scaling  $\psi\ge c_u\big\{s/T\big\}^{1/2}$ is $O_p(s\psi^{-2}),$ i.e., slower than $O_p(\psi^{-2})$ obtained above and in turn disallowing inference. Thus, at the very least, one may conclude that the sparsity level necessary for feasibility of inference should be diverging at a slower rate such as that assumed in Condition A.  Additional indirect evidence for the sharpness of this super-sparse scaling arises from recent results on inference for a regression coefficient in the presence of high dimensionality. The debiased lasso (\cite{van2014asymptotically}) and orthogonalized moment estimators \citep{belloni2011inference,belloniinference,belloni2017confidence} and \cite{ning2017general} developed for this purpose, require a similar super-sparsity assumption $s\log p/\surd{T}=o(1)$ for validity of inference results, over an ordinary sparsity assumption $s\log p/T=o(1),$ the latter permitting only near optimal estimation properties. The necessity of this assumption remains unknown in this regression coefficient setting as well, however it is the sharpest sufficient condition currently available.

Next, we obtain the limiting distribution in the non-vanishing jump size regime $\psi\to\psi_{\iny},$ $0<\psi_{\iny}<\iny.$ Note that available results in the literature for this non-vanishing regime are primarily available for mean shift models either for fixed $p$ \citep{jandhyala1999capturing,fotopoulos2010exact}, or for growing $p$, but dense settings \cite{bhattacharjee2017common,bhattacharjee2019change,wang2020dating}. Further, the first two papers require $p$ diverging more slowly than $T,$ while the last one requiring $p$ diverging more slowly than $T^2/\log T.$  \cite{kaul2020inference} provides an analysis of the latter case under high dimensional scaling, with $p$ potentially diverging exponentially with $T.$ 

To proceed further, we require an additional distributional assumption, as explained next. The stochastic term that controls the change point estimator $\tilde\tau$ has a distribution of the form $\psi\sum_{t=1}^{r\psi^{-2}}u_t,$ for constant $r>0$, with $u_t$ being independent random variables of finite variance. In the vanishing regime $\psi\to 0,$ we have that $r\psi^{-2}\to \iny$ and thus a functional central limit theorem becomes applicable, yielding a Brownian motion as the resulting process over $r.$ On the other hand, in the non-vanishing regime $\psi\to \psi_{\iny},$ the stochastic term described earlier is no longer over a diverging number of r.v.'s, and is instead a sum of a finite number of finite variance ones. Thus, central limit theoretic results are no longer applicable on this sum, and thus under this non-vanishing case one requires a further parametric assumption on the underlying distribution to characterize the distribution of the above described term. This condition is stated below.

\vspace{1.5mm}
{\it {{\noi{\bf Condition B$'$ (further distributional assumption):}} Suppose Conditions B and D hold. let $\si_1^2,\si_2^2$ be as defined in Condition D and let $\bar\si_1^2=\lim_{T}{\rm var}\Big[	p^{-1}\sum_{j=1}^p\Big\{2\vep_{tj}z_{t,-j}^{T}\eta^0_{(j)}-\eta^{0T}_{(j)}z_{t,-j}z_{t,-j}^T\eta^0_{(j)}\Big\}\Big],$ $t\le\lfloor T\tau^0\rfloor,$ and similarly define $\bar\si_2^2$ for $t>\lfloor T\tau^0\rfloor,$ such that $0<{\bar\si}_1^{2},{\bar\si}_2^{2}<\iny.$ Then, assume 
		\benr
		p^{-1}\sum_{j=1}^p\Big\{2\vep_{tj}z_{t,-j}^{T}\eta^0_{(j)}-\eta^{0T}_{(j)}z_{t,-j}z_{t,-j}^T\eta^0_{(j)}\Big\}&\Rightarrow& \cL\big(-\psi_{\iny}^2\si_1^2,\,{\bar\si}_1^{2}\big),\quad t\le \lfloor T\tau^0\rfloor\nn\\
		p^{-1}\sum_{j=1}^p\Big\{2\vep_{tj}z_{t,-j}^{T}\eta^0_{(j)}-\eta^{0T}_{(j)}z_{t,-j}z_{t,-j}^T\eta^0_{(j)}\Big\}&\Rightarrow& \cL\big(-\psi_{\iny}^2\si_2^2,\,{\bar\si}_2^{2}\big),\quad t> \lfloor T\tau^0\rfloor\nn
		\eenr
		for some distribution law $\cL$ which is continuous and supported in $\R.$}}

Note that the only additional requirement imposed by Condition B$'$, in comparison to Conditions B and D, is that the random variables under consideration are continuously distributed,  which is also clearly true in the typical Gaussian graphical model framework. The arguments in the notation $\cL(\mu,\si^2)$ are used to represent the mean and variance of the distribution $\cL,$ i.e,  $E\cL(\mu,\si^2)=\mu,$ and ${\rm var}\big(\cL(\mu,\si^2)\big)=\si^2.$  Further note that the representation $\cL(\mu,\si^2)$ is only for ease of presentation and does not imply that $\cL$ is characterized by only its mean and variance. 

Next, consider the mean of the sequence of r.v.'s under consideration for $t\le\lfloor T\tau^0\rfloor$ 
\benr
Ep^{-1}\sum_{j=1}^p\Big\{\vep_{tj}z_{t,-j}^{T}\eta^0_{(j)}-\eta^{0T}_{(j)}z_{t,-j}z_{t,-j}^T\eta^0_{(j)}\Big\}=-p^{-1}\sum_{j=1}^p\eta^{0T}_{(j)}\Si\eta^0_{(j)}\to-\psi_{\iny}^2\si_1^2,\nn
\eenr
and analogously for $t>\lfloor T\tau^0\rfloor.$ The equality follows since $E\eta^{0T}_{(j)}z_{t,-j}z_{t,-j}^T\eta_{(j)}^0=\eta^{0T}_{(j)}\Si\eta^0_{(j)},$ and  $E\vep_{tj}=Ez_{t,-j}^T\eta^0_{(j)}=0,$ and moreover, $\vep_{tj}$ and $z_{t,-j}$ are uncorrelated by construction in (\ref{def:epsilons}). Then, convergence in expected value follows from Condition D(i) provided that $\psi\to\psi_{\iny}.$ 

Next, we consider the variance terms of these random variables. From the properties of subgaussian and subexponential distributions (also see, discussion after Condition D), we have, 
\benr\label{eq:finite.var}
&&\text{var} \Big[p^{-1}\sum_{j=1}^p2\vep_{tj}z_{t,-j}^T\eta^0_{(j)}\Big]=O\big(\xi_{2,2}^2/p\big)=O(\psi^2_{\iny})<\iny,\quad{\rm and}\nn\\
&&\text{var}\Big[p^{-1}\sum_{j=1}^p\eta^{0T}_{(j)}z_{t,-j}z_{t,-j}^T\eta^0_{(j)}\Big]=O\big(\xi_{2,2}^4\big/p^2\big)=O(\psi^4_{\iny})<\iny,\quad {\rm thus},\nn\\
&&\text{var}\Big[p^{-1}\sum_{j=1}^p\Big\{2\vep_{tj}z_{t,-j}^{T}\eta^0_{(j)}-\eta^{0T}_{(j)}z_{t,-j}z_{t,-j}^T\eta^0_{(j)}\Big\}\Big]=\nn\\
&&\hspace{2.75in}2\big(O(\psi^2_{\iny})+O(\psi^4_{\iny})\big)<\iny.
\eenr
Relation (\ref{eq:finite.var}) implies that the sequence of r.v.'s in Condition B$'$ have bounded variances, thereby implying the distribution of the limiting random variable is well defined ($\cL<\iny,$ a.s.), i.e., supported in $\R.$ Consequently, Condition B$'$ simply provides a notation $\cL$ to whatever distribution this may be, with an appropriate variance notation $\bar\si_1^2$ or $\bar\si_2^2,$ respectively. The reader may observe that thus far in our discussion no additional assumption has been made in addition to Condition B and Condition D and the change of regime to the non-vanishing jump size. A further notational comment here is that the result to follow does not assume $p$ to be necessarily diverging. In the case of fixed $p$, the weak convergence ($\Rightarrow$) of Condition B$'$ can be replaced with an equality in distribution ($=^d$). Alternatively, one may view $p$ as a constant sequence in $T$ to maintain notational precision. 

The two-sided random walk defined in (\ref{def:cCr}) can now be utilized to characterize the limiting distribution of the change point estimator in the current non-vanishing regime. For this stochastic process, we have that   $z_t\sim^{i.i.d}\cL(-\psi_{\iny}^2\si_1^2,\,\bar\si^{2}_{1})$ and $z_t^*\sim^{i.i.d}\cL(-\psi_{\iny}^2\si_2^2,\,\bar\si^{2}_{2}),$ and $z_t$ and $z_t^*$ are also independent of each other over all $t.$ The only additional assumption of Condition B$',$ of continuity of the distribution law $\cL$ is assumed for the regularity of the \textit{argmax} of this two sided random walk (see, Lemma \ref{lem:regularity.argmax}). 

\begin{thm} (Non-vanishing jump regime)\label{thm:wc.non.vanishing} Suppose Conditions A, B$'$, C, and D hold. Further, assume that $\psi\to \psi_{\iny},$ $0<\psi_{\iny}<\iny,$ and that $\big\{s\log^{3/2} (p\vee T)\big/\surd (Tl_T)\big\}= o(1).$ 
	Then, the estimator $\tilde\tau$ in (\ref{est:optimal}) has the following limiting distribution:
	\benr
	(\lfloor T\tilde\tau\rfloor-\lfloor T\tau^0\rfloor)\Rightarrow \argmax_{r\in\Z}   \cC_{\iny}(r).\nn
	\eenr
	where $\cC_{\iny}(r)$ is as defined in (\ref{def:cCr}).
\end{thm}

The process $\cC_{\iny}(r)$ is a two sided random walk with negative drift and continuously distributed increments. Further, the map $\argmax_{r\in \Z}\cC_{\iny}(r)$ is almost surely unique and possesses a distribution supported on $\Z$, as shown in the proof of Theorem \ref{thm:wc.non.vanishing}.

\begin{rem} (A comparison on the limiting distribution results obtained to those established for mean shift models)
	{\rm  The obvious distinction between the stochastic processes (\ref{def:Zr}) and (\ref{def:cCr}) is that the first is continuous and the other discrete. An additional subtle observation distinguishing these processes is the stochastic term that characterizes them. Specifically, the limiting process in the vanishing regime is characterized by the sequence  $\z_t=p^{-1}\sum_{j=1}^p\vep_{tj}z_{t,-j}^{T}\eta^0_{(j)},$ $t=1,...,T,$ whereas in the non-vanishing regime by the sequence $\z_t=p^{-1}\sum_{j=1}^p\Big\{\vep_{tj}z_{t,-j}^{T}\eta^0_{(j)}-\eta^{0T}_{(j)}z_{t,-j}z_{t,-j}^T\eta^0_{(j)}\Big\},$ $t=1,...,T.$ A somewhat unusual consequence of this distinction is that the increments of the limiting process change from symmetric to asymmetric in the vanishing and non-vanishing regimes, respectively. We further note that this observation distinguishes the above result from that for mean shift models, where instead the same sequence of r.v.'s characterizes limiting processes for both vanishing and non-vanishing regimes. Another consequence of the above discussion is that the presence of an additional quadratic form in the sequence of interest leads to an inflation in the variance of the limiting process in the non-vanishing regime. The reason as to why this happens can be intuitively observed from (\ref{eq:finite.var}), where the variance of the quadratic form is $O(\psi^4),$ whereas the variance of the remainder is  $O(\psi^2).$ Thus, in the vanishing regime the first part of the r.v. under consideration dominates the quadratic form, which is no longer true in the non-vanishing jump regime.} 	
\end{rem}

\begin{rem} (Numerical approximations of distribution law $\cL$ and using Theorem \ref{thm:wc.non.vanishing} in applications)  {\rm  To construct a confidence interval for $\lfloor T\tau^0\rfloor$ based on Theorem \ref{thm:wc.non.vanishing}, one needs to obtain quantiles of the given limiting distribution. Unlike the limiting distribution of Theorem \ref{thm:limitingdist}, the cdf of this distribution is not available analytically. This can be achieved by simulating realizations of the two sided random walk $\cC_{\iny}(r),$ to obtain Monte Carlo approximations of the required quantiles. Doing so in turn requires producing realizations from the incremental distributions $\cL\big(-\psi_{\iny}^2\si_1^2,\,\bar\si_1^2\big)$ and $\cL\big(-\psi_{\iny}^2\si_1^2,\,\bar\si_1^2\big)$ of Condition B$',$ which first needs to be identified. We first note that the means of these distributions can be computed as plug in estimates from the estimated jump size and the given form of Condition D(i). The variances $\bar\si_1^2$ and $\si_2^2$ can also be estimated as piecewise sample variances from the observed data by noting that one has available $T$ predicted realizations, $\h\z_t=p^{-1}\sum_{j=1}^p\Big\{2\vep_{tj}z_{t,-j}^{T}\h\eta_{(j)}-\h\eta_{(j)}z_{t,-j}z_{t,-j}^T\h\eta_{(j)}\Big\},$ $t=1,...,T.$ The details of this estimation process are provided in Appendix \ref{sec:add.numerical}. In view of this discussion, the only missing link that remains is the form of the distribution $\cL.$ Since no explicit assumptions on the form of the underlying data generating distribution have been made in the article, thus identifying the distribution $\cL$ is not analytically feasible. For the Gaussian case, the distribution $\cL$ becomes an average of inter-dependent Variance-Gamma distributed random variables, which to the best of our knowledge has no known analytical form. We overcome this hurdle of choosing the form of $\cL$ by performing an empirical fit to the predicted realizations $\h\z_t,$ $t=1,...,T,$ by means of the Kolmogorov-Smirnov goodness of fit test. Details of this process are described in Appendix \ref{sec:add.numerical} and Algorithm 3 therein.}
\end{rem}

We conclude this section with a natural question that arises due to the inherent characteristic of change point estimators which splits distributional behavior into distinct regimes based on the jump size as discussed above. Given the fact that the distinction of a vanishing versus a non-vanishing jump size is unverifiable in practice, it remains unclear as to which of the two confidence intervals constructed using Theorems \ref{thm:limitingdist} or \ref{thm:wc.non.vanishing} would be a better representation in a real data setting. An immediate, but naive observation is that since the space of validity of the vanishing and non-vanishing regimes is $\psi_{\iny}=0$ and $\psi_{\iny}\in(0,\iny),$ respectively, thus without any additional information one may place more emphasis on the latter regime based solely on the larger size of the space of validity. A principled approach to this question has been undertaken in Section 5 of \cite{bhattacharjee2018change} in a stochastic block model framework under dense alternatives. They propose obtaining an empirical distribution of the change point estimator via replicated estimates obtained on synthetic data simulated under estimated nuisance parameters. They establish convergence of this empirical distribution to the underlying limiting distributions irrespective of the jump size regime. A similar approach may be used here, even though it entails high computational costs, further compounded by the high dimensional nature of the problem. Consequently, we do not pursue this further.     

\section{Construction of a feasible $O_p(\psi^{-2})$ estimator of $\lfloor T\tau^0\rfloor$}\label{sec:nuisance}

The results of Section \ref{sec:mainresults} rely on estimates
of the nuisance parameters satisfying Condition C. A procedure to obtain such estimates is discussed next. 
We start by introducing some more notation. For any $\tau\in(0,1),$ such that $\lfloor T\tau\rfloor\ge 1,$ consider $\ell_1$ regularized (lasso) estimates of the regression of each column of the observed variable $z$ on the remaining columns, for each of the two binary partitions induced by $\tau.$ Specifically, for each $j=1,...,p,$ define,
\benr\label{est:lasso}
\h\mu_{(j)}(\tau) = \argmin_{\substack{\mu_{(j)}\in \R^{p-1}}} \Big\{\frac{1}{\lfloor T\tau\rfloor}\sum_{t=1}^{\lfloor T\tau\rfloor} \big(z_{tj}- z_{t, -j}^T\mu_{(j)}\big)^2 + \la_j\|\mu_{(j)}\|_1 \Big\},\\
\h\g_{(j)}(\tau) = \argmin_{\substack{\g_{(j)}\in \R^{p-1}}} \Big\{\frac{1}{(T-\lfloor T\tau\rfloor)}\sum_{t=\lfloor T\tau\rfloor+1}^{T} \big(z_{tj}- z_{t, -j}^T\g_{(j)}\big)^2 + \la_j\|\g_{(j)}\|_1 \Big\},\hspace{-0.75cm}\nn
\eenr
where $\quad{\la_j>0}.$ To develop a feasible estimator for $\tau^0,$ recall the following from Section \ref{sec:mainresults}: (a) The missing links required to implement $\tilde\tau(\h\mu,\h\g)$ are the edge parameter estimates $\h\mu_{(j)}$ and $\h\g_{(j)},$ $j=1,....,p.$ (b) These edge estimates require sufficient Condition C to be satisfied in order to retain the results of Section \ref{sec:mainresults}. We shall fulfill these nuisance estimate requirements using the estimators in (\ref{est:lasso}) implemented in a twice iterated manner. The iterations are between the change point parameter $\tau$ and the edge parameters $\mu_{(j)}$ and $\g_{(j)}.$ This iterative approach is conceptually similar to that in \cite{atchade2017scalable}, with the added refinement of limiting the procedure to two iterations and further illustrating the redundancy of any further iterations. 

The twice iterative approach of the estimator to be considered is as follows. Rough edge estimates $\check\mu_{(j)}=\h\mu_{(j)}(\check\tau),$ and $\check\g_{(j)}=\h\g_{(j)}(\check\tau),$ $j=1,...,p,$ computed using a nearly arbitrary $\check\tau\in(0,1)$ (see, initializing condition of Algorithm 1 below) possess sufficient information, so that a single step update $\lfloor T\h\tau\rfloor=\lfloor T\tilde\tau(\check\mu,\check\g)\rfloor,$ moves into a near optimal neighborhood $O_p\big(\psi^{-2}\log(p\vee T)\big)$. With the availability of such a near optimal estimate $\lfloor T\h\tau\rfloor,$ we show that another update $\h\mu_{(j)}=\h\mu_{(j)}(\h\tau),$ and $\h\g_{(j)}=\h\g_{(j)}(\h\tau)$ satisfies all theoretical requirements of Condition C. This allows us to perform another update wherein Condition C is now applicable, thus ensuring that the results of Section \ref{sec:mainresults} hold for this second update. Note that the second update of the change point $\lfloor T\tilde\tau\rfloor=\lfloor T\tilde\tau(\h\mu,\h\g)\rfloor,$ moves $\lfloor T\h\tau\rfloor$ from the near optimal neighborhood $O_p\big(\psi^{-2}\log(p\vee T)\big)$ into an $O_p\big(\psi^{-2}\big)$ neighborhood of $\lfloor T\tau^0\rfloor.$ This is a direct consequence of Theorem \ref{thm:optimalapprox}. Additionally, Theorem \ref{thm:limitingdist} also provides the limiting distribution of this second update $\tilde\tau,$ thereby allowing inference on $\tau^0.$ The procedure is stated as Algorithm 1 below and is described visually in Figure \ref{fig:schematic}.

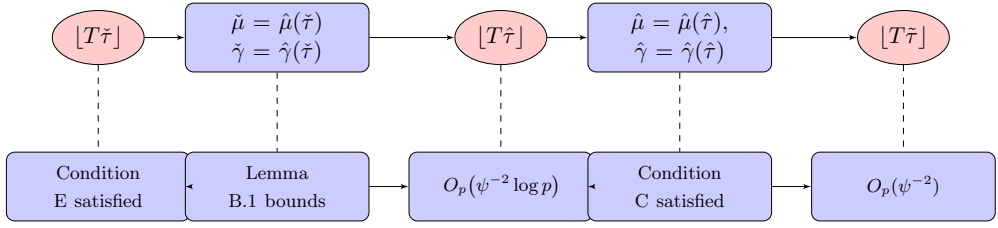
\begin{figure}[]
	\centering
	\resizebox{0.9\textwidth}{!}{
		\begin{tikzpicture}[node distance = 3.75cm,auto]
		\node [cloud] (ctau) {$\lfloor T\check\tau\rfloor$};
		\node [block,right of=ctau] (ctheta) {$\check\mu=\h\mu(\check\tau)$\\ $\check\g=\h\g(\check\tau)$};
		\node [cloud, right of=ctheta] (htau) {$\lfloor T\h\tau\rfloor$};
		\node [block,right of=htau] (htheta) {$\h\mu=\h\mu(\h\tau),$ $\h\g=\h\g(\h\tau)$};
		\node [cloud,right of=htheta] (btau) {$\lfloor T\tilde\tau\rfloor$};
		\node [block, below of=ctau, node distance=2.5cm] (arb) {\footnotesize{Condition E satisfied}};
		\node [block, below of=ctheta, node distance=2.5cm] (c1) {\footnotesize{Lemma \ref{lem:check.mu.g}} bounds};
		\node [block, below of=htau, node distance=2.5cm] (nopt) {\footnotesize{$O_p\big(\psi^{-2}\log p\big)$}};
		\node [block, below of=btau, node distance=2.5cm] (opt) {\footnotesize{$O_p(\psi^{-2})$}};
		\node [block, below of=htheta, node distance=2.5cm] (c2) {\footnotesize{Condition C satisfied}};
		\path [line] (ctau) -- (ctheta);
		\path [line] (ctheta) -- (htau);
		\path [line] (htau) -- (htheta);
		\path [line] (htheta) -- (btau);
		\path [linena,dashed] (ctau) -- (arb);
		\path [linena,dashed] (htau) -- (nopt);
		\path [linena,dashed] (btau) -- (opt);
		\path [linena,dashed] (ctheta) -- (c1);
		\path [linena,dashed] (htheta) -- (c2);
		\path[line](arb)  --  (c1);
		\path[line](c1)  -- (nopt);
		\path[line](nopt)  --  (c2);
		\path[line](c2)  --  (opt);
		\end{tikzpicture}}
	\caption{\footnotesize{A schematic of the proposed Algorithm 1 and its underlying working mechanism.}}
	\label{fig:schematic}
\end{figure}

\begin{figure}[]
	\centering
	\noi\rule{\textwidth}{0.5pt}
	
	\vspace{-2mm}
	\flushleft {\bf Algorithm 1:} $O_p(\psi^{-2})$ estimation of $\lfloor T\tau^0\rfloor:$
	
	\vspace{-2mm}
	\noi\rule{\textwidth}{0.5pt}
	
	\vspace{-2mm}
	\flushleft{\bf (Initialize):} Choose any $\check\tau\in (0,1)$ satisfying Condition E.
	
	\vspace{-1.5mm}
	\flushleft{\bf Step 1:} Obtain $\check\mu_{(j)}=\h\mu_{(j)}(\check\tau),$ and $\check\gamma_{(j)}=\h\gamma(\check\tau),$ $j=1,...,p,$ and update change point as,
	
	\vspace{-5mm}
	\benr
	\h\tau=\argmin_{\tau\in (0,1)}Q(z,\tau,\check\mu,\check\g)\nn
	\eenr
	
	\vspace{-5mm}
	\flushleft{\bf Step 2:} Obtain $\h\mu_{(j)}=\h\mu_{(j)}(\h\tau),$ and $\h\gamma_{(j)}=\h\gamma(\h\tau),$ $j=1,...,p,$ and perform another update,
	
	\vspace{-5mm}
	\benr
	\tilde\tau=\argmin_{\tau\in (0,1)}Q(z,\tau,\h\mu,\h\g)\nn
	\eenr
	
	\vspace{-4mm}
	\flushleft{\bf (Output):} $\tilde\tau$
	
	\vspace{-2.5mm}
	\noi\rule{\textwidth}{0.5pt}
\end{figure}


The following condition is imposed on the initializer $\check\tau$ of Algorithm 1.

\vspace{1.5mm}
{\it {{\noi{\bf Condition E (initializer):}} Assume initializer $\check\tau$ of Algorithm 1 satisfies, 
		\benr
		(i)\,\,\lfloor T\check\tau\rfloor\wedge(T-\lfloor T\check\tau\rfloor)\ge c_uTl_T,\qquad(ii)\,\, |\lfloor T\check\tau\rfloor-\lfloor T\tau^0\rfloor|\le \frac{c_{u}\ka l_T}{s(\si^2\vee\phi)} T^{(1-k)},\nn
		\eenr		
		for any constant $k>0.$\footnote{Without loss of generality we assume $k<b,$ where $b$ is as defined in Condition A$'$.}}}

The first requirement of Condition E is clearly innocuous and simply requires a separation of the chosen initializer from the boundaries of the parametric space of the change point which is satisfied with any $\check\tau\in [c_{u1},c_{u2}]\subset(0,1).$ Regarding the second requirement, for simplicity consider the case when $l_T\ge c_u<1,$ i.e., the true change point $\tau^0$ lies in some bounded subset of $(0,1),$ and the sparsity parameter is bounded above by a constant. Then, the requirement reduces to  $|\lfloor T\check\tau\rfloor-\lfloor T\tau^0\rfloor|=o(T^{1-k}),$ where the constant $k$ is any arbitrarily small but fixed value; in other words, the initializer may be in any arbitrary polynomial neighborhood $o(T^{(1-k)})$ of $\lfloor T\tau^0\rfloor.$ 

We establish that Step 1 of Algorithm 1 moves any starting value in this $o(T^{1-k})$ neighborhood into a near optimal neighborhood. Subsequently, the next iteration of Step 2 moves it to an $O_p(\psi^{-2})$ neighborhood of $\lfloor T\tau^0\rfloor,$ i.e., $o(T^{1-k})$-nbd.$\longrightarrow^{\rm Step 1}$ near optimal-nbd., $O_p(\psi^{-2}\log p)$ $\longrightarrow^{\rm Step 2}$ optimal-nbd., $O_p(\psi^{-2}).$ Note the sequential improvement in the rate of convergence from initialization to Step 2. Moreover, the improvement to optimality occurs in exactly two iterations. Another important consequence of these results is that it shows the redundancy of any further iterations, in the sense that since an optimal rate has been obtained at Step 2, performing further iterations will not yield any statistical improvement in the estimation of $\tau^0.$ This perspective showcases the mildness of Condition E.

From a practical perspective, a valid initializer $\check\tau$ in an $o(T^{1-k})$ neighborhood of $\lfloor T\tau^0\rfloor$ can be obtained by means of a preliminary coarse grid search
as follows: consider $T^{k}$ equally separated values in $\{1,...,T\}$ forming a coarse grid of possible initializers. Then, select the best fitting value $\check\tau$ for Algorithm 1, which by the pigeonhole principle it must be in an $o(T^{1-k})$ neighborhood of $\tau^0,$ and hence a valid initializer. A similar preliminary coarse grid search has also been heuristically utilized in \cite{roy2017change} in a different high dimensional model setting and also in \cite{kaul2019efficient,kaul2020inference} together with empirical evidence in its support. All simulation experiments in Section \ref{sec:numerical}, as well as the application Section \ref{sec:application} consider a preliminary search grid of $\check\tau\in\{0.25,0.5,0.75\}$ to select the initializer for Algorithm 1. Alternatively, one may also resort to Algorithm 2 below for the implementation of inference results without the requirement of Condition E.

\begin{rem} {\rm The restriction (ii) of Condition E can be further relaxed. One may eliminate the parameter $s$ from the bound $O(s^{-1}T^{1-k})$ and instead assume the initializer $\lfloor T\check\tau\rfloor$ to be in a $O(T^{1-k})$ neighborhood of $\lfloor T\tau^0\rfloor.$ The only consequence of this relaxation, assuming all other assumptions to hold, is
		that rate of convergence of Step 1 of Algorithm 1 will slow down to $O_p(\psi^{-2}s\log p)$ instead of $O_p(\psi^{-2}\log p).$ There will be no consequence in the rate of convergence of Step 2 of Algorithm 1, thus all inferential properties of $\tilde\tau$ of Algorithm 1 are retained.}  
\end{rem}

Next, we provide a modification of Condition A that is sufficient to obtain near optimality of $\h\tau$ in Step 1 of Algorithm 1 and is weaker than the original.

\vspace{1.5mm}
{\it {{\noi{\bf Condition A$'$ (assumption on model parameters):}} Let $\psi$ be as in (\ref{def:jumpsize}), $s$ and $l_T$ as in Condition A and parameters $\si^2,\nu$ and $\ka$ as in Condition B.\\~
		(i) Assume that for an appropriately chosen small constant $c_{u1}>0,$ the following holds
		\benr
		c_u\surd(1+\nu^2)\frac{\si^2s}{\psi\ka}\Big\{\frac{\log (p\vee T)}{T^{(1-2b)}l_T} \Big\}^{\frac{1}{2}}\le c_{u1},\nn
		\eenr
		for some $0<b<(1/2).$\\~
		(ii) Assume that $c_u\ka T^{(1-b)}l_T \ge (\si^2\vee\phi)s\log (p\vee T)$ \\~
		(iii) Assume $\tau^0$ is separated from the parametric boundary, i.e., $(\lfloor T\tau^0\rfloor)\wedge (T-\lfloor T\tau^0\rfloor)  \ge Tl_T.$}}

The following Theorem establishes that $\lfloor T\h\tau\rfloor$ in Step 1 of Algorithm 1 lies in an $O\big(\psi^{-2}\log (p\vee T)\big)$ neighborhood of $\tau^0.$ Inferential properties of $\tilde\tau$ of Step 2 rely critically on this result.

\begin{thm}\label{thm:alg1.nearoptimal} Suppose Conditions A$'$, B and E hold. Let $\h\tau$ be the change point estimate in Step 1 of Algorithm 1. Then, for sufficiently large $T$, we obtain
	\benr\label{eq:32}
	\qquad(1\wedge \psi^{2})(1+\nu^2)^{-1}(\si^{2}\vee\phi)^{-2}\ka^2\big|\lfloor T\h\tau\rfloor-\lfloor T\tau^0\rfloor\big|\le c_u\log(p\vee T)
	\eenr
	with probability $1-o(1).$ In other words, $(1\wedge \psi^2)\big(\lfloor T\h\tau\rfloor-\lfloor T\tau^0\rfloor\big)=O\big(\log(p\vee T)\big),$ with probability converging to $1.$
\end{thm}
%
Theorem \ref{thm:alg1.nearoptimal} shows that $\lfloor T\h\tau\rfloor$ in Step 1 of Algorithm 1 will satisfy a bound $O_p\big(\psi^{-2}\log(p\vee T)\big),$ despite the algorithm initializing with any $\lfloor T\check\tau\rfloor$ in a $O(s^{-1}T^{(1-k)})$ neighborhood of $\tau^0.$ This result now allows us to study the behavior of estimates of the edge parameters and the change point parameter obtained from the second iteration in Step 2 of Algorithm 1. We note here that the properties of these second iteration estimates rely solely on the bound (\ref{eq:32}) of $\lfloor T\h\tau\rfloor,$ and the availability of this bound renders no further use of the initial edge estimates $\check\mu_{(j)}$ and $\check\g_{(j)},$ $j=1,...,p.$ This feature allows Algorithm 1 to be modular in its construction, in the sense that for Step 2 to yield an estimate $\lfloor T\tilde\tau\rfloor$ that is $O_p(\psi^{-2})$, it does not require the estimator in Step 1 to be specifically the one selected in Algorithm 1. Alternatively, Step 1 of Algorithm 1 can readily be replaced with any other near optimal estimator available in the literature, i.e.,  satisfying a bound $O\big(\psi^{-2}\log (p\vee T)\big)$ with probability $1-o(1).$ This is described below as Algorithm 2.

\vspace{-2mm}
\begin{figure}[H]
	\noi\rule{\textwidth}{0.5pt}
	
	\vspace{-2mm}
	\flushleft {\bf Algorithm 2:} $O_p(\psi^{-2})$ estimation of $\lfloor T\tau^0\rfloor:$
	
	\vspace{-2mm}
	\noi\rule{\textwidth}{0.5pt}
	
	\vspace{-1mm}
	\flushleft{\bf Step 1:} Implement any $\h\tau$ from the literature that satisfies (\ref{eq:32}) with probability $1-o(1).$
	
	\vspace{-1mm}
	\flushleft{\bf Step 2:}  Obtain $\h\mu_{(j)}=\h\mu_{(j)}(\h\tau),$ and $\h\gamma_{(j)}=\h\gamma(\h\tau),$ $j=1,...,p,$ and perform update,
	
	\vspace{-4mm}
	\benr
	\tilde\tau=\argmin_{\tau\in (0,1)}Q(z,\tau,\h\mu,\h\g)\nn
	\eenr
	
	\vspace{-4mm}
	\flushleft{\bf (Output):} $\tilde\tau$
	
	\vspace{-2mm}
	\noi\rule{\textwidth}{0.5pt}
\end{figure}

An estimator from the literature that can be used in Step 1 of Algorithm 2 comes from \cite{atchade2017scalable}, which obeys the bound of Theorem \ref{thm:alg1.nearoptimal}. However, this estimator is likelihood based and hence limits the algorithm to the Gaussian setting. Further, it requires stronger sufficient conditions on the minimum jump size and separation sequence $l_T$ for analytical validity. To the best of our knowledge, there is no available estimator in the current literature that would serve as a replacement for Step 1 of Algorithm 1 under the  assumptions of Condition A$'$ (or Condition A) and Condition B.

The following results describe the statistical behavior of $\h\mu_{(j)}$ and $\h\g_{(j)},$ $j=1,...,p$ and $\tilde\tau$ obtained from Step 2 of Algorithm 1 or Algorithm 2. These results show that edge parameter updates $\h\mu_{(j)}$ and $\h\g_{(j)},$ $j=1,...,p$ obtained using the near optimal $\h\tau,$ are of a tighter precision than those in Step 1. In particular, these satisfy all requirements of Condition C.
Thus, the results of Section \ref{sec:mainresults} hold and a a higher precision estimate $\tilde\tau$ is obtained compared to that from Step 1 of Algorithms 1 or 2.

\begin{cor}\label{cor:optimalmeans.step2} Suppose Conditions A$'$, B and E hold. Let $\h\mu_{(j)},$ and $\h\g_{(j)},$ $j=1,...,p,$ be the edge estimates from Step 2 of Algorithms 1 or 2. Then, the following two properties hold with probability at least $1-o(1).$\\~
	(i) $\h\mu_{(j)}-\mu^0_{(j)}\in\cA_{1j},$ and $\h\g_{(j)}-\g^0_{(j)}\in\cA_{2j},$ $j=1,...,p$, where $\cA_{ij}$ are sets defined in Condition C.\\~
	(ii) $\max_{1\le j\le p}\Big(\|\h\mu_{(j)}-\mu^0_{(j)}\|_2\vee\|\h\g_{(j)}-\g^0_{(j)}\|_2\Big)\le c_u\surd(1+\nu^2)\frac{\si^2}{\ka}\Big\{\frac{s \log (p\vee T)}{Tl_T}\Big\}^{\frac{1}{2}}.$ \\~
	Consequently, these second iteration edge estimates satisfy all requirements of Condition C.
\end{cor}

Corollary \ref{cor:optimalmeans.step2} provides the feasibility of Condition C, while
the following corollary is now a direct consequence of Theorems \ref{thm:optimalapprox} and \ref{thm:limitingdist}.

\begin{cor}\label{cor:final}  Suppose Conditions A, B and E hold, additionally assume that model dimensions are restricted to satisfy $c_u\ka T^{(1-b)}l_T \ge (\si^2\vee\phi)s\log (p\vee T).$ Then, $\tilde\tau$ obtained from Algorithms 1 or 2 satisfies the error bounds of Theorem \ref{thm:optimalapprox}. Additionally, assuming Conditions D, B$'$ and (\ref{eq:rateextra}) hold, then $\tilde\tau$ converges to the limiting distributions of Theorem \ref{thm:limitingdist} and Theorem \ref{thm:wc.non.vanishing} in the vanishing and non-vanishing jump size regimes, respectively.
\end{cor}


\section{Simulation Studies}\label{sec:numerical}

We investigate the performance of Algorithm 1 and the inference results obtained in Theorems \ref{thm:limitingdist} and Theorem \ref{thm:wc.non.vanishing}.

Next, we describe the design of the simulation studies. In all settings considered, the unobserved variables $w_t,x_t$ of model (\ref{model:dggm}) are generated as independent, $p$-dimensional, zero mean Gaussian r.v.'s with distinct covariance structures. Specifically, we set $w_t\sim \cN(0,\Si),$ $t=1,...,\lfloor T\tau^0 \rfloor$ and $x_t\sim \cN(0,\Delta),$ $t=\lfloor T\tau^0\rfloor+1,...,T.$ The observation period $T$ is set to $\{300,400,500\}$, the dimension $p$ to $\{25,50,150,250\}$ and the relative location of the change point $\tau^0$ to $\{0.2,0.4,0.6,0.8\}.$ All computations are carried out in \texttt{R} using the \texttt{glmnet} package. In all simulation settings, the initializer for Algorithm 1 is chosen via a preliminary search grid of $\check\tau\in\{0.25,0.5,0.75\}$ as described in the discussion ensuing Condition E.

\noindent\textit{Structure of the covariance matrices:} To construct the pre-change point covariance $\Si,$ we consider a Toeplitz type matrix $\G$ with the $(l,m)^{th}$ component set as $\G_{(l,m)}=\rho^{|l-m|^a},$ $l,m=1,...,p.$ We set $\rho=0.4$ and $a=1/\log s,$ where $s$ specified below.\footnote{We choose the $\log s$ root of $|l-m|$ so as to somewhat preserve the magnitude of correlations and in turn condition dependencies}. Then, set $\Si=\cdotp A\cdotp \G,$ where $\cdotp$ denotes a componentwise product.  The matrix $A$ is constructed as a symmetric block diagonal matrix with alternating signs $\{-1,1\}$ within each block of size $s\times s.$ This allows both positive and negative correlations in $\Si$ and also induces a sparsity structure with each row and column having $s$ non-zero components. We set $s=0.15*p,$ i.e., sparsity in the pre-change covariance is set at $15\%.$ The post-change point covariance $\D$ is a banded matrix with the sparsity (length of bands) set at $20\%$ of the dimension size, i.e., $s=0.2*p.$ The non-zero correlations for each row and column of $\D$ are chosen as a sequence of $s$ equally spaced values between $\{\rho_2=0.5,...,0\}_{s\times 1}.$ Examples of the adjacency matrices corresponding to $\Si$ and $\D$ obtained from this construction are depicted in Figure \ref{fig:covariance.picture}. 

\begin{figure}[]
	\centering
	\resizebox{\textwidth}{!}{
		\begin{minipage}[b]{0.45\textwidth}
			\includegraphics[width=\textwidth]{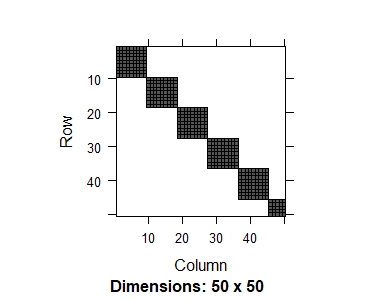}
		\end{minipage}
		\begin{minipage}[b]{0.45\textwidth}
			\includegraphics[width=\textwidth]{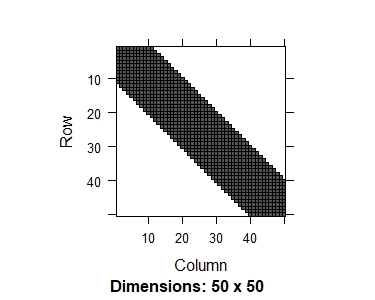}
	\end{minipage}}
	\caption{Adjacency matrices $|{\rm sign}(\Si)|$ and $|{\rm sign}(\D)|,$ with $p=50,$ $s_1=0.2\cdotp p,$ and $s_2=0.15\cdotp p.$ Here $|\cdotp|$ is componentwise absolute value. These matrices represent the underlying networks. Each dark pixel is an edge between the corresponding nodes, i.e., a non-zero entry of the precision matrix.}
	\label{fig:covariance.picture}
\end{figure}

\noindent\textit{Selection of tuning parameters:}  The tuning parameters $\la_{j},$ $j=1,...,p$ used to obtain $\ell_1$ regularized mean estimates of nuisance parameters in Steps 1 and 2 of Algorithm 1 are selected based on a BIC type criterion. Specifically, we set $\la_j=\la,$ $j=1,...,p,$ and evaluate $\h\mu_{(j)}(\la),$ and $\h\g_{(j)}(\la)$ over an equally spaced grid of seventy five values in the interval $(0,1).$ Upon letting $\h S=\cup_{j=1}^p\big[\{k;\,\,\h\mu_{(j)}k\ne 0\}\cup \{k;\,\,\h\g_{(j)}k\ne 0\}\big]$ we evaluate the criteria,
\benr\label{eq:bic}
BIC(\la,\tau)&=& \sum_{t=1}^{\lfloor T\tau\rfloor}\sum_{j=1}^p\big(z_{tj}-z_{t,-j}^T\h\mu_{(j)}(\la)\big)^2+\\
&&\sum_{t=\lfloor T\tau\rfloor+1}^{T}\sum_{j=1}^p(z_{tj}-z_{t,-j}^T\h\g_{(j)}(\la)\big)^2+|\h S|\log T.\nn
\eenr
For Step 1 of Algorithm 1, we set $\la$ as the minimizer of $BIC(\la,\check\tau),$ and for Step 2 of Algorithm 1 we select $\la$ as the minimizer of $BIC(\la,\h\tau).$ 

We construct confidence intervals using the limiting distributions from Theorem's \ref{thm:limitingdist} and  \ref{thm:wc.non.vanishing}. The significance level is set to $\al\in\{0.1,0.05,0.01\}$ for all settings. Confidence intervals are constructed in the integer time scale as $\big[\big(\lfloor T\tilde\tau\rfloor-ME\big),\, \big(\lfloor T\tilde\tau\rfloor+ME\big)\big],$ wherein $\tilde\tau$ is the output of Algorithm 1 and the margin of error ($ME$) is computed based on the corresponding jump size regime as follows. Under the vanishing regime, we have $ME=q_{\alpha}^v\si_1^{*2}\si_1^{-4}\psi_{\iny}^{-2},$ where $q_{\al}^v$ represents the $\big(1-\alpha/2\big)^{th}$ symmetric quantile of the argmax of the two sided negative drift Brownian motion in Theorem \ref{thm:limitingdist}.  This critical value is evaluated by using its distribution function provided in \cite{bai1997estimation}. Under the non-vanishing regime we have $ME=q_{\al}^{nv},$ where the quantile is computed based on the results of Theorem \ref{thm:wc.non.vanishing}. The critical value $q_{\al}^{nv}$ of the argmax of the two sided negative drift random walk is computed based on its Monte Carlo approximation, by simulating $3000$ realizations of this distribution. These calculations also require estimation of the drift and variance parameters $\si_1^2,\si_2^2$ and $\si_1^{*2},\si_2^{*2},\bar\si_1^2,\bar\si_2^2,$ respectively, as well as identification of the distribution law $\cL$ of Condition $B'.$  Succinctly, the drift parameters are estimated as plug-in estimates based on the estimated jump size and the corresponding covariance matrices. The variance parameters are estimated as sample variances using predicted realizations of r.v.' of interest and finally $\cL$ is chosen based on a negative centered and scaled chi-squared distribution with its degrees of freedom chosen via the Kolmogorov-Smirnov goodness of fit test. The details of these computations are provided in Appendix \ref{sec:add.numerical} of the Supplement.

The following performance metrics are computed for the change point estimates on the integer time scale: bias ($\big|E\big(\lfloor T\h\tau\rfloor-\lfloor T\tau^0\rfloor\big)\big|$), root mean squared error (rmse, $E^{1/2}\big(\lfloor T\h\tau\rfloor-\lfloor T\tau^0\rfloor)^2$), coverage (relative frequency of the number of times $\tau^0$ lies in the confidence interval) and the average margin of error (average over replicates of the margin of error of each confidence interval). All reported metrics are based on 500 replicates for each simulation setting. Selected results are provided in Tables \ref{tab:t01} and \ref{tab:t02}, while for the other settings in Tables \ref{tab:t03} and \ref{tab:t04} in Appendix \ref{sec:add.numerical} of the Supplement.

\begin{table}[]
	\centering
	\resizebox{1\textwidth}{!}{
		\begin{tabular}{ccccccccc}
			\toprule
			\multicolumn{3}{c}{\multirow{2}{*}{$\tau^0=0.2$}} & \multicolumn{6}{c}{Coverage (Av. margin of error)}                                        \\ \cmidrule{4-9} 
			\multicolumn{3}{c}{}                                                                                                                                     & \multicolumn{3}{c}{Non-vanishing}         & \multicolumn{3}{c}{Vanishing}                 \\ \midrule
			$T$                                            & $p$                                           & bias (rmse)                                             & $\al=0.1$ & $\al=0.05$    & $\al=0.01$    & $\al=0.1$     & $\al=0.05$    & $\al=0.01$    \\ \midrule
			300                                            & 25                                            & 0.016 (0.245)                                           & 0.94 (0)  & 0.95 (0.036) & 0.98 (0.772)  & 0.95 (0.148)  & 0.95 (0.214)  & 0.95 (0.389) \\
			300                                            & 50                                            & 0.032 (0.219)                                           & 0.97 (0)  & 0.97 (0)     & 0.97 (0.116)  & 0.97 (0.068)  & 0.97 (0.098)  & 0.97 (0.178)  \\
			300                                            & 150                                           & 2.294 (12.89)                                           & 0.93 (0)  & 0.93 (0)     & 0.93 (0.014)  & 0.93 (0.036)  & 0.93 (0.053)  & 0.93 (0.098)  \\
			300                                            & 250                                           & 12.97 (32.05)                                           & 0.72 (0)  & 0.72 (0)     & 0.72 (0)      & 0.72 (0.051)  & 0.72 (0.075)  & 0.72 (0.141) \\  \midrule
			400                                            & 25                                            & 0.008 (0.253)                                           & 0.95 (0)  & 0.95 (0.024) & 0.99 (0.824)  & 0.95 (0.157)  & 0.95 (0.228)  & 0.95 (0.417) \\
			400                                            & 50                                            & 0.016 (0.155)                                           & 0.97 (0)  & 0.97 (0)     & 0.98 (0.102)  & 0.98 (0.071)  & 0.98 (0.103)  & 0.98 (0.189) \\
			400                                            & 150                                           & 0.086 (1.661)                                           & 0.99 (0)  & 0.99 (0)     & 0.99 (0)      & 0.99 (0.022)  & 0.99 (0.032)  & 0.99 (0.059)  \\
			400                                            & 250                                           & 2.886 (18.42)                                           & 0.93 (0)  & 0.93 (0)     & 0.93 (0)      & 0.93 (0.026)  & 0.93 (0.039)  & 0.93 (0.073) \\  \midrule
			500                                            & 25                                            & 0.002 (0.272)                                           & 0.95 (0)  & 0.95 (0.020) & 0.99 (0.834)  & 0.95 (0.167) & 0.95 (0.242) & 0.95 (0.448) \\
			500                                            & 50                                            & 0.006 (0.118)                                           & 0.98 (0)  & 0.98 (0)     & 0.98 (0.098)  & 0.98 (0.073) & 0.98 (0.107) & 0.98 (0.198) \\
			500                                            & 150                                           & 0.006 (0.077)                                           & 0.99 (0)  & 0.99 (0)     & 0.99 (0)      & 0.99 (0.021) & 0.99 (0.03)  & 0.99 (0.055) \\
			500                                            & 250                                           & 0.042 (0.475)                                           & 0.98 (0)  & 0.98 (0)     & 0.98 (0)      & 0.98 (0.016) & 0.98 (0.023) & 0.98 (0.043) \\ \bottomrule
	\end{tabular}}
	\caption{\footnotesize{Simulation results for $\tau^0=0.20$ based on 500 replicates. Bias, rmse and av.margin of error rounded to three decimals, coverage rounded to two decimals.}}
	\label{tab:t01}
\end{table}

\begin{table}[]
	\centering
	\resizebox{1\textwidth}{!}{
		\begin{tabular}{cclllllll}
			\toprule
			\multicolumn{3}{c}{\multirow{2}{*}{$\tau^0=0.4$}} & \multicolumn{6}{c}{Coverage (Av. margin of error)}                                                                                                                                                \\ \cmidrule{4-9} 
			\multicolumn{3}{c}{}                              & \multicolumn{3}{c}{Non-vanishing}                                                               & \multicolumn{3}{c}{Vanishing}                                                                   \\ \midrule
			$T$   & $p$   & \multicolumn{1}{c}{bias (rmse)}   & \multicolumn{1}{c}{$\al=0.1$} & \multicolumn{1}{c}{$\al=0.05$} & \multicolumn{1}{c}{$\al=0.01$} & \multicolumn{1}{c}{$\al=0.1$} & \multicolumn{1}{c}{$\al=0.05$} & \multicolumn{1}{c}{$\al=0.01$} \\ \midrule
			300   & 25    & 0.012 (0.245)                     & 0.94 (0.002)                  & 0.95 (0.143)                  & 0.99 (0.884)                  & 0.94 (0.179)                 & 0.94 (0.262)                  & 0.94 (0.492)                   \\
			300   & 50    & 0.006 (0.134)                     & 0.98 (0)                      & 0.98 (0)                      & 0.99 (0.274)                  & 0.98 (0.084)                 & 0.98 (0.125)                  & 0.98 (0.238)                  \\
			300   & 150   & 0 (0.063)                         & 0.99 (0)                      & 0.99 (0)                      & 0.99 (0.002)                  & 0.99 (0.025)                 & 0.99 (0.037)                  & 0.99 (0.071)                  \\
			300   & 250   & 0.534 (4.56)                      & 0.98 (0)                      & 0.98 (0.008)                  & 0.98 (0.016)                  & 0.98 (0.016)                 & 0.98 (0.024)                  & 0.98 (0.046)                  \\    \midrule
			400   & 25    & 0.014 (0.279)                     & 0.93 (0)                      & 0.94 (0.088)                  & 0.99 (0.898)                  & 0.93 (0.187)                 & 0.93 (0.275)                  & 0.93 (0.519)                  \\
			400   & 50    & 0.006 (0.100)                     & 0.99 (0)                      & 0.99 (0)                      & 0.99 (0.324)                  & 0.99 (0.086)                 & 0.99 (0.128)                  & 0.99 (0.245)                   \\
			400   & 150   & 0 (0)                             & 1 (0)                         & 1 (0)                         & 1 (0.002)                     & 1 (0.024)                    & 1 (0.037)                     & 1 (0.072)                      \\
			400   & 250   & 0 (0)                             & 1 (0)                         & 1 (0)                         & 1 (0)                         & 1 (0.014)                    & 1 (0.021)                     & 1 (0.042)                      \\   \midrule
			500   & 25    & 0.026 (0.326)                     & 0.93 (0)                     & 0.94 (0.068)                   & 0.99 (0.95)                   & 0.94 (0.194)                 & 0.94 (0.287)                  & 0.94 (0.546)                  \\
			500   & 50    & 0.004 (0.110)                     & 0.98 (0)                     & 0.98 (0)                       & 0.99 (0.306)                  & 0.98 (0.088)                 & 0.98 (0.132)                  & 0.98 (0.252)                  \\
			500   & 150   & 0.002 (0.045)                     & 0.99 (0)                     & 0.99 (0)                       & 0.99 (0)                      & 0.99 (0.025)                 & 0.99 (0.037)                  & 0.99 (0.072)                  \\
			500   & 250   & 0 (0)                             & 1 (0)                        & 1 (0)                          & 1 (0)                         & 1 (0.014)                    & 1 (0.021)                     & 1 (0.041)                      \\ \bottomrule
	\end{tabular}}
	\caption{\footnotesize{Simulation results for $\tau^0=0.40$ based on 500 replicates. Bias, rmse and av.margin of error rounded to three decimals, coverage rounded to two decimals.}}
	\label{tab:t02}
\end{table}

The estimates exhibit very little bias for smaller values of $p$ for a given sample size $T$. Further, the results improve whenever the change point is located closer to the middle of the observation interval. The proposed inference methodology also provides reasonably good control on the nominal significance levels with an expected deterioration observed for larger values of $p$ and values of $\tau^0$ closer to the boundary of the parametric space (also see, results of Table \ref{tab:t03} and \ref{tab:t04}). The cases where coverage was found to be poor was at the largest considered value of $p=250$ and $\tau^0=0.8,$ but importantly the coverage improves to the nominal level as $T$ increases. 

These numerical results lead to the following observations: (a)  The margin of error may become less than one and due to the underlying discrete-ness of the inference problem it may become infeasible to distinguish certain significance levels. (b) The distribution becomes more concentrated around the change point parameter for larger dimension $p$, thus making it more difficult to distinguish between coverage levels. To relate this feature to the results of Section \ref{sec:mainresults}, recall that in the non-vanishing regime, the variance of the limiting process is of $O(\psi^2+\psi^4),$ where $\psi=\xi_{2,2}/\surd{p}$ (see (\ref{eq:finite.var}). Owing to the simulation design under consideration, when $p$ increases, $\psi$ is observed to be decreasing, thus leading to a diminishing variance as well. (c) An inherent high computational cost for the recovery of large size graphical models, limits our numerical experiments to $500$ replicates. 

\section{Age Evolving Associations of the Gut Microbiome}\label{sec:application}

Microbiome studies are becoming increasingly important, due to recent findings on interactions of human microbiota with several gastrointestinal, as well as potentially neurological health outcomes \citep{svoboda2020could,sharma2019gut}. Large scale microbiome data have become available in the last decade and are obtained by 16s rRNA sequencing technology. The resulting data correspond to operational taxonomic units (OTUs) which represent counts of observed microbial taxa identified by their genetic signature. In this setting it is often of interest to investigate relationships among the microbes to understand their effects on health outcomes. For our analyses we consider the global human gut microbiome data of \cite{yatsunenko2012human} available publicly at the repository MG-RAST (\url{http://metagenomics.anl.gov/}) under accession numbers {\it qiime:621}.

It has been discussed in the literature that the gut microbiome  of an individual undergoes a significant transformation from infancy/adolescence to adulthood. This transition age is often determined based on domain knowledge or other significant life events.  \cite{lozupone2013meta} suggest this transition age at around two years due to a switch over from breast milk (or formula milk) to solid food. This age has also been employed in \cite{kaul2017structural} for geographical classification of subjects based on their microbiota. However, \cite{lane2019household} suggest that such a transition could occur well into adolescence of an individual, due to various social interactions that children are exposed to, including those with siblings, pets and farm animals, or early exposure to antibiotics. Hence, it is of interest to estimate this transition point based on microbiome data. We employ the model in (\ref{model:dggm}), and estimate the transition age in the second order association structure of the taxa and also quantify its uncertainty through a confidence interval.

Although the Global gut data set contains measurements from individuals from the United States (US), Venezuela and Malawi, we limit our analysis to only $T=310$ individuals from the US. This is done to avoid location heterogeneity since our objective is to study evolution of microbial associations over age. Analysis is carried out at the second to finest, i.e., the {\it genus} level of bacterial taxonomy. We subset the analysed set of {\it genera} by retaining only those present in at least $35\%$ of the samples. This limits the number of  {\it genera} to $p=166$ for model (\ref{model:dggm}). A further pre-processing of the data set is carried out via a $\log$-relative abundance transformation of the raw OTU data, in order to switch over from a count to a continuous scale. The reference group chosen for this transformation is {\it Bifidobacterium}\footnote{Phylogeny: {\it Bacteria$\to$Actinobacteria$\to$Actinobacteria $\to$Bifidobacteriales$\to$Bifidobacteriaceae$\to$Bifidobacterium}} due to it being a highly observed taxa which is present in all analyzed samples. This transformation is motivated by the compositional structure of the data set, see, e.g., \cite{aitchison1982statistical} and is often adopted in the microbiome literature, see, e.g., \cite{mandal2015analysis}, \cite{kaul2017analysis,kaul2017structural}. We also note here that despite the above $log$-transformation the underlying data are clearly non-Gaussian, as illustrated in Figure \ref{fig:age.hist}.

The data specimens under consideration have an associated age variable distributed over ($0.08$ years, $57$ years). This distribution is also presented in Figure \ref{fig:age.hist}. To study the age evolution of the underlying graphical models, all specimens are first sorted according to the age variable. Model (\ref{model:dggm}) is then implemented with Algorithm 1 of Section \ref{sec:nuisance}, with a preliminary search grid over $\check\tau\in\{0.25,0.5,0.75\},$ which is the same as that used in the simulation studies of Section \ref{sec:numerical}. All other computations such as tuning parameters selection, drift and asymptotic variance computation are as described in Section \ref{sec:numerical} and Appendix \ref{sec:add.numerical} of the supplement.
The estimated change point and the corresponding confidence intervals are obtained in the integer scale associated with index numbers of observations. We choose the significance level at $\al=0.05,0.01,$ i.e., a coverage of $95\%$ and $99\%,$ respectively. These estimated values are then mapped back to the age variable to obtain the transition point in the age scale. The results of our analyses are discussed below and depicted in Figure \ref{fig:appl.results}.

\begin{figure}
	\centering
	\begin{minipage}[]{0.45\textwidth}
		\includegraphics[scale=0.5]{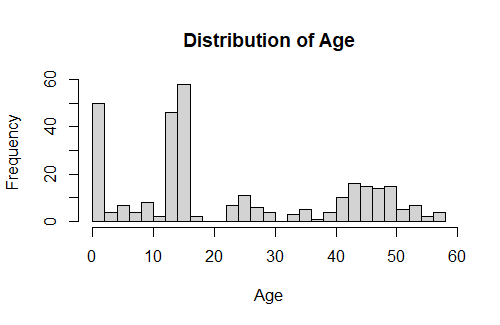}
	\end{minipage}
	\begin{minipage}[]{0.45\textwidth}
		\includegraphics[scale=0.5]{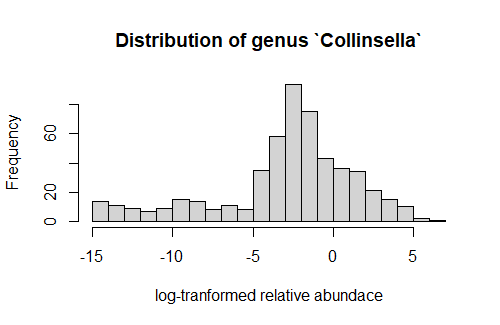}
	\end{minipage}
	\caption{\footnotesize{{\bf Left panel}: Distribution of observed {\it Age} variable. {\bf Right panel}: Distribution of log-relative abundance of genus {\it Collinsella} (one of $p=166$ {\it genera})}}
	\label{fig:age.hist}
\end{figure}

\begin{figure}[]	\centering
	\resizebox{0.75\textwidth}{!}{
		\begin{tikzpicture}[node distance = 5cm,auto]
		\node[left of=st,xshift=1.75cm] (g1) at (-0.75,-2.25)
		{\includegraphics[width=0.3\textwidth]{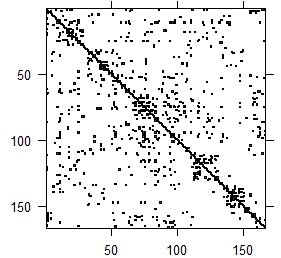}};
		\node[right of=st,xshift=-1.25cm] (g2) at (-0.75,-2.25)
		{\includegraphics[width=0.3\textwidth]{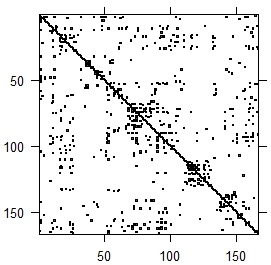}};
		\node[below of=g1, left of=st,xshift=1.5cm,yshift=0.25cm] (vg1) at (0,-2.2)
		{\includegraphics[width=0.25\textwidth]{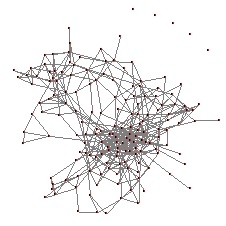}};
		\node[below of=g2,right of=st,xshift=-1.25cm,yshift=0.25cm] (vg2) at (0,-2.25)
		{\includegraphics[width=0.25\textwidth]{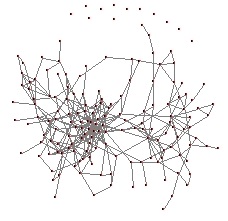}};
		\node [] (st){};
		\node [below of=st,align=center,yshift=-3.75cm] (end) {Est. age = 14yrs, \\
			99\%-CI(v):\,\,\,\,\big[13yrs, 14yrs\big]\\
			99\%-CI(nv):\,\big[14yrs, 14yrs\big]};
		\node [below of=vg1,xshift=1.25cm,yshift=2.5cm] (a1)  {};
		\node [below of=vg2,xshift=-1.25cm,yshift=2.5cm] (a2)  {Transition};
		\path [linena,dashed] (st) -- (end);
		\path [line] (a1) -- (a2);
		\end{tikzpicture}}
	\caption{\footnotesize{Summary of results: Estimated adjacency of precision matrices. Estimated change point located at $\lfloor T\tilde\tau\rfloor=88$ (Age = $14yrs$). Under vanishing jump regime: $99\%$ confidence intervals for $\lfloor T\tau^0\rfloor$ is $[86.56,\,\, 89.43]$ $(Age: [13,\,\,14].$ Under non-vanishing jump regime:$99\%$ confidence intervals for $\lfloor T\tau^0\rfloor$ is $[88,\,\,88]$ $(Age: [14,\,\,14].$}}
	\label{fig:appl.results}
\end{figure}
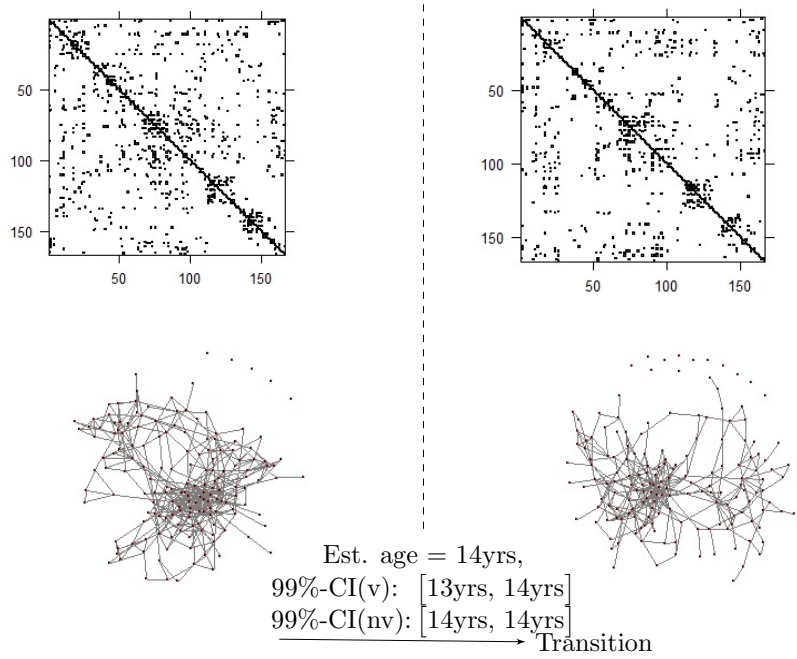

The estimated change point is $\lfloor T\tilde\tau\rfloor=\lfloor 0.283\cdotp T\rfloor=88.$ Upon mapping this index back to the age variable yields a transition age of $14yrs.$ Confidence intervals are constructed assuming both vanishing and non-vanishing jump size regimes and presented for both the index level and the age level in Table \ref{tab:cis}. At a coverage level of $99\%,$ and under the vanishing jump regime, the associated confidence interval at the index level is $[86.56,89.43],$ which yields an interval $[13yrs,14yrs]$ for the age of transition. All other intervals yield only the single index of $88,$ i.e., the margin of error for these intervals was less than one.

\begin{table}[]
	\centering
	\resizebox{0.75\textwidth}{!}{
		\begin{tabular}{lcccc}
			\toprule
			& \multicolumn{2}{c}{$\al=0.05$}     & \multicolumn{2}{c}{$\al=0.01$}     \\ \midrule 
			& Vanishing          & Non-Vanishing & Vanishing          & Non-vanishing \\ 
			\midrule
			Index level & {[}87.20, 88.79{]} & {[}88, 88{]}  & {[}86.56, 89.43{]} & {[}88, 88{]}  \\
			Age level   & {[}14, 14{]}       & {[}14, 14{]}  & {[}13, 14{]}       & {[}14, 14{]}  \\ \hline
	\end{tabular}}
	\vspace{1mm}
	\caption{Estimated confidence intervals under vanishing and non-vanishing jump size regimes at $95\%$ and $99\%$ coverage. Intervals presented at both index level and corresponding age level.}
	\label{tab:cis}
\end{table}

From Figure \ref{fig:appl.results}, the visual distinguishing feature between the two estimated graphical models is the increased sparsity of the post $14$ years network. The conditional dependencies appear to get more consolidated for older subjects and a more prominent hub structure is seen to develop. Note that in the network adjacency matrices illustrated in Figure \ref{fig:appl.results}, the {\it genera} under consideration are ordered according to their phylogenic classification and this appears to be the reason for the observed hub structure, i.e., the relative abundance of {\it genera} of the same {\it phylum} seem to be associated amongst themselves, with these associations getting more consolidated for older subjects. The relatively dense network observed in the pre 14 years group is also reasonable from a microbiome domain perspective, since the microbiome structure of younger subjects are known to be more volatile, moreover they are often more exposed to foreign microbes via a wide variety of social and environmental interactions, see, e.g., \cite{lane2019household}.

\bibliography{meanchange}
\bibliographystyle{\style}

\pagebreak

\setcounter{page}{1}

\begin{center}
	{\sc{\Large{Supplementary Materials: Inference on the Change Point in High Dimensional Dynamic Graphical Models}}}
\end{center}	
%
\startcontents[sections]
\printcontents[sections]{ }{1}{}

\appendix

\numberwithin{equation}{section}
\numberwithin{lem}{section}
\numberwithin{thm}{section}

\section{Proofs of results in Section 2}\label{sec:appA}

The following notations are required for readability of this section. In addition to $\xi_{2,2}$ defined in (\ref{def:jumpsize}), we also define $\xi_{2,1}=\sum_{j=1}^p\|\eta_{(j)}^0\|_2$ in the $\ell_{2,1}$ norm. Also, in all to follow we denote as $\h\eta_{(j)}=\h\mu_{(j)}-\h\g_{(j)},$ $j=1,...,p.$ We also recall the definition of r.v.'s $\vep_{tj}$ from (\ref{def:epsilons}),
\benr
\vep_{tj}=\begin{cases}z_{tj}-z_{t,-j}^T\mu^0_{(j)}, & t=1,...,\lfloor T\tau^0\rfloor\\
	z_{tj}-z_{t,-j}^T\g^0_{(j)}, & t=\lfloor T\tau^0\rfloor+1,...,T.\end{cases}\nn
\eenr
\begin{proof}[Proof of Lemma \ref{lem:mainlowerb}] For any fixed $\tau\ge \tau^0$ consider,
	\benr\label{eq:10}	
	&&\hspace{1cm}\cU(z,\tau,\h\mu,\h\g)=Q(z,\tau,\h\mu,\h\g)-Q(z,\tau^0,\h\mu,\h\g)\\
	&&\hspace{1cm}=\frac{1}{T}\sum_{t=1}^{\lfloor T\tau\rfloor}\sum_{j=1}^p(z_{tj}-z_{t,-j}^T\h\mu_{(j)})^2 + \frac{1}{T}\sum_{t=\lfloor T\tau\rfloor+1}^{T}\sum_{j=1}^p(z_{tj}-z_{t,-j}^T\h\g_{(j)})^2\nn\\
	&&\hspace{1cm}- \frac{1}{T}\sum_{t=1}^{\lfloor T\tau^0\rfloor}\sum_{j=1}^p(z_{tj}-z_{t,-j}^T\h\mu_{(j)})^2 - \frac{1}{T}\sum_{t=\lfloor T\tau^0\rfloor+1}^{T}\sum_{j=1}^p(z_{tj}-z_{t,-j}^T\h\g_{(j)})^2\nn\\
	&&\hspace{1cm}=\frac{1}{T}\sum_{t=\lfloor T\tau^0\rfloor+1}^{\lfloor T\tau\rfloor}\sum_{j=1}^p(z_{tj}-z_{t,-j}^T\h\mu_{(j)})^2 - \frac{1}{T}\sum_{t=\lfloor T\tau^0\rfloor+1}^{\lfloor T\tau\rfloor}\sum_{j=1}^p(z_{tj}-z_{t,-j}^T\h\g_{(j)})^2\nn\\
	&&\hspace{1cm}=\frac{1}{T}\sum_{t=\lfloor T\tau^0\rfloor+1}^{\lfloor T\tau\rfloor}\sum_{j=1}^p\big(z_{t,-j}^T\h\eta_{(j)}\big)^2-\frac{2}{T}\sum_{t=\lfloor T\tau^0\rfloor+1}^{\lfloor T\tau\rfloor}\sum_{j=1}^p\vep_{tj}z_{t,-j}^T\h\eta_{(j)}\nn\\
	&&\hspace{1cm}+\frac{2}{T}\sum_{t=\lfloor T\tau^0\rfloor+1}^{\lfloor T\tau\rfloor}\sum_{j=1}^p(\h\g_{(j)}-\g^0_{(j)})^Tz_{t,-j}z_{t,-j}^T\h\eta_{(j)}.\nn
	\eenr
	
	The expansion in (\ref{eq:10}) provides the following relation,
	\benr\label{eq:11}
	\inf_{\substack{\tau\in\cG(u_T,v_T);\\ \tau\ge\tau^0}}\cU(z,\tau,\h\mu,\h\g)\ge 	\inf_{\substack{\tau\in\cG(u_T,v_T);\\ \tau\ge\tau^0}} \frac{1}{T}\sum_{t=\lfloor T\tau^0\rfloor+1}^{\lfloor T\tau\rfloor}\sum_{j=1}^p\big(z_{t,-j}^T\h\eta_{(j)}\big)^2\nn\\
	-2\sup_{\substack{\tau\in\cG(u_T,v_T);\\ \tau\ge\tau^0}}\frac{1}{T}\Big|\sum_{t=\lfloor T\tau^0\rfloor+1}^{\lfloor T\tau\rfloor}\sum_{j=1}^p\vep_{tj}z_{t,-j}^T\h\eta_{(j)}\Big|\hspace{1.21cm}\nn\\
	-2\sup_{\substack{\tau\in\cG(u_T,v_T);\\ \tau\ge\tau^0}}\frac{1}{T}\Big|\sum_{t=\lfloor T\tau^0\rfloor+1}^{\lfloor T\tau\rfloor}\sum_{j=1}^p(\h\g_{(j)}-\g^0_{(j)})^Tz_{t,-j}z_{t,-j}^T\h\eta_{(j)}\Big|\hspace{-1.21cm}\nn\\
    =R1-R2-R3\hspace{1.5in}
	\eenr
	Bounds for the terms $R1,R2$ and $R3$ are provided in Lemma \ref{lem:term123} and Lemma \ref{lem:assumptionbounds}. In particular,
	\benr
	R1&\ge& \ka\xi_{2,2}^2\Big[v_T-\frac{c_{a1}\si^2}{\ka}\Big(\frac{u_T}{T}\Big)^{\frac{1}{2}}- c_u(\si^2\vee\phi)\frac{u_T}{\ka\xi_{2,2}}\Big\{s\log(p\vee T)\sum_{j=1}^p\|\h\eta_{(j)}-\eta^0_{(j)}\|_2^2\Big\}^{\frac{1}{2}}\Big]\nn\\
	&\ge& \ka\xi_{2,2}^2\Big[v_T-\frac{c_{a1}\si^2}{\ka}\Big(\frac{u_T}{T}\Big)^{\frac{1}{2}}- c_{u1}(\si^2\vee\phi)\frac{u_T}{\ka T^{b}}\Big],\nn
	\eenr
	with probability at least $1-a-o(1).$ The first inequality follows from Lemma \ref{lem:term123} and the final inequality follows by using the bounds of Lemma \ref{lem:assumptionbounds}. Next we obtain upper bounds for the terms $R2\big/\ka\xi_{2,2}^2$ and $R3\big/\ka\xi_{2,2}^2.$ For this purpose, first note that  $(\xi_{2,1}\big/\xi_{2,2})\le \surd{p},$ consequently $(\xi_{2,1}\big/\xi_{2,2}^2)\le 1\big/\psi.$ Now consider,
	\benr
	\frac{R2}{\ka\xi_{2,2}^2}&\le& c_{a1}\surd(1+\nu^2)\frac{\si^2\xi_{2,1}}{\ka\xi_{2,2}^2}\Big(\frac{u_T}{T}\Big)^{\frac{1}{2}}+ c_{u}\surd(1+\nu^2)\frac{\si^2}{\ka\xi_{2,2}^2}\Big(\frac{u_T}{T}\Big)^{\frac{1}{2}}\log (p\vee T)\sum_{j=1}^p\|\h\eta_{(j)}-\eta^0_{(j)}\|_1\nn\\
	&\le& c_{a1}\surd(1+\nu^2)\frac{\si^2}{\ka\psi}\Big(\frac{u_T}{T}\Big)^{\frac{1}{2}}+ \Big\{\surd(1+\nu^2)\frac{\si^2}{\ka\psi}\Big(\frac{u_T}{T}\Big)^{\frac{1}{2}}\Big\}\Big\{c_u\surd(1+\nu^2)\frac{\si^2}{\ka\psi}\frac{s\log^{3/2}(p\vee T)}{\surd(Tl_T)}\Big\}\nn\\
	&\le& c_uc_{a1}\surd(1+\nu^2)\frac{\si^2}{\ka\psi}\Big(\frac{u_T}{T}\Big)^{\frac{1}{2}}\nn
	\eenr
	with probability at least $1-a-o(1).$ As before, the first inequality follows from Lemma \ref{lem:term123} and the final inequality follows by using the bounds of Lemma \ref{lem:assumptionbounds}. Similarly we can also obtain,
	\benr
	\frac{R3}{\ka\xi_{2,2}^2}&\le&  c_u(\si^2\vee \phi)\frac{u_T}{\ka\xi_{2,2}}\Big\{s\log(p\vee T)\sum_{j=1}^p\|\h\g_{(j)}-\g^0_{(j)}\|_2^2\Big\}^{\frac{1}{2}}\nn\\
	&&\cdotp\Big[1+
	\frac{1}{\xi_{2,2}}\Big\{s\log(p\vee T)\sum_{j=1}^p\|\h\eta_{(j)}-\eta^0_{(j)}\|_2^2\Big\}^{\frac{1}{2}}\Big]\nn\\
	&\le& c_{u1}(\si^2\vee \phi)\frac{u_T}{\ka T^{b}}\nn
	\eenr
	with probability at least $1-a-o(1).$ Substituting these bounds in (\ref{eq:11}) and applying a union bound over these events yields the bound (\ref{eq:12}) uniformly over the set $\{\cG(u_T,v_T);\,\tau\ge\tau^0\}.$ The mirroring case of $\tau\le\tau^0$ follows with similar arguments.
\end{proof}

\bc$\rule{3.5in}{0.1mm}$\ec

The main idea of the proof of Theorem \ref{thm:optimalapprox} is to use a contradiction argument as follows. Using Lemma \ref{lem:mainlowerb} recursively, we show that any value of $\lfloor T\tau\rfloor$ lying outside an $O(c_{a3}^2)$ neighborhood of $\lfloor T\tau^0\rfloor$ satisfies, $\cU(z,\tau,\h\mu,\h\g)>0,$ with probability at least $1-3a-o(1).$ Upon noting that by definition of $\tilde\tau,$ we have, $\cU(z,\tilde\tau,\h\mu,\h\g)\le 0,$ yields the desired result.  The complete argument is below.

\vspace{1.5mm}
\begin{proof}[Proof of Theorem \ref{thm:optimalapprox}]
	To prove this result, we show that for any $0<a<1,$ the bound
	\benr\label{eq:13}
	\big|\lfloor T\tilde\tau\rfloor-\lfloor T\tau^0\rfloor\big|\le c_{a3}^2,
	\eenr
	holds with probability at least $1-3a-o(1).$ Note that Part (i) of this theorem is a direct consequence of the bound (\ref{eq:13}) in the case where $\psi\to 0.$ The proof for the bound (\ref{eq:13}) to follow relies on a recursive argument on Lemma \ref{lem:mainlowerb}, where the optimal rate of convergence $O_p(1)$ is obtained by a series of recursions with the rate of convergence being sharpened at each step.
	
	We begin by considering any $v_T>0,$ and applying Lemma \ref{lem:mainlowerb} on the set $\cG(1,v_T)$ to obtain,
	\benr
	\inf_{\tau\in\cG(1,v_T)} \cU(z,\tau,\h\mu,\h\g)\ge \ka\xi^2_{2,2}\Big[v_T-c_{a3}\max\Big\{\Big(\frac{1}{T}\Big)^{\frac{1}{2}},\frac{1}{T^{b}}\Big\}\Big]\nn
	\eenr	
	with probability at least $1-3a-o(1).$ Recall by assumption $b<(1/2),$ and choose any $v_T>v_T^*=c_{a3}/T^{b}.$ Then we have $\inf_{\tau\in\cG(1,v_T)}\cU(z,\tau,\h\mu,\h\g)>0,$ thus implying that $\tilde\tau\notin\cG(1,v_T),$ i.e., $\big|\lfloor T\tilde\tau\rfloor-\lfloor T\tau^0\rfloor\big|\le Tv_T^*,$ with probability at least $1-3a-o(1)$\footnote{Since by construction of $\tilde\tau$ we have, $\cU(\tilde\tau,\h\g,\h\g)\le 0.$}. Now reset $u_T=v_T^*$ and reapply Lemma \ref{lem:mainlowerb} for any $v_T>0$ to obtain,
	\benr
	\inf_{\tau\in\cG(u_T,v_T)} \cU(z,\tau,\h\mu,\h\g)\ge \ka\xi^2_{2,2}\Big[v_T-c_{a3}\max\Big\{\Big(\frac{c_{a3}}{T^{1+b}}\Big)^{\frac{1}{2}},\frac{c_{a3}}{T^{b+b}}\Big\}\Big]\nn
	\eenr
	Again choosing any,
	\benr
	v_T>v_T^*=\max\Big\{\frac{c_{a3}^{g_2}}{T^{u_2}},\,\,\frac{c_{a3}^{2}}{T^{v_2}}\Big\},\quad \eenr
	where,
	\benr
	g_2=1+\frac{1}{2},\,\, u_2=\frac{1}{2}+\frac{u_1}{2},\,\,{\rm and}\,\, v_2=b+v_1\ge 2b,\,\,{\rm with}\,\,u_1=v_1=b,\nn
	\eenr
	we obtain $\inf_{\cG(u_T,v_T)}\cU(z,\tau,\h\mu,\h\g)>0,$ with probability at least $1-3a-o(1).$ Consequently $\tilde\tau\notin\cG(u_T,v_T),$ i.e., $\big|\lfloor T\tilde\tau\rfloor-\lfloor T\tau^0\rfloor\big|\le Tv_T^*.$ Note the rate of convergence of $\tilde\tau$ has been sharpened at the second recursion in comparison to the first. Continuing these recursions by resetting $u_T$ to the bound of the previous recursion, and applying Lemma \ref{lem:mainlowerb}, we obtain for the $m^{th}$ recursion,
	\benr
	&&\big|\lfloor T\tilde\tau\rfloor-\lfloor T\tau^0\rfloor\big|\le T\max\Big\{\frac{c_{a3}^{g_m}}{T^{u_m}},\,\,\frac{c_{a3}^{m}}{T^{v_m}}\Big\}:=T\max\{R_{1m},R_{2m}\},\quad{\rm where,}\nn\\
	&&g_m=\sum_{k=0}^{m-1}\frac{1}{2^k},\,\, u_m=\frac{1}{2}+\frac{u_{m-1}}{2}=\frac{b}{m}+\sum_{k=1}^{m}\frac{1}{2^m},\,\,{\rm and}\nn\\
	&&v_m=b+v_{m-1}\ge mb,\,\,{\rm with}\,\,u_1=v_1=b.\nn
	\eenr
	Next, we observe that for $m$ large enough, $R_{2m}\le R_{1m}.$ This follows since $R_{2m}$ is faster than any polynomial rate of $1/T.$\footnote{Consider $c_1^m/T^{mb}\le (c_1\big/\log T)^m(\log T\big/T)^{mb}\le (1/T^{mb_1}),$ for any $0<b_1<b,$ for $T$ sufficiently large.} Consequently for $m$ large enough we have $\big|\lfloor T\tilde\tau\rfloor-\lfloor T\tau^0\rfloor\big|\le T R_{1m},$ with probability at least $1-3a-o(1).$ Finally, we continue these recursions an infinite number of times to obtain, $g_{\iny}=\sum_{k=0}^{\iny}1/2^{k},$ $u_\iny=\sum_{k=1}^{\iny}(1/2^k),$ thus yielding,
	\benr
	\big|\lfloor T\tilde\tau\rfloor-\lfloor T\tau^0\rfloor\big|\le T\frac{c_{a3}^2}{T}=c_{a3}^2\nn
	\eenr
	with probability at least $1-3a-o(1).$ This proves the bound (\ref{eq:13}). To finish the proof, note that despite the recursions in the argument, the probability bound after every step is maintained at $1-3a-o(1).$ This follows since the probability statement of Lemma \ref{lem:mainlowerb} arises from stochastic upper bounds of Lemma \ref{lem:optimalcross}, Lemma \ref{lem:nearoptimalcross}, Lemma \ref{lem:optimalsqterm} and Lemma \ref{lem:WUURE}, applied recursively, with a tighter bound at each recursion. This yields a sequence of events such that each event is a proper subset of the event at the previous recursion.
\end{proof}

\bc$\rule{3.5in}{0.1mm}$\ec

The limiting distributions of Theorem \ref{thm:limitingdist} and Theorem \ref{thm:wc.non.vanishing} are presented in more conventional {\it argmax} notation instead of the {\it argmin} notation of the problem setup in Section \ref{sec:intro}. This is purely a notational change and all results can equivalently be stated in the {\it argmin} language.

For a clear presentation of the proofs below we use the following additional notation.  Let $\cU(z,\tau,\theta_1,\theta_2)$ be as in (\ref{def:cU}) and consider,
\benr\label{def:cC}
\cC(\tau,\mu,\g)&=&-Tp^{-1}\cU(z,\tau,\mu,\g)
\eenr

The multiplication of $\cU$ with the product $Tp^{-1}$ is only meant for notational convenience later on. Then, we can re-express the change point estimator $\tilde\tau(\mu,\g)$ defined in (\ref{est:optimal}) as,
\benr
\tilde\tau(\mu,\g)=\argmax_{\tau\in(0,1)}\cC(\tau,\mu,\g)\nn
\eenr
The proofs of Theorem \ref{thm:limitingdist} and Theorem \ref{thm:wc.non.vanishing} below are applications of the Argmax Theorem (reproduced as Theorem \ref{thm:argmax}). The arguments here are largely an exercise in verification of requirements of this theorem. 

\bc$\rule{3.5in}{0.1mm}$\ec

\begin{proof}[Proof of Theorem \ref{thm:limitingdist}] In this vanishing jump regime of $\psi\to 0,$ the applicability of the argmax theorem requires verification of the following conditions (see, page 288 of \cite{vaart1996weak}).
	\benr
	&(i)& {\rm The\, sequence}\,\, \psi^{2}T(\tilde\tau-\tau^0)\,\, {\rm is\, uniformly\, tight}.\nn\\
	&(ii)& {\rm For\, any}\,\, r\in [-c_u,c_u]\subseteq\R\,\, {\rm we\, have,}\,\, \cC(\tau^0+rT^{-1}\psi^{-2},\h\mu,\h\g)\Rightarrow Z(r).\nn\\
	&(iii)& {\rm The\, process}\,\,Z(r) \,\, {\rm satisfies\, suitable\, regularity\, conditions}.\footnotemark. \nn
	\eenr
	\footnotetext{Almost all sample paths $\z\to \big\{2\si_{\iny}W(\z)-|\z|\}$ are upper semicontinuous and posses a unique maximum at a (random) point $\argmax_{\z\in\R}\big\{2\si_{\iny}W(\z)-|\z|\},$ which as a random map in the indexing metric space is tight.}	
	We begin by noting that the sequence of r.v.'s under consideration here is  $\psi^{2}T(\tilde\tau-\tau^0),$ which are supported on $\R,$ which forms the underlying indexing metric space for the limiting process under consideration for this vanishing jump size case. Now Part (i) follows from the result of Theorem \ref{thm:optimalapprox} and Part (iii) follows from well known properties of Brownian motion's. Thus, it only remains to prove Part (ii). For this purpose,  let $\tau^*=\tau^0+rT^{-1}\psi^{-2},$ with $r\in(0,c_1],$ then using Lemma \ref{lem:design.convergence.for.limiting.dist} we have,
	\benr\label{eq:17}
	p^{-1}\sum_{\lfloor T\tau^0\rfloor+1}^{\lfloor T\tau^*\rfloor}\sum_{j=1}^p\eta^{0T}_{(j)}z_{t,-j}z_{t,-j}^T\eta^0_{(j)}\to_p r\si_2^2.
	\eenr
	Also, let $\z_t=\sum_{j=1}^p\z_{tj}=\sum_{j=1}^p\vep_{tj}z_{t,-j}^T\eta^0_{(j)},$ then from Condition D(ii) we have that $\z_t^*=\xi_{2,2}^{-2}p^{-1}{\rm var}\big(\z_t\big)\to\si_2^{*2},$ thus the sequence $\{\z_t^*\}$ are finite variance i.i.d. random variables\footnote{More precisely, sequence $\{\z_t^*\}$ forms an i.i.d triangular array}, now applying the function central limit theorem on the sequence $\{\z_t^*\}$ in $t,$ we obtain,
	\benr\label{eq:20}
	p^{-1}\sum_{t=\lfloor T\tau^0\rfloor+1}^{\lfloor T\tau^*\rfloor}\sum_{j=1}^p\vep_{tj}z_{t,-j}\eta^0_{(j)}&=&\psi\sum_{t=\lfloor T\tau^0\rfloor+1}^{\lfloor T\tau^*\rfloor}\psi^{-1}p^{-1}\sum_{j=1}^p\z_{tj}\nn\\
	&=&\psi\sum_{t=\lfloor T\tau^0\rfloor+1}^{\lfloor T\tau^*\rfloor}\big\{\xi^{-1}_{2,2}p^{-1/2}\sum_{j=1}^p\z_{tj}\big\}\nn\\
	&=& \psi\sum_{t=\lfloor T\tau^0\rfloor+1}^{\lfloor T\tau^*\rfloor}\z_t^*\Rightarrow \si_2^*W_2(r),
	\eenr
	where $W_2(r)$ is a Brownian motion on $[0,\iny).$ Now define the process,
	\benr
	G(r)=\begin{cases} 2\si_1^{*}W_1(r)+\si_1^2r & {\rm if}\,\,r<0,\\
		0, & {\rm if}\,\,r=0,\\
		2 \si_2^{*}W_2(r)-\si^2_2r & {\rm if}\,\, r>0,
	\end{cases}
	\eenr
	and consider the function $\cC$ evaluated at $\tau^*$ and at the known nuisance parameters.
	\benr\label{eq:knownmeanconv}
	\cC(\tau^*,\mu^0,\g^0)&=&-p^{-1}\sum_{t=\lfloor T\tau^0\rfloor+1}^{\lfloor T\tau^*\rfloor}\sum_{j=1}^p\big(z_{tj}-z_{t,-j}^T\mu^0_{(j)}\big)^2\nn\\
	&&+p^{-1}\sum_{t=\lfloor T\tau^0\rfloor+1}^{\lfloor T\tau^*\rfloor}\sum_{j=1}^p\big(z_{tj}-z_{t,-j}^T\g^0_{(j)}\big)^2 \nn\\
	&=&2p^{-1}\sum_{t=\lfloor T\tau^0\rfloor+1}^{\lfloor T\tau^*\rfloor}\sum_{j=1}^p\vep_{tj}z_{t,-j}^T\eta^0_{(j)}\nn\\
	&&-p^{-1}\sum_{t=\lfloor T\tau^0\rfloor+1}^{\lfloor T\tau^*\rfloor}\sum_{j=1}^p\eta^{0T}_{(j)}z_{t,-j}z_{t,-j}^T\eta^0_{(j)}\nn\\
	&\Rightarrow& \big\{2 \si_2^*W_2(r)-\si^2_2 r\big\},
	\eenr
	where the convergence in distribution follows from (\ref{eq:17}) and (\ref{eq:20}). Next, from Lemma \ref{lem:est.known.cC.approx} we have that,
	\benr\label{eq:op1approx}
	\sup_{\tau\in\cG\big((c_1T^{-1}\psi^{-2}),0\big)} \big|\cC(\tau,\h\mu,\h\g)-\cC(\tau,\mu^0,\g^0)\big|=o_p(1).
	\eenr
	Combining the results of (\ref{eq:op1approx}) and (\ref{eq:knownmeanconv}) we obtain,
	\benr
	\cC(\tau^*,\h\mu,\h\g)\Rightarrow \big\{2 \si_2^*W_2(r)-\si^2_2 r\big\}\nn
	\eenr
	Symmetrical arguments for the case of $r<0$ yields an analogous result. Finally, a change of variable yields the relation, $\argmin_r G(r)=^d\big(\si_1^{*2}\big/\si_1^4\big)\argmin_r Z(r),$ where $Z(r)$ is as defined in (\ref{def:Zr}) and $=^d$ represents equality in distribution, see, e.g. proof of Proposition 3 of \cite{bai1997estimation}. This completes the proof of Part (ii) and the statement of this theorem now follows as an application of the argmax theorem.  
\end{proof}
\bc$\rule{3.5in}{0.1mm}$\ec
\begin{proof}[Proof of Theorem \ref{thm:wc.non.vanishing}] The broad structure of the argument of this proof is similar to that of the proof of Theorem \ref{thm:limitingdist} in the sense that this proof is also an application of the argmax theorem. 
	
	The first important distinction is that the sequence of r.v's under consideration  $\big(\lfloor T\tilde\tau\rfloor-\lfloor T\tau^0\rfloor\big),$ are supported on the set of integers $\Z.$ Consequently, the underlying indexing metric space for the limiting process for this non-vanishing jump size framework is the set of integers $\Z.$ Now consider any $c_u>0$ and $r\in\{-c_u,-c_u+1,...,0,1,...,c_u\}\subseteq\Z.$ Let $\lfloor T\tau^*\rfloor=\lfloor T\tau^0\rfloor+r$, then the requirements for the applicability of the argmax theorem requires verification of the following conditions.
	\benr
	&(i)& {\rm The\, sequence}\,\, \big(\lfloor T\tilde\tau\rfloor
	-\lfloor T\tau^0\rfloor\big)\,\, {\rm is\, uniformly\, tight}.\nn\\
	&(ii)& \cC(\tau^*,\h\mu,\h\g)\Rightarrow \cC_{\iny}(r).\nn\\
	&(iii)& {\rm The\, process}\,\,\cC_{\iny}(r) \,\, {\rm satisfies\, suitable\, regularity\, conditions}. \nn
	\eenr
	Part (i) follows directly from the result of Theorem \ref{thm:optimalapprox}. Part (iii) is provided in Lemma \ref{lem:regularity.argmax}. A verification of Part (ii) is provided below. Let $r>0$ and consider $\cC$ evaluated at $\tau^*$ and at the known nuisance parameters, i.e., 
	\benr\label{eq:knownmeanconv2}
	\cC(\tau^*,\mu^0,\g^0)&=&-p^{-1}\sum_{t=\lfloor T\tau^0\rfloor+1}^{\lfloor T\tau^0\rfloor+r}\sum_{j=1}^p\big(z_{tj}-z_{t,-j}^T\mu^0_{(j)}\big)^2\nn\\
	&&+p^{-1}\sum_{t=\lfloor T\tau^0\rfloor+1}^{\lfloor T\tau^0\rfloor+r}\sum_{j=1}^p\big(z_{tj}-z_{t,-j}^T\g^0_{(j)}\big)^2 \nn\\
	&=&2p^{-1}\sum_{t=\lfloor T\tau^0\rfloor+1}^{\lfloor T\tau^0\rfloor+r}\sum_{j=1}^p\vep_{tj}z_{t,-j}^T\eta^0_{(j)}\nn\\
	&&-p^{-1}\sum_{t=\lfloor T\tau^0\rfloor+1}^{\lfloor T\tau^0\rfloor+r}\sum_{j=1}^p\eta^{0T}_{(j)}z_{t,-j}z_{t,-j}^T\eta^0_{(j)}\nn\\
	&=&p^{-1}\sum_{t=\lfloor T\tau^0\rfloor+1}^{\lfloor T\tau^0\rfloor+r}\sum_{j=1}^p\Big\{ 2\vep_{tj}z_{t,-j}^T\eta^0_{(j)}-\eta^{0T}_{(j)}z_{t,-j}z_{t,-j}^T\eta^0_{(j)}\Big\}\nn\\
	&\Rightarrow& \sum_{t=1}^{r}\cL\big(-\psi_{\iny}^2\si^2_2,\,\,\bar\si_2^{2}\big)
	\eenr
	Here weak convergence follows directly from Condition B$'$ and since $r\le c_u,$ which in turn is due to the non-vanishing jump size regime under consideration. Next, from Lemma \ref{lem:est.known.cC.approx} we have,
	\benr
	\sup_{\tau\in\cG\big((c_1T^{-1}\psi^{-2}),0\big)} \big|\cC(\tau,\h\mu,\h\g)-\cC(\tau,\mu^0,\g^0)\big|=o_p(1).\nn
	\eenr
	This result together with (\ref{eq:knownmeanconv2}) yields the statement of Part (ii).  Repeating the same argument for $r<0$ yields the symmetric result.  An application of the argmax theorem now yields the statement of this theorem.  
\end{proof} 

\bc$\rule{3.5in}{0.1mm}$\ec

\begin{lem}[Regularity conditions of $\argmax\cC_{\iny}(r)$]\label{lem:regularity.argmax} Let $\cC_{\iny}(r)$ be as defined in (\ref{def:cCr}) and suppose Condition B$'$ holds. Then the map $r\to \cC_{\iny}(r)$ is continuous with respect to the domain space $\Z.$ Additionally suppose Condition D and that the jump size is non-vanishing, i.e, $0<\psi_{\iny}<\iny.$ Then $\argmax_{r\in\Z}\cC_{\iny}(r)$ possesses an almost sure unique maximum at $\omega_{\iny,}$ which as a random map in $\Z$ is tight.
\end{lem}

\begin{proof}[Proof of Lemma \ref{lem:regularity.argmax}] From Condition B$',$ each side of the random walk $\cC(r)$ has increments supported on $\R,$ thus the first assertion on the continuity of the map $r\to \cC_{\iny}(r)$ follows trivially since the domain space of this map is restricted to only the integers $\Z$ ($\ep-\delta$ definition of continuity). To prove the remaining assertions note that from Condition B$'$, Condition D and the assumed framework of the non-vanishing jumpsize, we have that each side of $\cC_{\iny}(r)$ has i.i.d increments with a negative drift of $-\psi_{\iny}^2\si_1^2$ or $-\psi_{\iny}^2\si_2^2.$ Consequently, we have $\cC_{\iny}(r)\to-\iny,$ as $r\to \iny$ almost surely (strong law of large numbers). Using elementary properties of random walks, this implies that $\max_{r}\cC_{\iny}(r)<\iny,$ a.s. (follows from the Hewitt-Savage $0$-$1$ law, see, e.g. (1.1) and (1.2) on Page 172, 173 of \cite{durrett2010probability}). Additionally $\om_{\iny}\ge 0,$ from the construction of $\cC_{\iny}(r).$ Thus, we have $0\le \om_{\iny}<\iny,$ a.s.  which directly implies that when $\om_{\iny}$ is well defined (unique) then it must be tight. To show that $\om_{\iny}$ is unique, note that since by assumption (Condition B$'$) the increments are continuously distributed and supported on $\R$, therefore $\max\cC_{\iny}(r)$ is continuously distributed on $(0,\iny),$ with some additional probability mass at the singleton zero. Hence, the probability of $\max\cC_{\iny}(r)$ attaining any two identical values is zero. Consequently $\om_{\iny}$ is unique a.s. This completes the proof of this lemma. 
\end{proof}

\bc$\rule{3.5in}{0.1mm}$\ec

\section{Proofs of results in Section 3}

The main result of Section \ref{sec:nuisance} is Theorem \ref{thm:alg1.nearoptimal}, which forms the basis of the subsequent corollaries. The proof of Theorem \ref{thm:alg1.nearoptimal} requires some preliminary work in the form of Theorem \ref{thm:est.nuisance.para}, Lemma \ref{lem:check.mu.g} and Lemma \ref{lem:lower.b.near.optimal} below. We begin with Theorem \ref{thm:est.nuisance.para} that provides uniform bounds (over $\tau$) of the $\ell_2$ error in the lasso estimates (\ref{est:lasso}) obtained from a regression of each column of $z$ on the rest.

\begin{thm}\label{thm:est.nuisance.para}
	Suppose Condition A$'$ and B holds. Let $u_T\ge 0$ and $\la_j=2(\la_{1j}+\la_{2j}),$ where
	\benr
	\la_{1j}=c_u\si^2\surd(1+\nu^2)\Big\{\frac{\log (p\vee T)}{Tl_T}\Big\}^{\frac{1}{2}},\,\,\,
	\la_{2j}=c_u(\si^2\vee\phi)\|\eta^0_{(j)}\|_2\max\Big\{\frac{\log (p\vee T)}{Tl_T},\,\,\frac{u_T}{l_T}\Big\}\nn
	\eenr
	Then uniformly over all $j=1,...,p,$ the following two properties hold with probability at least $1-c_{u2}\exp \big\{-(c_{u3}\log(p\vee T) \big\},$ for some $c_{u2},c_{u3}>0.$\\~
	(i) The vectors $\h\mu_{(j)}(\tau)-\mu_{(j)}^0\in\cA_{1j},$ and $\h\g_{(j)}(\tau)-\g_{(j)}^0\in\cA_{2j},$ where the sets $\cA_{ij},$ $i=1,2,$ and $j=1,...,p$ are as defined in Condition C.\\~
	(ii) For any constant $c_{u1}>0,$ we have,
	\benr
	\sup_{\substack{\tau\in\cG(u_T,0);\\ (\lfloor T\tau\rfloor)\wedge(T-\lfloor T\tau\rfloor)\ge c_{u1}Tl_T}}\|\h\mu_{(j)}(\tau)-\mu^0_{(j)}\|_2\le c_u\frac{\surd{s}}{\ka}\la_j.\nn
	\eenr
	The same upper bounds also hold for $\h\g_{(j)}(\tau)-\g^0_{(j)},$ uniformly over $j$ and $\tau.$
\end{thm}

\begin{proof}[Proof of Theorem \ref{thm:est.nuisance.para}]
	Consider any $\tau\in\cG(u_T,0),$ and w.l.o.g. assume that $\tau\ge\tau^0.$ Then for any $j=1,..,p,$ by construction of the estimator $\h\mu_{(j)}(\tau),$ we have the basic inequality,
	\benr
	\frac{1}{\lfloor T\tau\rfloor}\sum_{t=1}^{\lfloor T\tau\rfloor} \big(z_{tj}- z_{t, -j}^T\h\mu_{(j)}(\tau)\big)^2 + \la_j\|\h\mu_{(j)}(\tau)\|_1
	\le \frac{1}{\lfloor T\tau\rfloor}\sum_{t=1}^{\lfloor T\tau\rfloor} \big(z_{tj}- z_{t, -j}^T\mu_{(j)}^0\big)^2 + \la_j\|\mu_{(j)}^0\|_1.\nn
	\eenr
	An algebraic rearrangement of this inequality yields,
	\benr
	\frac{1}{\lfloor T\tau\rfloor}\sum_{t=1}^{\lfloor T\tau\rfloor}\big(z_{t,-j}^T(\h\mu_{(j)}-\mu_{(j)}^0)\big)^2+\la_j\|\h\mu_{(j)}(\tau)\|_1\le \la_j\|\mu_{(j)}^0\|_1+ \frac{2}{\lfloor T\tau\rfloor}\sum_{t=1}^{\lfloor T\tau\rfloor} \tilde\vep_{tj}z_{t,-j}^T(\h\mu_{(j)}-\mu_{(j)}^0),\nn
	\eenr
	where $\tilde\vep_{tj}=\vep_{tj}=z_{tj}-z_{t,-j}^T\mu_{(j)}^0,$ for $t\le \lfloor T\tau^0\rfloor,$ and $\tilde\vep_{tj}=z_{tj}-z_{t,-j}^T\mu_{(j)}^0=\vep_{tj}-z_{t,-j}^T(\mu_{(j)}^0-\g_{(j)}^0),$ for $t>\lfloor T\tau^0\rfloor.$ A further simplification using these relations yields,
	\benr\label{eq:23}
	\frac{1}{\lfloor T\tau\rfloor}\sum_{t=1}^{\lfloor T\tau\rfloor}\big(z_{t,-j}^T(\h\mu_{(j)}-\mu_{(j)}^0)\big)^2+\la_j\|\h\mu_{(j)}(\tau)\|_1\le\hspace{2.5cm}\nn\\ \la_j\|\mu_{(j)}^0\|_1+ \frac{2}{\lfloor T\tau\rfloor}\sum_{t=1}^{\lfloor T\tau\rfloor} \vep_{tj}z_{t,-j}^T(\h\mu_{(j)}-\mu_{(j)}^0)\nn\\
	-\frac{2}{\lfloor T\tau\rfloor}\sum_{t=\lfloor T\tau^0\rfloor+1}^{\lfloor T\tau\rfloor} (\mu_{(j)}^0-\g_{(j)}^0)z_{t,-j}z_{t,-j}^T(\h\mu_{(j)}-\mu_{(j)}^0)\hspace{-1cm}\nn\\
	\le\la\|\mu_{(j)}^0\|_1+  \frac{2}{\lfloor T\tau\rfloor}\big\|\sum_{t=1}^{\lfloor T\tau\rfloor} \vep_{tj}z_{t,-j}^T\big\|_{\iny}\|\h\mu_{(j)}-\mu_{(j)}^0\|_1\hspace{-0.5cm}\nn\\
	+\frac{2}{\lfloor T\tau\rfloor}\big\|\sum_{t=\lfloor T\tau^0\rfloor+1}^{\lfloor T\tau\rfloor} (\mu_{(j)}^0-\g_{(j)}^0)z_{t,-j}z_{t,-j}^T\big\|_{\iny}\|\h\mu_{(j)}-\mu_{(j)}^0\|_1\hspace{-2cm}
	\eenr
	Now using the bounds of Lemma \ref{lem:bounds.for.nuis.thm} we have that,
	\benr
	&&\frac{1}{\lfloor T\tau\rfloor}\Big\|\sum_{t=1}^{\lfloor T\tau\rfloor}\vep_{tj}z_{t,-j}\Big\|_{\iny}\le c_u\si^2\surd(1+\nu^2)\Big\{\frac{\log(p\vee T)}{Tl_T}\Big\}^{\frac{1}{2}}=\la_{1j}\nn\\
	&&
	\frac{1}{\lfloor T\tau\rfloor} \Big\|\sum_{t=\lfloor T\tau^0\rfloor+1}^{\lfloor T\tau\rfloor}\eta^{0T}_{(j)}z_{t,-j}z_{t,-j}^T\Big\|_{\iny}\le c_{u}(\si^2\vee\phi)\|\eta^0_{(j)}\|_2\max\Big\{\frac{\log (p\vee T)}{Tl_T},\,\,\frac{u_T}{l_T}\Big\}=\la_{2j},\nn
	\eenr
	with probability at least $1-c_{u2}\exp\{-c_{u3}\log(p\vee T)\}.$ Applying these bounds in (\ref{eq:23}) yields,
	\benr
	\frac{1}{\lfloor T\tau\rfloor}\sum_{t=1}^{\lfloor T\tau\rfloor}\big(z_{t,-j}^T(\h\mu_{(j)}-\mu_{(j)}^0)\big)^2+\la_j\|\h\mu_{(j)}(\tau)\|_1\le \la_j\|\mu_{(j)}^0\|_1 + (\la_{1j}+ \la_{2j})\|\h\mu_{(j)}(\tau)-\mu_{(j)}^0\|_1,\nn
	\eenr
	with probability at least $1-c_{u2}\exp\{-c_{u3}\log(p\vee T)\}.$ Choosing $\la_j\ge 2(\la_{1j}+ \la_{2j}),$ leads to $\big\|\big(\h\mu_{(j)}(\tau)\big)_{S_{1j}^c}\big\|_1\le 3\big\|\big(\h\mu_{(j)}(\tau)-\mu_j^0\big)_{S_{1j}}\big\|_1,$ and thus by definition $\h\mu_{(j)}-\mu^0_{(j)}\in\cA_{1j},$ with the same probability. This proves the first assertion of this theorem. Next applying the restricted eigenvalue condition of (\ref{lem:lower.RE.ordinary}) to the l.h.s. of the inequality (\ref{eq:23}), we also have that,
	\benr
	\ka\|\h\mu_{(j)}(\tau)-\mu^0_{(j)}\|_2^2\le 3\la\|\h\mu_{(j)}(\tau)-\mu^0_{(j)}\|_1\le 3\surd{s}\la_j\|\h\mu_{(j)}(\tau)-\mu^0_{(j)}\|_2.\nn
	\eenr
	This directly implies that $\|\h\mu_{(j)}(\tau)-\mu^0_{(j)}\|_2\le 3\surd{s}(\la_j/\ka),$ which yields the desired $\ell_2$ bound. To finish the proof recall that the stochastic bounds used here hold uniformly over $\cG(u_T,0),$ and $j,$ consequently the statements of this theorem also hold uniformly over the same collections. The case of $\tau\le\tau^0,$ and the corresponding results for $\h\g_{(j)}(\tau)-\g^0_{(j)}$ can be obtained by symmetrical arguments.
\end{proof}

\bc$\rule{3.5in}{0.1mm}$\ec

The following lemma obtains $\ell_2$ error bounds for the Step 1 edge estimates by utilizing the initializing Condition E and Theorem \ref{thm:est.nuisance.para}.

\begin{lem}\label{lem:check.mu.g} Suppose Condition A$',$ B and E hold. Choose regularizers $\la_j,$ $j=1,...,p,$ as prescribed in Theorem \ref{thm:est.nuisance.para}, with $u_T=\big(c_ul_T\ka\big)\big/\big(sT^k(\si^2\vee\phi)\big).$ Then edge estimates $\check\mu_{(j)},$ $j=1,...,p$ of Step 1 of Algorithm 1 satisfy the following bound.
	\benr	
	(i)\surd{s}\sum_{j=1}^p\|\check\mu_{(j)}-\mu^0_{(j)}\|_2\le \frac{c_u\xi_{2,1}}{T^k},\,\,{\rm and}\,\,
	(ii)\Big(s\sum_{j=1}^p\|\check\mu_{(j)}-\mu^0_{(j)}\|_2^2\Big)^{\frac{1}{2}}\le \frac{c_u\xi_{2,2}}{T^k}\nn	
	\eenr
	with probability $1-o(1).$ Corresponding bounds also holds for $\check\g_{(j)},$ $j=1,...,p.$
\end{lem}

\begin{proof}[Proof of Lemma \ref{lem:check.mu.g}] We begin by noting that Part (ii) of the initializing Condition E of Algorithm 1 guarantees that $\check\tau$ satisfies,
	\benr
	|\lfloor T\check\tau\rfloor-\lfloor T\tau^0\rfloor|\le \frac{c_ul_T\ka}{s(\si^2\vee\phi)}T^{(1-k)}\nn
	\eenr
	In other words, $\check\tau\in\cG(u_T,0),$ where $u_T=\big(c_ul_T\ka\big)\big/\big(sT^k(\si^2\vee\phi)\big),$ where $k<b.$ This choice of $u_T$ provides the following relations,
	\benr
	&&\frac{u_T}{l_T}=\frac{c_u\ka}{(\si^2\vee\phi)T^ks}\ge \frac{\log(p\vee T)}{Tl_T}.\label{eq:27}\\
	&&c_u(\si^2\vee\phi)\surd s\frac{\xi_{2,1}u_T}{\ka l_T}= \frac{c_u\xi_{2,1}}{T^k\surd{s}}\ge c_u\si^2\surd(1+\nu^2)\frac{p}{\ka}\Big\{\frac{s\log (p\vee T)}{Tl_T}\Big\}^{\frac{1}{2}}\label{eq:28}
	\eenr
	Here the inequality of (\ref{eq:27}) follows from the assumption $c_u\ka T^{(1-k)}l_T\ge (\si^2\vee\phi)s\log (p\vee T)$ of Condition E. The equality of (\ref{eq:28}) follows directly upon substituting the choice of $u_T,$ and the inequality follows from assumption A$'$(iii) and since w.l.o.g we have $k<b.$ Now using this choice of $u_T$ in $\la_j$ of Part (ii) of Theorem \ref{thm:est.nuisance.para} we obtain,
	\benr
	\sum_{j=1}^p\frac{\surd s}{\ka}(\la_{1j}+\la_{2j})&\le& c_u\si^2\surd(1+\nu^2)\frac{p}{\ka}\Big\{\frac{s\log (p\vee T)}{Tl_T}\Big\}^{\frac{1}{2}}\nn\\
	&&+c_u(\si^2\vee\phi)\xi_{2,1}\frac{\surd s}{\ka}\Big\{\frac{\log(p\vee T)}{Tl_T},\,\,\frac{u_T}{l_T}\Big\}\nn\\
	&\le&  c_u\si^2\surd(1+\nu^2)\frac{p}{\ka}\Big\{\frac{s\log (p\vee T)}{Tl_T}\Big\}^{\frac{1}{2}}\nn\\
	&&+c_u(\si^2\vee\phi)\Big\{\frac{\xi_{2,1}u_T\surd s}{\ka l_T}\Big\}\le c_u\frac{\xi_{2,1}}{T^k\surd{s}}.\nn
	\eenr
	The second inequality follows from (\ref{eq:27}) and the final inequality follows from (\ref{eq:28}).  The bound of Part (i) is now a direct consequence of Theorem \ref{thm:est.nuisance.para}. We proceed similarly to prove Part (ii); note that,
	\benr
	\sum_{j=1}^p\frac{s}{\ka^2}(\la_{1j}+\la_{2j})^2&\le& c_u\si^4(1+\nu^2)\frac{p}{\ka^2}\Big\{\frac{s\log (p\vee T)}{Tl_T}\Big\}\nn\\
	&&+c_u(\si^4\vee\phi^2)\xi_{2,2}^2\frac{s}{\ka^2}\Big\{\frac{\log(p\vee T)}{Tl_T},\,\,\frac{u_T}{l_T}\Big\}^2\nn\\
	&\le& c_u\si^4(1+\nu^2)\frac{p}{\ka^2}\Big\{\frac{s\log (p\vee T)}{Tl_T}\Big\}+c_u(\si^4\vee\phi^2)\Big\{\frac{\xi_{2,2}u_T\surd s}{\ka l_T}\Big\}^2\nn\\
	&\le&c_u\si^4(1+\nu^2)\frac{p}{\ka^2}\Big\{\frac{s\log (p\vee T)}{Tl_T}\Big\} +\frac{c_u\xi_{2,2}^2}{sT^{2k}}\le \frac{c_u\xi_{2,2}^2}{sT^{2k}}.\nn
	\eenr
	The final inequality follows from Condition A$'$(iii). Part (ii) is now a direct consequence.
\end{proof}

\bc$\rule{3.5in}{0.1mm}$\ec

\begin{lem}\label{lem:lower.b.near.optimal} Suppose Condition A$'$, B and E hold and let $\check\mu_{(j)}$ and $\check\g_{(j)},$ $j=1,...,p$ be edge estimates of Step 1 of Algorithm 1. Additionally, let  $\log(p\vee T)\le Tv_T\le Tu_T$ be non-negative sequences. Then,
	\benr
	\inf_{\tau\in\cG(u_T,v_T)}\cU(z,\tau,\h\mu,\h\g)\ge \ka\xi^{2}_{2,2}\Big[v_T-c_m\max\Big\{\Big(\frac{u_T\log(p\vee T)}{T}\Big)^{\frac{1}{2}},\,\,\frac{u_T}{T^{k}}\Big\}\Big]\nn
	\eenr
	with probability at least $1-o(1).$ Here $c_m=\{c_u(\si^2\vee\phi)\surd(1+\nu^2)\big\}\big/\big\{\ka(1\wedge \phi)\big\}.$
\end{lem}

\begin{proof}[Proof of Lemma \ref{lem:lower.b.near.optimal}] The structure of this proof is similar to that of Lemma \ref{lem:mainlowerb}, the distinction being the use of weaker available error bounds of the edge estimates $\check\mu_{(j)},$ $\check\g_{(j)},$ and sharper bounds for other stochastic terms made possible by the additional assumption $\log(p\vee T)\le Tv_T\le Tu_T.$ Proceeding as in (\ref{eq:11}) we have that,
	\benr\label{eq:29}
	\inf_{\substack{\tau\in\cG(u_T,v_T);\\ \tau\ge\tau^0}}\cU(z,\tau,\check\mu,\check\g)&\ge& R1-R2-R3\nn
	\eenr
	Where $R1,R2$ and $R_3$ are as defined in (\ref{eq:11}) with $\h\mu_{(j)},$ $\h\g_{(j)}$ and $\h\eta_{(j)}$ replaced with $\check\mu_{(j)},$ $\check\g_{(j)}$ and $\check\eta_{(j)}=\check\mu_{(j)}-\check\g_{(j)},$ $j=1,...,p.$ Now applying the bounds of Lemma \ref{lem:term123.check} we obtain,
	\benr
	R1&\ge& \ka\xi_{2,2}^2\Big[v_T-\frac{c_{u}\si^2}{\ka}\Big\{\frac{u_T\log(p\vee T)}{T}\Big\}^{\frac{1}{2}}- c_u(\si^2\vee \phi)\frac{u_T}{\ka\xi_{2,2}}\Big(s\sum_{j=1}^p\|\h\eta_{(j)}-\eta^0_{(j)}\|_2^2\Big)^{\frac{1}{2}}\Big]\nn\\
	&\ge&\ka\xi_{2,2}^2\Big[v_T-\frac{c_{u}\si^2}{\ka}\Big\{\frac{u_T\log(p\vee T)}{T}\Big\}^{\frac{1}{2}}- c_u(\si^2\vee \phi)\frac{u_T}{T^k\ka}\Big]\nn
	\eenr
	with probability $1-o(1).$ Where the final inequality follows from Lemma \ref{lem:check.mu.g}. Next we obtain upper bounds for the terms $R2\big/\ka\xi_{2,2}^2$ and $R3\big/\ka\xi_{2,2}^2.$ Consider,
	\benr
	\frac{R2}{\ka\xi_{2,2}^2}&\le& 	c_u\surd(1+\nu^2)\frac{\si^2\xi_{2,1}}{\ka\xi^2_{2,2}}\Big(\frac{u_T\log(p\vee T)}{T}\Big)^{\frac{1}{2}}\nn\\
	&&+ c_{u}\surd(1+\nu^2)\frac{\si^2}{\ka\xi^2_{2,2}}\Big(\frac{u_T\log(p\vee T)}{T}\Big)^{\frac{1}{2}}\sum_{j=1}^p\|\check\eta_{(j)}-\eta^0_{(j)}\|_1\nn\\
	&\le& c_u\surd(1+\nu^2)\frac{\si^2}{\ka\psi}\Big(\frac{u_T\log(p\vee T)}{T}\Big)^{\frac{1}{2}}+ c_{u}\surd(1+\nu^2)\frac{\si^2}{\ka\psi}\Big(\frac{u_T\log(p\vee T)}{T}\Big)^{\frac{1}{2}}\frac{1}{T^k}\nn\\
	&\le&  c_u\surd(1+\nu^2)\frac{\si^2}{\ka\psi}\Big(\frac{u_T\log(p\vee T)}{T}\Big)^{\frac{1}{2}}\nn
	\eenr
	with probability $1-o(1).$ Here the first and second inequalities follow from Lemma \ref{lem:term123.check} and Lemma \ref{lem:check.mu.g}, respectively. Similarly we can also obtain,
	\benr
	\frac{R3}{\ka\xi_{2,2}^2}&\le& c_u(\si^2\vee \phi)\frac{u_T}{\ka\xi_{2,2}}\Big\{s\sum_{j=1}^p\|\check\g_{(j)}-\g^0_{(j)}\|_2^2\Big\}^{\frac{1}{2}}\Big[1+\frac{1}{\xi_{2,2}}\Big\{s\sum_{j=1}^p\|\check\eta_{(j)}-\eta^0_{(j)}\|_2^2\Big\}^{\frac{1}{2}}\Big]\nn\\
	&\le& c_{u1}(\si^2\vee \phi)\frac{u_T}{\ka T^{k}}\nn
	\eenr
	with probability $1-o(1).$ Substituting these bounds in (\ref{eq:29}) and applying a union bound over these three events yields the bound of the statement of this lemma uniformly over the set $\{\cG(u_T,v_T);\,\tau\ge\tau^0\}.$ The mirroring case of $\tau\le\tau^0$ can be obtained by similar arguments.
\end{proof}

\bc$\rule{3.5in}{0.1mm}$\ec

Following is the proof of the main result of Section \ref{sec:nuisance}.

\begin{proof}[Proof of Theorem \ref{thm:alg1.nearoptimal}] This proof relies on the same recursive argument as that of Theorem \ref{thm:optimalapprox}, the distinction being that recursions are made on the bound of Lemma \ref{lem:lower.b.near.optimal} instead of Lemma \ref{lem:mainlowerb}. Consider any $Tv_T>\log(p\vee T),$ and apply Lemma \ref{lem:lower.b.near.optimal} on the set $\cG(u_T,v_T)$ to obtain,
	\benr
	\inf_{\tau\in\cG(1,v_T)} \cU(z,\tau,\check\mu,\check\g)&\ge& \ka\xi^2_{2,2}\Big[v_T-c_m\max\Big\{\Big(\frac{u_T\log(p\vee T)}{T}\Big)^{\frac{1}{2}},\frac{u_T}{T^{k}}\Big\}\Big]\nn\\
	&\ge&\ka\xi^2_{2,2}\Big[v_T-c_m\max\Big\{\Big(\frac{\log(p\vee T)}{T}\Big)^{\frac{1}{2}},\Big(u_T\frac{\log(p\vee T)}{T}\Big)^k\Big\}\Big]\nn
	\eenr
	with probability at least $1-o(1).$ Substituting $u_T=1,$ yields,
	\benr
	\inf_{\tau\in\cG(1,v_T)} \cU(z,\tau,\check\mu,\check\g)\ge \ka\xi^2_{2,2}\Big[v_T-c_m\max\Big\{\Big(\frac{\log(p\vee T)}{T}\Big)^{\frac{1}{2}},\Big(\frac{\log(p\vee T)}{T}\Big)^k\Big\}\Big]\nn
	\eenr	
	with probability at least $1-o(1).$ Recall that w.l.og $k<b<(1/2),$ and now choose any $v_T>v_T^*=c_{m}\big(\log(p\vee T)/T\big)^k.$ Then we have $\inf_{\tau\in\cG(1,v_T)}\cU(z,\tau,\check\mu,\check\g)>0,$ thus implying that $\h\tau\notin\cG(1,v_T),$ i.e., $\big|\lfloor T\check\tau\rfloor-\lfloor T\tau^0\rfloor\big|\le Tv_T^*,$ with probability at least $1-o(1).$ Now reset $u_T=v_T^*$ and reapply Lemma \ref{lem:mainlowerb} for any $v_T>0$ to obtain,
	\benr
	\inf_{\tau\in\cG(u_T,v_T)} \cU(z,\tau,\check\mu,\check\g)\ge \ka\xi^2_{2,2}\Big[v_T-c_{m}\max\Big\{c_{m}^{1/2}\Big(\frac{\log(p\vee T)}{T}\Big)^{\frac{1}{2}+\frac{k}{2}},c_{m}\Big(\frac{\log(p\vee T)}{T}\Big)^{k+k}\Big\}\Big]\nn
	\eenr
	Again choosing any,
	\benr
	v_T>v_T^*=\max\Big\{c_{m}^{g_2}\Big(\frac{\log(p\vee T)}{T}\Big)^{u_2},\,\,c_{m}^{2}\Big(\frac{\log(p\vee T)}{T}\Big)^{v_2}\Big\},
	\eenr
	where,
	\benr
	g_2=1+\frac{1}{2},\,\, u_2=\frac{1}{2}+\frac{u_1}{2},\,\,{\rm and}\,\, v_2=k+v_1\ge 2k,\,\,{\rm with}\,\,u_1=v_1=k,\nn
	\eenr
	we obtain $\inf_{\cG(u_T,v_T)}\cU(z,\tau,\check\mu,\check\g)>0,$ with probability at least $1-o(1).$ Consequently $\h\tau\notin\cG(u_T,v_T),$ i.e., $\big|\lfloor T\h\tau\rfloor-\lfloor T\tau^0\rfloor\big|\le Tv_T^*.$ Continuing these recursions by resetting $u_T$ to the bound of the previous recursion, and applying Lemma \ref{lem:mainlowerb}, we obtain for the $l^{th}$ recursion,
	\benr
	&&\big|\lfloor T\tilde\tau\rfloor-\lfloor T\tau^0\rfloor\big|\le T\max\Big\{c_{m}^{g_l}\Big(\frac{\log(p\vee T)}{T}\Big)^{u_l},\,\,c_{m}^{l}\Big(\frac{\log(p\vee T)}{T}\Big)^{v_l}\Big\}\nn\\
	&&\hspace{2.7cm}:=T\max\{R_{1l},R_{2l}\},\,\,\,{\rm where,}\nn\\
	&&g_l=\sum_{j=0}^{l-1}\frac{1}{2^j},\,\, u_l=\frac{1}{2}+\frac{u_{l-1}}{2}=\frac{k}{l}+\sum_{j=1}^{l}\frac{1}{2^j},\,\,{\rm and}\nn\\
	&& v_l=k+v_{l-1}\ge lk,\,\,{\rm with}\,\,u_1=v_1=k.\nn
	\eenr
	Next, it is straightforward to observe that for $l$ large enough, $R_{2l}\le R_{1l},$ for $T$ sufficiently large. Consequently for $l$ large enough we have $\big|\lfloor T\tilde\tau\rfloor-\lfloor T\tau^0\rfloor\big|\le T R_{1m},$ with probability at least $1-o(1).$ Finally, we continue these recursions an infinite number of times to obtain, $g_{\iny}=\sum_{j=0}^{\iny}1/2^{j},$ $u_\iny=\sum_{j=1}^{\iny}(1/2^j),$ thus yielding,
	\benr
	\big|\lfloor T\tilde\tau\rfloor-\lfloor T\tau^0\rfloor\big|\le T\frac{c_{m}^2\log(p\vee T)}{T}=c_{m}^2\log(p\vee T)\nn
	\eenr
	with probability at least $1-o(1).$ This completes the proof of this result.
\end{proof}

\bc$\rule{3.5in}{0.1mm}$\ec

\begin{proof}[Proof of Corollary \ref{cor:optimalmeans.step2}] Under the assumed conditions, we have from Theorem \ref{thm:alg1.nearoptimal} that $\h\tau\in\cG(u_T,0),$ with probability at least $1-o(1),$ where $u_T=c_m^2T^{-1}\log(p\vee T),$ where $c_m$ is as defined in Lemma \ref{lem:lower.b.near.optimal}. The relation of Part (i) follows directly from Theorem \ref{thm:est.nuisance.para}. To obtain Part (ii), substitute this choice of $u_T$ in $\la_{2j},$ $j=1,...p,$ of Theorem \ref{thm:est.nuisance.para} to obtain,
	\benr
	\la_{2j}=c_u(\si^2\vee\phi)\|\eta^0_{(j)}\|_2\max\Big\{\frac{\log (p\vee T)}{Tl_T},\,\,c_m^2\frac{\log(p\vee T)}{Tl_T}\Big\}	\le o(1)\Big\{\frac{\log(p\vee T)}{Tl_T}\Big\}^{\frac{1}{2}}\nn
	\eenr
	Here the final inequality follows since by Condition A$'$(i) we have $\log(p\vee T)=o(Tl_T),$ furthermore from Lemma \ref{lem:condnumberbound} we have $\|\eta^0_{(j)}\|_2\le 2\nu,$ $j=1,...,p.$ Consequently $\la_{2j}\le \la_{1j},$ $j=1,...,p,$ and thus applying Theorem \ref{thm:est.nuisance.para} we obtain,
	\benr
	\|\h\mu_{(j)}-\mu^0_{(j)}\|_2\le c_u\la_{j}\frac{\surd{s}}{\ka}\le c_u\surd(1+\nu^2)\frac{\si^2}{\ka}\Big\{\frac{s\log (p\vee T)}{Tl_T}\Big\}^{\frac{1}{2}}\nn
	\eenr
	for all $j=1,...,p,$ with probability at least $1-o(1).$ Corresponding bound for $\h\g_{(j)}-\g^0_{(j)},$ $j=1,...,p,$ can be obtained using symmetrical arguments. This completes the proof of this corollary.
\end{proof}

\bc$\rule{3.5in}{0.1mm}$\ec

\begin{proof}[Proof of Corollary \ref{cor:final}] Note that Corollary \ref{cor:optimalmeans.step2} has established that the edge estimates $\h\mu_{(j)}$ and $\h\g_{(j)},$ $j=1,...,p,$ satisfy the requirements of Condition C of Section \ref{sec:mainresults}. Thus, this result is now a direct consequence of Theorem \ref{thm:optimalapprox} and Theorem \ref{thm:limitingdist}.
\end{proof}
\bc$\rule{3.5in}{0.1mm}$\ec

\section{Deviation bounds used for proofs in Section 2}

\begin{lem}\label{lem:subez} Suppose Condition B holds and let  $\vep_{tj}$ be as in (\ref{def:epsilons}). Then, (i) the r.v. $\vep_{tj}z_{t,-j,k}$ is sub-exponential with parameter $\la_1=48\si^2\surd{(1+\nu^2)},$ for each $j=1,...,p,$ $k=1,...,p-1$ and $t=1,...,T.$ (ii) The r.v. $\z_t=\sum_{j=1}^p\vep_{tj}z_{t,-j}^T\eta^0_{(j)}$ is sub-exponential with parameter $\la_2=48\si^2\xi_{2,1}\surd{(1+\nu^2)},$ for each $t=1,...,T.$ (iii) $E\big[|\z_t|^k\big]\le 4\la_2^k k^k,$ for any $k> 0.$
\end{lem}

\begin{proof}[Proof of Lemma \ref{lem:subez}]
	Here we only prove Part (ii) of this lemma, Part (i) follows using similar arguments, and Part (iii) follows from properties of sub-exponential random variables, see, Lemma \ref{lem:momentprop}. We begin by noting that the following r.v's are mean zero,  $E(\vep_{tj})=0,$ $E(z_{t,-j}^T\eta^0_{(j)})=0$ and  $E\big(\vep_{tj}z_{t,-j}^T\eta^0_{(j)}\big)=0.$ Also note that for $t\le \lfloor T\tau^0\rfloor,$ we have,
	\benr
	\vep_{tj}=z_{tj}-z_{t,-j}^T\mu^0_{(j)}=\big(1, -\mu^{0T}_{(j)}\big)\big(z_{tj},z_{t,-j}^T\big)^T \nn
	\eenr
	Using Lemma \ref{lem:condnumberbound} and by properties of sub-gaussian distributions we have $\vep_{tj},$ $1\le j\le p\sim {\rm subG(\si_1)}$ with $\si_1=\si\surd(1+\nu^2).$ The same also holds for $\vep_{tj},$ for $t>\lfloor T\tau^0\rfloor.$ Similarly, $z_{t,-j}\eta^0_{(j)}\sim {\rm subG(\si_2)}$ with $\si_2=\si\|\eta_{(j)}^0\|_2.$ Recall that if $Z\sim {\rm subG}(\si),$ then the rescaled variable $Z/\si\sim{\rm subG}(1).$ Next observe that,
	\benr
	\frac{\vep_{tj}z_{t,-j}^T\eta^0_{(j)}}{\si_1\si_2}=\frac{1}{2}\Big\{\Phi\Big(\frac{\vep_{tj}}{\si_1}+\frac{z_{t,-j}^T\eta^0_{(j)}}{\si_2}\Big)-\Phi\Big(\frac{\vep_{tj}}{\si_1}\Big)-\Phi\Big(\frac{z_{t,-j}^T\eta^0_{(j)}}{\si_2}\Big)\Big\}=\frac{1}{2}[T1-T2-T3]\nn	
	\eenr
	where $\Phi(v)=\|v\|_2^2-E\big(\|v\|_2^2\big).$ Using Lemma \ref{lem:lcsubG} and Lemma \ref{lem:sqsubGsubE} we have that $T1\sim{\rm subE}(64),$ $T2\sim{\rm subE}(16),$ and $T3\sim{\rm subE}(16).$ Applying Lemma \ref{lem:lcsubE} and rescaling with $\si_1,$ and $\si_2$ we obtain that $\vep_{tj}z_{t,-j}^T\eta^0_{(j)}\sim {\rm subE}(48\si_1\si_2).$ Another application of Lemma \ref{lem:lcsubE} yields $\z_t=\sum_{j=1}^p \vep_{tj}z_{t,-j}^T\eta^0_{(j)}\sim {\rm subE}(\la_2)$ where
	\benr
	\la_2=48 \si^2\sum_{j=1}^p \|\eta_{(j)}^0\|_2\surd(1+\nu^2)=48 \si^2\xi_{2,1}\surd(1+\nu^2)\nn
	\eenr
	This completes the proof of Part (ii).
\end{proof}

\bc$\rule{3.5in}{0.1mm}$\ec

\begin{lem}\label{lem:optimalcross} Suppose Condition B holds and let  $\vep_{tj}$ be as defined in (\ref{def:epsilons}). Additionally, let $u_T,v_T$ be any non-negative sequences satisfying $0\le v_T\le u_T.$ Then for any $0<a<1,$ choosing $c_{a1}=4\cdotp 48c_{a2},$ with $c_{a2}\ge \surd{(1/a)},$ we have for $T\ge 2,$
	\benr
	\sup_{\substack{\tau\in\cG(u_T,v_T)\\\tau\ge \tau^0}}\frac{1}{T}\Big|\sum_{t=\lfloor T\tau^0\rfloor+1}^{\lfloor T\tau\rfloor}\sum_{j=1}^p\vep_{tj}z_{t,-j}^T\eta^0_{(j)}\Big|\le c_{a1} \si^2\xi_{2,1}\surd{(1+\nu^2)}\Big(\frac{u_T}{T}\Big)^{\frac{1}{2}},\nn
	\eenr	
	with probability at least $1-a.$
\end{lem}

\begin{proof}[Proof of Lemma \ref{lem:optimalcross}]
	First note that without loss of generality we can assume $u_T\ge (1/T).$ This is because when $u_T<(1/T),$ the set $\cG(u_T,0)$ contains only the singleton $\tau^0$ with a distinct value $\lfloor T\tau^0\rfloor.$ Consequently, the sum of interest is over indices $t$ in an empty set, and is thus trivially zero. Now, let $\z_t=\sum_{j=1}^p\vep_{tj}z_{t,-j}^T\eta^0_{(j)},$ then using Lemma \ref{lem:subez} we have that $\z_t\sim {\rm subE(\la)},$ where $\la=48\xi_{2,1}\surd{(1+\nu^2)}\si^2.$ Additionally, from part (iii) of Lemma \ref{lem:subez}, we have, ${\rm var(\z_t)}=E(\z_t)^2\le 16\la^2.$ Consider the set $\cG(u_T,v_T) \cap \{\tau\ge\tau^0\}$ and note that in this set, there are at most $Tu_T$ distinct values of $\lfloor T\tau\rfloor.$ Applying Kolmogorov's inequality (Theorem \ref{thm:kolmogorov}) with any $d>0,$
	\benr
	pr\Big(\sup_{\substack{\tau\in \cG(u_T,v_T)\\ \tau\ge\tau^0}}\Big|\sum_{t=\lfloor T\tau^0\rfloor+1}^{\lfloor T\tau \rfloor} \z_{t}\Big|> d\Big)\le
	\frac{1}{d^2} \sum_{\substack{t \in \cG(u_T,v_T)\\ t\ge\tau^0}} {\rm var (z_t)}\le \frac{16Tu_T\la^2}{d^2}\nn \eenr
	Choosing $d=4c_{a2}\la\surd{(Tu_T)},$ with $c_{a2}\ge \surd{(1/a)}$ yields the lemma. 	
\end{proof}

\bc$\rule{3.5in}{0.1mm}$\ec

\begin{lem}\label{lem:nearoptimalcross} Suppose Condition B holds and let  $\vep_{tj}$ be as defined in (\ref{def:epsilons}) and let $0\le v_T\le u_T$ be any non-negative sequences. Then for any $c_{u2}>3$ and $c_{u1}\ge 96c_{u2},$ we have for $T\ge 2,$
	\benr
	&(i)&\,\, \sup_{\substack{\tau\in\cG(u_T,v_T)\\\tau\ge \tau^0}}\frac{1}{T}\Big\|\sum_{t=\lfloor T\tau^0\rfloor+1}^{\lfloor T\tau\rfloor}\vep_{tj}z_{t,-j}^T\Big\|_{\iny}\le c_{u1}\si^2 \surd{(1+\nu^2)}\Big(\frac{u_T}{T}\Big)^{\frac{1}{2}}\log (p\vee T),\nn\\
	&(ii)&\,\,\sup_{\substack{\tau\in\cG(u_T,v_T)\\\tau\ge \tau^0}}\frac{1}{T}\Big|\sum_{t=\lfloor T\tau^0\rfloor+1}^{\lfloor T\tau\rfloor}\sum_{j=1}^p\vep_{tj}z_{t,-j}^T\big(\h\eta_{(j)}-\eta^0_{(j)}\big)\Big|\le\nn\\ &&\hspace{2cm}c_{u1}\si^2 \surd{(1+\nu^2)}\Big(\frac{u_T}{T}\Big)^{\frac{1}{2}}\log (p\vee T)\sum_{j=1}^p\|\h\eta_{(j)}-\eta^0_{(j)}\|_1,\nn
	\eenr
	each with probability at least $1-2\exp\big\{-(c_{u2}-3)\log(p\vee T)\big\}.$\footnote{Here $\Big\|\sum\vep_{tj}z_{t,-j}^T\Big\|_{\iny}= \max_{j,k}\Big|\sum\vep_{tj}z_{t,-j,k}\Big|.$}
\end{lem}

\begin{proof}[Proof of Lemma \ref{lem:nearoptimalcross}] Part (ii) is a direct consequence of Part (i), thus we only prove Part (i). Without loss of generality, we can assume $v_T\ge (1/T).$ This follows since the only additional distinct element $\lfloor T\tau\rfloor$ in the set $\cG(u_T,0)$ in comparison to $\cG(u_T,(1/T))$ is $\lfloor T\tau^0\rfloor,$ and at this value, the sum of interest is over indices $t$ in an empty set and is thus trivially zero.
	
	Let $z_{t,-j}=(z_{t,-j,1},....,z_{t,-j,p-1})^T,$ then from Lemma \ref{lem:subez} we have $\vep_{tj}z_{t,-j,k}\sim {\rm subE}(\la_1),$ with $\la_1= 48\surd{(1+\nu^2)}\si^2.$ Now applying Bernstein's inequality (Lemma \ref{lem:subetail}) for any fixed $\tau\in\cG(u_T,v_T)$ satisfying $\tau\ge\tau^0,$ we have for any $d>0,$
	\benr\label{eq:3}
	pr\Big(\Big|\sum_{t=\lfloor T\tau^0\rfloor+1}^{\lfloor T\tau\rfloor}\vep_{tj}z_{t,-j,k}\Big|>d\big(\lfloor T\tau\rfloor-\lfloor T\tau^0\rfloor\big)\Big)\le 2\exp\Big\{-\frac{\big(\lfloor T\tau\rfloor-\lfloor T\tau^0\rfloor\big)}{2}\Big(\frac{d^2}{\la_1^2}\wedge \frac{d}{\la_1}\Big)\Big\}\nn
	\eenr	
	Choose $d=2c_{u2}\la_1\log (p\vee T)\big/\surd\big(\lfloor T\tau\rfloor-\lfloor T\tau^0\rfloor\big),$ then,
	\benr
	\big(\lfloor T\tau\rfloor-\lfloor T\tau^0\rfloor\big)\frac{d^2}{2\la_1^2}&=&2c_{u2}^2\log^2 (p\vee T),\quad{\rm and},\nn\\
	\big(\lfloor T\tau\rfloor-\lfloor T\tau^0\rfloor\big)\frac{d}{2\la_1}&=&c_{u2}\log(p\vee T),\nn
	\eenr
	where we have used $\big(\lfloor T\tau\rfloor-\lfloor T\tau^0\rfloor\big)\ge Tv_T\ge 1.$ A substitution back in the probability bound yields,
	\benr
	\Big|\sum_{t=\lfloor T\tau^0\rfloor+1}^{\lfloor T\tau\rfloor}\vep_{tj}z_{t,-j,k}\Big|\le 2c_{u2}\la_1\log(p\vee T)\big(\lfloor T\tau\rfloor-\lfloor T\tau^0\rfloor\big)^{1/2}\le 2c_{u2}\la_1\log (p\vee T) \big(Tu_T\big)^{\frac{1}{2}},\nn
	\eenr
	with probability at least $1-2\exp\{-c_{u2} \log (p\vee T)\}.$ Finally applying a union bound over $j=1,...,p,$ $k=1,...,p-1$ and over the at most $T$ distinct values of $\lfloor T\tau\rfloor$ for $\tau\in\cG(u_T,v_T),$ we obtain,
	\benr
	\sup_{\substack{\tau\in\cG(u_T,v_T)\\ \tau\ge \tau^0}}\Big\|\frac{1}{T}\sum_{t=\lfloor T\tau^0\rfloor+1}^{\lfloor T\tau\rfloor}\vep_{tj}z_{t,-j,k}\Big\|_{\iny}\le 2c_{u2}\la_1 \log (p\vee T)\Big(\frac{u_T}{T}\Big)^{\frac{1}{2}},\nn
	\eenr
	with probability at least $1-2\exp\{-(c_{u2}-3)\log (p\vee T)\}.$ This completes the proof of Part (i).
\end{proof}

\bc$\rule{3.5in}{0.1mm}$\ec

\begin{lem}\label{lem:optimalsqterm} Suppose Condition B holds and let $u_T,$ $v_T$ be any non-negative sequences satisfying $0\le v_T\le u_T.$ Then for any $0<a<1,$ choosing $c_{a1}=64c_{a2},$ with $c_{a2}\ge \surd{(1/a)},$ we have for $T\ge 2,$
	\benr
	(i)\,\,\inf_{\substack{\tau\in\cG(u_T,v_T);\\ \tau\ge \tau^0}}\frac{1}{T}\sum_{t=\lfloor T\tau^0\rfloor+1}^{\lfloor T\tau\rfloor} \sum_{j=1}^p \eta^{0T}_{(j)}z_{t,-j}z_{t,-j}^T\eta^{0}_{(j)}\ge v_T\ka\xi_{2,2}^2- c_{a1} \si^2\xi_{2,2}^2\Big(\frac{u_T}{T}\Big)^{\frac{1}{2}},\nn\\
	(ii)\,\,\sup_{\substack{\tau\in\cG(u_T,v_T);\\ \tau\ge \tau^0}}\frac{1}{T}\sum_{t=\lfloor T\tau^0\rfloor+1}^{\lfloor T\tau\rfloor} \sum_{j=1}^p \eta^{0T}_{(j)}z_{t,-j}z_{t,-j}^T\eta^{0}_{(j)}\le u_T\phi\xi_{2,2}^2+ c_{a1} \si^2\xi_{2,2}^2\Big(\frac{u_T}{T}\Big)^{\frac{1}{2}}\nn
	\eenr
	with probability at least $1-a.$
\end{lem}

\begin{proof}[Proof of Lemma \ref{lem:optimalsqterm}] As before in Lemma \ref{lem:optimalcross}, w.l.o.g we assume $u_T\ge (1/T).$ Now, we have $\eta^{0T}_{(j)}z_{t,-j}\sim {\rm subG}\big(\si\|\eta^0_{(j)}\|_2\big),$ consequently, using Lemma \ref{lem:sqsubGsubE} and Lemma \ref{lem:lcsubE} we have
	\benr
	\sum_{j=1}^p\Big(\|\eta^{0T}_{(j)}z_{t,-j}\|_2^2-E\|\eta^{0T}_{(j)}z_{t,-j}\|_2^2\Big)\sim {\rm subE}\big(\la\big),\quad {\rm with}\quad \la=16\si^2\xi_{2,2}^2.\nn
	\eenr
	Using moment properties of sub-exponential distributions \big(Part (iii) of Lemma \ref{lem:subez}\big) we also have that
	\benr
	{\rm var}\Big\{\sum_{j=1}^p \Big(\|\eta^{0T}_{(j)}z_{t,-j}\|_2^2-E\|\eta^{0T}_{(j)}z_{t,-j}\|_2^2\Big)\Big\}\le 16\la^2.\nn
	\eenr
	Now applying Kolmogorov's inequality (Lemma \ref{thm:kolmogorov}) we obtain,
	\benr
	pr\left\{\sup_{\substack{\tau\in\cG(u_T,v_T);\\ \tau\ge\tau^0}}\Big|\sum_{t=\lfloor T\tau^0\rfloor+1}^{\lfloor T\tau\rfloor}\sum_{j=1}^p\Big(\|\eta^{0T}_{(j)}z_{t,-j}\|_2^2-E\|\eta^{0T}_{(j)}z_{t,-j}\|_2^2\Big)\Big|>d \right\}\le \frac{16\la^2Tu_T}{d^2}.\nn
	\eenr
	Choosing $d=4c_{a2}\la\surd{(Tu_T)},$ with $c_{a2}\ge \surd(1/a)$ yields,
	\benr
	\sup_{\substack{\tau\in\cG(u_T,v_T);\\ \tau\ge\tau^0}}\frac{1}{T}\Big|\sum_{t=\lfloor T\tau^0\rfloor+1}^{\lfloor T\tau\rfloor}\sum_{j=1}^p\Big(\|\eta^{0T}_{(j)}z_{t,-j}\|_2^2-E\|\eta^{0T}_{(j)}z_{t,-j}\|_2^2\Big)\Big|\le 4c_{a2}\la\Big(\frac{u_T}{T}\Big)^{\frac{1}{2}}\nn
	\eenr
	with probability at least $1-a.$ The statement of this lemma is now a direct consequence.
\end{proof}

\bc$\rule{3.5in}{0.1mm}$\ec

We require additional notation for the following results. Consider any sequence of $\alpha_{(j)},\psi_{(j)}\in\R^{p-1},$ $j=1,...,p,$ and let $\al,$ $\psi$ represent the concatenation of all $\al_{(j)}$'s and $\psi_{(j)}$'s. Then define
\benr\label{def:Phi}
\Phi(\al,\psi)=\frac{1}{T}\sum_{t=\lfloor T\tau^0\rfloor+1}^{\lfloor T\tau\rfloor}\sum_{j=1}^p\al_{(j)}^Tz_{t,-j}z_{t,-j}\psi_{(j)}
\eenr

\begin{lem}\label{lem:etabounds} Let $\Phi(\cdotp,\cdotp)$ be as defined in (\ref{def:Phi}) and suppose Condition B and C(ii) hold. Let $u_T, v_T$ be any non-negative sequences satisfying $0\le v_T\le u_T.$ Then for any $0<a<1,$ choosing $c_{a1}=64c_{a2},$ with $c_{a2}\ge\surd{(1/a)},$ we have for $T\ge 2,$
	\benr
	&(i)&\inf_{\substack{\tau\in\cG(u_T,v_T);\\\tau\ge\tau^0}}\Phi(\eta^0,\eta^0)\ge v_T\ka\xi_{2,2}^2-c_{a1}\si^2\xi_{2,2}^2\Big(\frac{u_T}{T}\Big)^{\frac{1}{2}}\,\,\nn\\
	&(ii)&\sup_{\substack{\tau\in\cG(u_T,v_T);\\\tau\ge\tau^0}}\Phi(\h\eta-\eta^0,\h\eta-\eta^0)\le  c_u(\si^2\vee \phi)s\log (p\vee T)u_T\sum_{j=1}^p\|\h\eta_{(j)}-\eta^0_{(j)}\|_2^2\nn
	\eenr
	with probability at least $1-a,$ and $1-o(1),$ respectively. Moreover, when $u_T\ge c_{a1}^2\si^4\big/T\phi^2,$ we have,
	\benr
	&(iii)&\sup_{\substack{\tau\in\cG(u_T,v_T);\\\tau\ge\tau^0}}\Phi(\eta^0,\eta^0)\le 2u_T\phi\xi_{2,2}^2,\nn\\ &(iv)&\sup_{\substack{\tau\in\cG(u_T,v_T);\\\tau\ge\tau^0}}\big|\Phi(\h\eta-\eta^0,\eta^0)\big|\le  c_u(\si^2\vee \phi) u_T \xi_{2,2}\Big\{s\log(p\vee T)\sum_{j=1}^p\|\h\eta_{(j)}-\eta^0_{(j)}\|_2^2\Big\}^{\frac{1}{2}},\nn
	\eenr
	with probability at least $1-a,$ and $1-a-o(1),$ respectively.
\end{lem}

\begin{proof}[Proof of Lemma \ref{lem:etabounds}]
	Part (i) and Part (iii) of this lemma are a direct consequence of Lemma \ref{lem:optimalsqterm}. To prove Part (ii), first note that,
	\benr\label{eq:9}
	\sum_{j=1}^p\|\h\eta_{(j)}-\eta^0_{(j)}\|_1^2&\le& 2\sum_{j=1}^p\Big(\|\h\mu_{(j)}-\mu^0_{(j)}\|_1^2+\|\h\g_{(j)}-\g^0_{(j)}\|_1^2\Big)\nn\\
	&\le& 32s\sum_{j=1}^p\Big(\|\h\mu_{(j)}-\mu^0_{(j)}\|_2^2+\|\h\g_{(j)}-\g^0_{(j)}\|_2^2\Big)\nn\\
	&\le& 32s\sum_{j=1}^p\|\h\eta_{(j)}-\eta^0_{(j)}\|_2^2,
	\eenr	
	with probability at least $1-\pi_T=1-o(1).$ Here the second inequality follows since by Condition C(ii) we have, $\h\mu_{(j)}-\mu^0_{(j)}\in\cA_{1j},$ and $\h\g_{(j)}-\g^0_{(j)}\in\cA_{2j},$ $j=1,...,p.$ Now applying Lemma \ref{lem:WUURE}, we have,
	\benr
	\sup_{\tau\in\cG(u_T,v_T)} \Phi(\h\eta-\eta^0,\h\eta-\eta^0)\hspace{2.6in}\nn\\
	\le c_u(\si^2\vee \phi)\log(p\vee T)u_T\Big(\sum_{j=1}^p\|\h\eta_{(j)}-\eta^0\|_2^2+\sum_{j=1}^{p}\|\h\eta_{(j)}-\eta^0\|_1^2\Big)\hspace{-1cm}\nn\\
	\le c_u(\si^2\vee \phi)s\log (p\vee T)u_T\sum_{j=1}^p\|\h\eta_{(j)}-\eta^0_{(j)}\|_2^2\hspace{2.1cm}\nn
	\eenr
	with probability at least $1-o(1).$ Here the final inequality follows by using (\ref{eq:9}). The proof of Part (iv) is an application of the Cauchy-Schwartz inequality together with the bounds of Part (ii) and Part (iii),
	\benr\label{eq:5}
	\sup_{\substack{\tau\in\cG(u_T,v_T);\\\tau\ge\tau^0}}\big|\Phi(\h\eta-\eta^0,\eta^0)\big|&\le& 	\sup_{\substack{\tau\in\cG(u_T,v_T);\\\tau\ge\tau^0}} \big\{\Phi\big(\h\eta-\eta^0,\h\eta-\eta^0\big)\big\}^{\frac{1}{2}}\sup_{\substack{\tau\in\cG(u_T,v_T);\\\tau\ge\tau^0}} \big\{\Phi\big(\eta^0,\eta^0\big)\big\}^{\frac{1}{2}}.\nn
	\eenr
	This completes the proof of this lemma.
\end{proof}	

\bc$\rule{3.5in}{0.1mm}$\ec

\begin{lem}\label{lem:term123} Suppose Condition B and C(ii) hold. Let $u_T, v_T$ be any non-negative sequences satisfying $0\le v_T\le u_T.$ Then for any $0<a<1,$ choosing
	$c_{a1}=4\cdotp 48c_{a2},$ with $c_{a2}\ge \surd{(1/a)},$ and for $u_T\ge c_{a1}^2\si^4\big/(T\phi^2),$ we have for $T\ge 2,$
	\benr
	(i)\,\,\inf_{\substack{\tau\in\cG(u_T,v_T);\\\tau\ge\tau^0}}\frac{1}{T}\sum_{t=\lfloor T\tau^0\rfloor+1}^{\lfloor T\tau\rfloor}\sum_{j=1}^p \big(\h\eta_{(j)}^Tz_{t,-j}\big)^2\ge\hspace{3.2in}\nn\\ \ka\xi_{2,2}^2\Big[v_T-\frac{c_{a1}\si^2}{\ka}\Big(\frac{u_T}{T}\Big)^{\frac{1}{2}}- c_u(\si^2\vee \phi)\frac{u_T}{\ka\xi_{2,2}}\Big\{s\log(p\vee T)\sum_{j=1}^p\|\h\eta_{(j)}-\eta^0_{(j)}\|_2^2\Big\}^{\frac{1}{2}}\Big]\nn\\
	(ii)\,\,\sup_{\substack{\tau\in\cG(u_T,v_T);\\\tau\ge\tau^0}}\frac{1}{T}\Big|\sum_{t=\lfloor T\tau^0\rfloor+1}^{\lfloor T\tau\rfloor}\sum_{j=1}^p(\h\g_{(j)}-\g^0_{(j)})^Tz_{t,-j}z_{t,-j}^T\h\eta_{(j)}\Big|\le\hspace{2.2in}\nn\\
	c_u(\si^2\vee \phi)\xi_{2,2}u_T\Big\{s\log(p\vee T)\sum_{j=1}^p\|\h\g_{(j)}-\g^0_{(j)}\|_2^2\Big\}^{\frac{1}{2}}\Big[1+
	\frac{1}{\xi_{2,2}}\Big\{s\log(p\vee T)\sum_{j=1}^p\|\h\eta_{(j)}-\eta^0_{(j)}\|_2^2\Big\}^{\frac{1}{2}}\Big]\nn\\
	(iii)\sup_{\substack{\tau\in\cG(u_T,v_T);\\\tau\ge\tau^0}}\frac{1}{T}\Big|\sum_{t=\lfloor T\tau^0\rfloor+1}^{\lfloor T\tau\rfloor}\sum_{j=1}^p\vep_{tj}z_{t,-j}^T\h\eta_{(j)}\Big|\le\hspace{3.05in}\nn\\
	c_{a1}\surd(1+\nu^2)\si^2\xi_{2,1}\Big(\frac{u_T}{T}\Big)^{\frac{1}{2}}+ c_{u}\surd(1+\nu^2)\si^2\Big(\frac{u_T}{T}\Big)^{\frac{1}{2}}\log (p\vee T)\sum_{j=1}^p\|\h\eta_{(j)}-\eta^0_{(j)}\|_1,\nn
	\eenr
	each with probability at least $1-a-o(1).$
\end{lem}

\begin{proof}[Proof of Lemma \ref{lem:term123}]
	Let $\Phi(\cdotp,\cdotp)$ be as defined in (\ref{def:Phi}). Then note that $\Phi(\h\eta,\h\eta)=\Phi(\eta^0,\eta^0)+2\Phi(\h\eta-\eta^0,\eta^0)+\Phi(\h\eta-\eta^0,\h\eta-\eta^0).$ Using this relation together with the bounds of Part (i) and Part (iv) of Lemma \ref{lem:etabounds} we obtain,
	\benr
	\inf_{\substack{\tau\in\cG(u_T,v_T);\\\tau\ge\tau^0}}\Phi(\h\eta,\h\eta)&\ge& \inf_{\substack{\tau\in\cG(u_T,v_T);\\\tau\ge\tau^0}}\Phi(\eta^0,\eta^0)-2\sup_{\substack{\tau\in\cG(u_T,v_T);\\\tau\ge\tau^0}}|\Phi(\h\eta-\eta^0,\eta^0)|\nn\\
	&\ge& v_T\ka\xi_{2,2}^2- c_{a1}\si^2\xi_{2,2}^2\Big(\frac{u_T}{T}\Big)^{\frac{1}{2}}\nn\\
	&&- c_u(\si^2\vee\phi)u_T \xi_{2,2}\Big(s\log(p\vee T)\sum_{j=1}^p\|\h\eta_{(j)}-\eta^0_{(j)}\|_2^2\Big)^{\frac{1}{2}}\nn
	\eenr
	with probability at least $1-a-o(1).$ To prove Part (ii), note that using identical arguments as in the proof of Lemma \ref{lem:etabounds} it can be shown that,
	\benr
	&&\sup_{\substack{\tau\in\cG(u_T,v_T);\\\tau\ge\tau^0}}\Phi(\h\g-\g^0,\h\g-\g^0)\le   c_u(\si^2\vee \phi) s\log(p\vee T)u_T\sum_{j=1}^p\|\h\g_{(j)}-\g^0_{(j)}\|_2^2,\nn\\
	&&\sup_{\substack{\tau\in\cG(u_T,v_T);\\\tau\ge\tau^0}}\big|\Phi(\h\g-\g^0,\eta^0)\big|\le c_u(\si^2\vee \phi) u_T \xi_{2,2}\Big\{s\log(\vee T)\sum_{j=1}^p\|\h\g_{(j)}-\g^0_{(j)}\|_2^2\Big\}^{\frac{1}{2}},\nn
	\eenr	
	with probability at least $1-a-o(1).$ The above inequalities and the relation $\Phi\big(\h\g-\g^0,\h\eta\big)\le \big|\Phi(\h\g-\g^0,\h\eta-\eta^0)\big|+\big|\Phi(\h\g-\g^0,\eta^0)\big|,$ together with applications of the Cauchy-Schwartz inequality yields,
	\benr
	\sup_{\substack{\tau\in\cG(u_T,v_T);\\\tau\ge\tau^0}} \big|\Phi\big(\h\g-\g^0,\h\eta\big)\big|\hspace{3in}\nn\\
	\le c_u(\si^2\vee \phi) s\log(p\vee T)u_T\Big(\sum_{j=1}^p\|\h\g_{(j)}-\g^0_{(j)}\|_2^2\Big)^{\frac{1}{2}}\Big(\sum_{j=1}^p\|\h\eta_{(j)}-\eta^0_{(j)}\|_2^2\Big)^{\frac{1}{2}}\hspace{-1cm}\nn\\
	+ c_u(\si^2\vee \phi) u_T \xi_{2,2}\Big\{s\log(p\vee T)\sum_{j=1}^p\|\h\g_{(j)}-\g^0_{(j)}\|_2^2\Big\}^{\frac{1}{2}}\hspace{0.5cm}\nn\\
	\le  c_u(\si^2\vee \phi)\xi_{2,2}u_T\Big\{s\log(p\vee T)\sum_{j=1}^p\|\h\g_{(j)}-\g^0_{(j)}\|_2^2\Big\}^{\frac{1}{2}}\hspace{1.85cm}\nn\\
	\cdotp\Big[1+	\frac{1}{\xi_{2,2}}\Big\{s\log(p\vee T)\sum_{j=1}^p\|\h\eta_{(j)}-\eta^0_{(j)}\|_2^2\Big\}^{\frac{1}{2}}\Big]\hspace{1.5cm}\nn
	\eenr	
	with probability at least $1-a-o(1).$ To prove Part (iii), note that,
	\benr
	\sup_{\substack{\tau\in\cG(u_T,v_T);\\\tau\ge\tau^0}}\frac{1}{T}\Big|\sum_{t=\lfloor T\tau^0\rfloor+1}^{\lfloor T\tau\rfloor}\sum_{j=1}^p\vep_{tj}z_{t,-j}^T\h\eta_{(j)}\Big|&\le &\sup_{\substack{\tau\in\cG(u_T,v_T);\\\tau\ge\tau^0}}\frac{1}{T}\Big|\sum_{t=\lfloor T\tau^0\rfloor+1}^{\lfloor T\tau\rfloor}\sum_{j=1}^p\vep_{tj}z_{t,-j}^T\eta^0_{(j)}\Big|\nn\\
	&&+\sup_{\substack{\tau\in\cG(u_T,v_T);\\\tau\ge\tau^0}}\frac{1}{T}\Big|\sum_{t=\lfloor T\tau^0\rfloor+1}^{\lfloor T\tau\rfloor}\sum_{j=1}^p\vep_{tj}z_{t,-j}^T(\h\eta_{(j)}-\eta^0_{(j)})\Big|\nn\\
	&:=& R1+R2.\nn
	\eenr 	
	Now using Lemma \ref{lem:optimalcross} we have for any $0<a<1,$ $R1\le c_{a1}\surd(1+\nu^2)\si^2\xi_{2,1}\big(u_T\big/T\big)^{1/2},$ with probability at least $1-a.$ Also, using Lemma \ref{lem:nearoptimalcross} we have,
	\benr
	R2\le c_{u}\surd(1+\nu^2)\si^2\Big(\frac{u_T}{T}\Big)^{\frac{1}{2}}\log (p\vee T)\sum_{j=1}^p\|\h\eta_{(j)}-\eta^0_{(j)}\|_1\nn
	\eenr
	with probability at least $1-o(1).$ Part (iv) now follows by combining bounds for terms $R1$ and $R2.$
\end{proof}

\bc$\rule{3.5in}{0.1mm}$\ec

\begin{lem}\label{lem:assumptionbounds} Suppose Condition A and C hold. Then we have,
	\benr
	&(i)&\sum_{j=1}^p\|\h\eta_{(j)}-\eta^0_{(j)}\|_2^2\le c_u(1+\nu^2)\frac{\si^4}{\ka^2}\Big\{\frac{sp\log (p\vee T)}{Tl_T}\Big\},\nn\\
	&(ii)&\sum_{j=1}^p\|\h\eta_{(j)}-\eta^0_{(j)}\|_1\le c_u\surd(1+\nu^2)\frac{\si^2 sp}{\ka}\Big\{\frac{\log (p\vee T)}{Tl_T}\Big\}^{\frac{1}{2}}\nn\\
	&(iii)& \frac{1}{\xi_{2,2}} \Big(s\log(p\vee T)\sum_{j=1}^p\|\h\eta_{(j)}-\eta^0_{(j)}\|_2^2\Big)^{\frac{1}{2}}\le \frac{c_{u1}}{T^{b}}=o(1),\nn\\
	&(iv)&\frac{1}{\xi_{2,2}^2}\sum_{j=1}^p\|\h\eta_{(j)}-\eta^0_{(j)}\|_1\le \frac{c_{u1}}{\psi}\Big\{\frac{1}{\log (p\vee T)}\Big\}^{\frac{1}{2}} \nn
	\eenr
	with probability at least $1-o(1).$
\end{lem}

\begin{proof}[Proof of Lemma \ref{lem:assumptionbounds}] Part (i) can be obtained as,
	\benr
	\sum_{j=1}^p\|\h\eta_{(j)}-\eta^0_{(j)}\|_2^2&\le& 2\sum_{j=1}^p\Big(\|\h\mu_{(j)}-\mu^0_{(j)}\|_2^2+\|\h\g_{(j)}-\g^0_{(j)}\|_2^2\Big)\nn\\
	&\le& c_u(1+\nu^2)\frac{\si^4}{\ka^2}\Big\{\frac{sp\log (p\vee T)}{Tl_T}\Big\},\nn
	\eenr
	with probability at least $1-o(1).$ Here the final inequality follows from (\ref{eq:optimalmeans}). Part (ii) can be obtained quite analogously. To prove Part (iii) note that from Condition A we have $(1\big/\xi_{2,2})= (1\big/\psi\surd{p})$ and consider,
	\benr
	\frac{1}{\xi_{2,2}} \Big(s\log(p\vee T)\sum_{j=1}^p\|\h\eta_{(j)}-\eta^0_{(j)}\|_2^2\Big)^{\frac{1}{2}}&\le& \frac{1}{\psi}\Big(sp^{-1}\log(p\vee T)\sum_{j=1}^p\|\h\eta_{(j)}-\eta^0_{(j)}\|_2^2\Big)^{\frac{1}{2}}\nn\\
	&\le& c_u\surd(1+\nu^2)\frac{\si^2}{\psi\ka}\Big\{\frac{s\log (p\vee T)}{\surd(Tl_T)}\Big\}\le\frac{c_{u1}}{T^{b}},\nn
	\eenr
	with probability at least $1-o(1).$ Here the second inequality follows by using the bound of Part (i) and the second follows from Condition A. To prove Part (iv) consider,
	\benr
	\frac{1}{\xi_{2,2}^2}\sum_{j=1}^p\|\h\eta_{(j)}-\eta^0_{(j)}\|_1&\le& c_u\surd(1+\nu^2)\frac{\si^2 s}{\psi^2\ka}\Big\{\frac{\log (p\vee T)}{Tl_T}\Big\}^{\frac{1}{2}}\nn\\
	&\le& \big\{\frac{1}{\psi\log(p\vee T)}\Big\}c_u\surd(1+\nu^2)\frac{\si^2}{\psi\ka}\Big\{\frac{s\log (p\vee T)}{\surd(Tl_T)}\Big\}\nn\\
	&\le& \frac{c_{u1}}{\psi}\Big\{\frac{1}{\log (p\vee T)}\Big\}^{\frac{1}{2}} \nn
	\eenr
	with probability at least $1-o(1).$ Here the first inequality follows by the assumption $(1\big/\xi_{2,2})=(1\big/\psi\surd{p})$ together with the bound in Part (ii). The final inequality follows from Condition A.
\end{proof}

\bc$\rule{3.5in}{0.1mm}$\ec

\begin{lem}\label{lem:est.known.cC.approx} Let $\cC(\tau,\mu,\g)$ be as defined in (\ref{def:cC}) and suppose Condition A, B and C hold. Additionally assume that the relation (\ref{eq:rateextra}) holds. Then, for any $c_u>0,$ we have,
	\benr
	\sup_{\tau\in\cG\big((c_uT^{-1}\psi^{-2}),0\big)} \big|\cC(\tau,\h\mu,\h\g)-\cC(\tau,\mu^0,\g^0)\big|=o_p(1)\nn
	\eenr
\end{lem}

\begin{proof}[Proof of Lemma \ref{lem:est.known.cC.approx}] For any $\tau\ge \tau^0,$ first define the following,
	\benr\label{eq:Rdef}
	R_1&=&p^{-1}\sum_{\lfloor T\tau^0\rfloor+1}^{\lfloor T\tau\rfloor}\sum_{j=1}^p\|z_{t,-j}^T\h\eta_{(j)}\|_2^2-2p^{-1}\sum_{\lfloor T\tau^0\rfloor+1}^{\lfloor T\tau\rfloor}\sum_{j=1}^p\vep_{tj}z_{t,-j}^T\h\eta_{(j)}\nn\\
	&&+2p^{-1}\sum_{\lfloor T\tau^0\rfloor+1}^{\lfloor T\tau\rfloor}\sum_{j=1}^p(\h\g_{(j)}-\g^0_{(j)})^Tz_{t,-j}z_{t,-j}^T\h\eta_{(j)}\nn\\
	&=&R_{11}-2R_{12}+2R_{13},\nn\\
	R_2&=&p^{-1}\sum_{\lfloor T\tau^0\rfloor+1}^{\lfloor T\tau\rfloor}\sum_{j=1}^p\|z_{t,-j}^T\eta^0_{(j)}\|_2^2-2p^{-1}\sum_{\lfloor T\tau^0\rfloor+1}^{\lfloor T\tau\rfloor}\sum_{j=1}^p\vep_{tj}z_{t,-j}^T\eta^0_{(j)}\nn\\
	&=&R_{21}-2R_{22}.
	\eenr
	Then we have the following algebraic expansion,
	\benr\label{eq:algebra}
	\big(\cC(\tau,\h\mu,\h\g)-\cC(\tau,\mu^0,\g^0)\big)&=&-Tp^{-1}\Big(Q(z,\tau,\h\mu,\h\g)-Q(z,\tau^0,\h\mu,\h\g)\Big)\nn\\
	&&+Tp^{-1}\Big(Q(z,\tau,\mu^0,\g^0)-Q(z,\tau^0,\mu^0,\g^0)\Big)\nn\\
	&=&\big(R_{2}-R_{1}\big)\nn\\
	&=&\Big\{\big(R_{21}-2R_{22}\big)-\big(R_{11}-2R_{12}+2R_{13}\big)\Big\}.
	\eenr
	Lemma \ref{lem:limiting.dist.residual.terms} shows that the expressions $\big|R_{21}-R_{11}\big|,$  $\big|R_{22}-R_{12}\big|,$ and $|R_{13}|$ are $o_p(1)$ uniformly over the set $\{\cG\big(c_1T^{-1}\psi^{-2},0\big)\}\cap\{\tau\ge \tau^0\}.$ The same result can be obtained symmetrically on the set $\{\cG\big(c_uT^{-1}\psi^{-2},0\big)\}\cap\{\tau\le \tau^0\},$ thereby yielding $o_p(1)$ bounds for these terms uniformly over $\cG\big(c_uT^{-1}\psi^{-2},0\big)$ Consequently,
	\benr
	\sup_{\tau\in\cG\big((c_1T^{-1}\psi^{-2}),0\big)} \big|\cC(\tau,\h\mu,\h\g)-\cC(\tau,\mu^0,\g^0)\big|&\le&
	\sup_{\tau\in\cG\big((c_1T^{-1}\psi^{-2}),0\big)}|R_{21}-R_{11}|\nn\\
	&&+\sup_{\tau\in\cG\big((c_1T^{-1}\psi^{-2}),0\big)}2|R_{22}-R_{12}|\nn\\
	&&+\sup_{\tau\in\cG\big((c_1T^{-1}\psi^{-2}),0\big)}2|R_{13}|\nn\\
	&=&o_p(1)\nn
	\eenr
	This completes the proof of this lemma.
\end{proof}

\bc$\rule{3.5in}{0.1mm}$\ec

\begin{lem}\label{lem:limiting.dist.residual.terms} Suppose Condition A, B and C hold and additionally assume that relation (\ref{eq:rateextra}) holds. Let $R_{11},R_{12},R_{13},$ and $R_{21},R_{22}$ be as defined in (\ref{eq:Rdef}). Let $0<c_u<\iny$ be any constant, then we have the following bounds.
	\benr
	&(i)&\sup_{\substack{\tau\in\cG\big((c_uT^{-1}\psi^{-2}),0\big);\\\tau\ge\tau^0}}|R_{11}-R_{21}|=o(1)\nn\\	&(ii)&\,\,\sup_{\substack{\tau\in\cG\big((c_uT^{-1}\psi^{-2}),0\big);\\\tau\ge\tau^0}}|R_{12}-R_{22}|=o(1)\nn\\
	&(iii)&\sup_{\substack{\tau\in\cG\big((c_uT^{-1}\psi^{-2}),0\big);\\\tau\ge\tau^0}}|R_{13}|=o(1)\nn
	\eenr
	each with probability at least $1-o(1).$	
\end{lem}

\begin{proof}[Proof of Lemma \ref{lem:limiting.dist.residual.terms}] Let $\Phi(\cdotp,\cdotp)$ be as defined in (\ref{def:Phi}) and consider,
	\benr\label{eq:unifpart1}
	\sup_{\substack{\tau\in\cG\big((c_1T^{-1}\psi^{-2}),0\big);\\\tau\ge\tau^0}}|R_{11}-R_{21}|\hspace{2.75in}\nn\\
	=\sup_{\substack{\tau\in\cG\big((c_1T^{-1}\psi^{-2}),0\big);\\\tau\ge\tau^0}}p^{-1}\Big|\sum_{\lfloor T\tau^0\rfloor+1}^{\lfloor T\tau\rfloor}\sum_{j=1}^p\Big(\|z_{t,-j}^T\h\eta_{(j)}\|_2^2-\|z_{t,-j}^T\eta^0_{(j)}\|_2^2\Big)\Big|\nn\\
	=\sup_{\substack{\tau\in\cG\big((c_1T^{-1}\psi^{-2}),0\big);\\\tau\ge\tau^0}} p^{-1}\Big|\sum_{\lfloor T\tau^0\rfloor+1}^{\lfloor T\tau\rfloor}\sum_{j=1}^p(\h\eta_{(j)}-\eta^0_{(j)})^Tz_{t,-j}z_{t,-j}^T(\h\eta_{(j)}+\eta^0_{(j)})\Big|\hspace{-0.85cm}\nn\\
	=\sup_{\substack{\tau\in\cG\big((c_1T^{-1}\psi^{-2}),0\big);\\\tau\ge\tau^0}}\Big| Tp^{-1}\Phi(\h\eta-\eta^0,\h\eta-\eta^0)+2Tp^{-1}\Phi(\h\eta-\eta^0,\eta^0)\Big|.\hspace{-0.35cm}
	\eenr
	Now from Part (ii) of Lemma \ref{lem:etabounds} we have
	\benr\label{eq:15}
	\sup_{\substack{\tau\in\cG\big((c_1T^{-1}\psi^{-2}),0\big);\\\tau\ge\tau^0}} Tp^{-1}\Phi(\h\eta-\eta^0,\h\eta-\eta^0)\hspace{2in}\nn\\
	\le c_uc_1(\si^2\vee\phi)\psi^{-2}p^{-1}s\log(p\vee T)\sum_{j=1}^p\|\h\eta_{(j)}-\eta^0_{(j)}\|_2^2\nn\\
	= O\Big(\frac{s^2\log^2 (p\vee T)}{\psi^{-2}Tl_T}\Big)=o(1),\hspace{3.7cm}
	\eenr
	with probability at least $1-o(1).$ Also, from Part (iv) of Lemma \ref{lem:etabounds}, we have for $u_T\ge c_{a1}^2\si_x^4\big/(T\phi^2),$
	\benr\label{eq:18}
	\sup_{\substack{\tau\in\cG(u_T,0\big);\\\tau\ge\tau^0}} 2Tp^{-1}\big|\Phi(\h\eta-\eta^0,\eta^0)\big|\hspace{2.25in}\nn\\
	\le c_u(\si^2\vee\phi) Tu_Tp^{-1}\xi_{2,2}\Big(s\log(p\vee T)\sum_{j=1}^p\|\h\eta_{(j)}-\eta^0_{(j)}\|_2^2\Big)^{\frac{1}{2}}\hspace{-0.5cm}
	\eenr
	with probability at least $1-a-o(1).$ Upon choosing $a=\big(64^2 \psi^2\si^4\big)\big/(c_1\phi^2)\to 0,$ we have $c_1T^{-1}\psi^{-2}= c_{a1}^2\si_x^4\big/(T\phi^2),$ consequently from (\ref{eq:18}) we have,
	\benr\label{eq:19}
	\sup_{\substack{\tau\in\cG\big((c_1T^{-1}\psi^{-2}),0\big);\\\tau\ge\tau^0}} 2T\big|\Phi(\h\eta-\eta^0,\eta^0)\big|\hspace{2in}\nn\\
	\le c_uc_1(\si^2\vee\phi) \frac{\xi_{2,2}}{p\psi^2}
	\Big(s\log(p\vee T)\sum_{j=1}^p\|\h\eta_{(j)}-\eta^0_{(j)}\|_2^2\Big)^{\frac{1}{2}}\nn\\
	=c_uc_1(\si^2\vee\phi)\frac{1}{\xi_{2,2}}\Big(s\log(p\vee T)\sum_{j=1}^p\|\h\eta_{(j)}-\eta^0_{(j)}\|_2^2\Big)^{\frac{1}{2}}\nn\\
	\le O\Big(\frac{1}{\psi}\frac{s\log (p\vee T)}{\surd(Tl_T)}\Big)=o(1)\hspace{1.5in}
	\eenr
	with probability at least $1-a-o(1)=1-o(1).$ Substituting this uniform bound together with (\ref{eq:15}) back in (\ref{eq:unifpart1}) yields Part (i) of this lemma. To prove Part (ii), note that
	\benr
	\sup_{\substack{\tau\in\cG\big((c_1T^{-1}\psi^{-2}),0\big);\\\tau\ge\tau^0}}|R_{12}-R_{22}|&=&
	\sup_{\substack{\tau\in\cG\big((c_1T^{-1}\psi^{-2}),0\big);\\\tau\ge\tau^0}}p^{-1}\Big|\sum_{\lfloor T\tau^0\rfloor+1}^{\lfloor T\tau\rfloor}\sum_{j=1}^p \vep_{tj}z_{t,-j}^T(\h\eta_{(j)}-\eta^0_{(j)})\Big|\nn\\
	&=& O\Big(p^{-1}\psi^{-1}\log(p\vee T) \sum_{j=1}^p\|\h\eta_{(j)}-\eta^0_{(j)}\|_1\Big)\nn\\
	&\le& O\Big(\frac{s\log^{3/2}(p\vee T)}{\psi \surd(Tl_T)}\Big)=o(1),\nn
	\eenr
	with probability at least $1-o(1).$ Here the second equality follows from Part (ii) of Lemma \ref{lem:nearoptimalcross}. To prove Part (iii) we first note that the expressions $\Phi\big(\h\g-\g^0,\h\eta-\eta^0\big),$ and $\Phi(\h\g-\g^0,\eta^0)$ can be bounded above with probability at least $1-o(1),$ by the same bounds as in (\ref{eq:15}) and (\ref{eq:19}), respectively. Now applications of the Cauchy-Schwartz inequality yields the following bound for the term $|R_{13}|.$
	\benr
	\sup_{\substack{\tau\in\cG\big((c_1T^{-1}\psi^{-2}),0\big);\\\tau\ge\tau^0}}|R_{13}|&=&\sup_{\substack{\tau\in\cG\big((c_1T^{-1}\psi^{-2}),0\big);\\\tau\ge\tau^0}}\Big|\sum_{\lfloor T\tau^0\rfloor+1}^{\lfloor T\tau\rfloor}\sum_{j=1}^p (\h\g_{(j)}-\g^0_{(j)})^Tz_{t,-j}z_{t,-j}^T\h\eta_{(j)}\Big|\nn\\
	&\le& \sup_{\substack{\tau\in\cG\big((c_1T^{-1}\psi^{-2}),0\big);\\\tau\ge\tau^0}}T\Big\{\big|\Phi(\h\g-\g^0,\h\eta-\eta^0)\big|+\big|\Phi(\h\g-\g^0,\eta^0)\big|\Big\}=o(1),\nn
	\eenr
	with probability at least $1-o(1),$ thus completing the proof of the lemma.
\end{proof}

\bc$\rule{3.5in}{0.1mm}$\ec

\begin{lem}\label{lem:design.convergence.for.limiting.dist} Suppose Condition B holds and that $\psi\to 0.$ Then for any constant $r>0,$ we have,
	\benr
	p^{-1}\Big|\sum_{\lfloor T\tau^0\rfloor+1}^{\lfloor T\tau^0+r\psi^{-2}\rfloor}\sum_{j=1}^p\Big(\|z_{t,-j}^T\eta^0_{(j)}\|_2^2-E\|z_{t,-j}^T\eta^0_{(j)}\|_2^2\Big)\Big|=o_p(1) \nn
	\eenr
	Additionally, if $\xi^{-2}_{2,2}\sum_{j=1}^pE\|z_{t,-j}^T\eta^0_{(j)}\|_2^2\to \si^*,$ then,
	\benr
	p^{-1}\sum_{\lfloor T\tau^0\rfloor+1}^{\lfloor T\tau^0+r\psi^{-2}\rfloor}\sum_{j=1}^p\big\|z_{t,-j}^T\eta^0_{(j)}\big\|_2^2 \to_p r\si^*. \nn
	\eenr
\end{lem}

\begin{proof}[Proof of Lemma \ref{lem:design.convergence.for.limiting.dist}]
	We begin with the following observation. For any $\tau\ge \tau^0,$ we have the deterministic inequality  $T(\tau-\tau^0)-1\le \big(\lfloor T\tau\rfloor-\lfloor T\tau^0\rfloor \big)\le T(\tau-\tau^0)+1.$ It is straightforward to verify that under the assumption $\psi\to 0,$ this inequality directly yields $c_{u1}r\psi^{-2}\le\big(\lfloor T\tau^0+ r\xi^{-2}\rfloor-\lfloor T\tau^0\rfloor \big)\le c_{u2}r\psi^{-2}.$ Also, note that from Lemma \ref{lem:lcsubE} and Lemma \ref{lem:sqsubGsubE} we have,
	\benr
	\hspace{1.25cm} p^{-1}\psi^{-2}\sum_{j=1}^p\Big(\big\|z_{t,-j}^T\eta^0_{(j)}\big\|_2^2-E\big\|z_{t,-j}^T\eta^0_{(j)}\big\|_2^2\Big)\sim {\rm subE}(\la),\,\, \la=16\si^2.
	\eenr
	Now upon applying Bernstein's inequality (Lemma \ref{lem:subetail}) together with the above observations, we obtain for any $d>0,$
	\benr
	pr\Big\{p^{-1}\Big|\sum_{\lfloor T\tau^0\rfloor+1}^{\lfloor T\tau^0+r\psi^{-2}\rfloor}\sum_{j=1}^p\Big(\big\|z_{t,-j}^T\eta^0_{(j)}\big\|_2^2-E\big\|z_{t,-j}^T\eta^0_{(j)}\big\|_2^2\Big)\Big|> c_{u2}dr\Big\}\nn\\
	\le 2\exp\Big\{-\frac{c_{u1}r\psi^{-2}}{2}\Big(\frac{d^2}{\la^2}\wedge\frac{d}{\la}\Big)\Big\}\hspace{-1.25cm}.\nn
	\eenr
	Choosing $d$ as any sequence converging to zero slower than $\psi,$ say $d=\psi^{1-b},$ for any $0<b<1,$ and noting that in this case $(d\big/\la) \le 1$ for $T$ large, we obtain,
	\benr
	p^{-1}\Big|\sum_{\lfloor T\tau^0\rfloor+1}^{\lfloor T\tau^0+r\psi^{-2}\rfloor}\sum_{j=1}^p\Big(\big\|z_{t,-j}^T\eta^0_{(j)}\big\|_2^2-E\big\|z_{t,-j}^T\eta^0_{(j)}\big\|_2^2\Big)\Big|=o_p(1), \nn
	\eenr
	This completes the proof of the first part of this lemma, the second part can be obtained as a direct consequence of Part (i).
\end{proof}

\bc$\rule{3.5in}{0.1mm}$\ec

\section{Deviation bounds used for proofs in Section 3}

\begin{lem}\label{lem:bounds.for.nuis.thm} Suppose Condition A$'$(i), A$'$(ii) and B holds, and $c_{u1}>0$ be any constant. Then uniformly over $j=1,...,p,$ we have,
	\benr
	\sup_{\substack{\tau\in(0,1);\\ \lfloor T\tau\rfloor \ge c_{u1}Tl_T}}\frac{1}{\lfloor T\tau\rfloor}\Big\|\sum_{t=1}^{\lfloor T\tau\rfloor}\vep_{tj}z_{t,-j}\Big\|_{\iny}\le 48\si^2(c_u/\surd c_{u1})\surd(1+\nu^2)\Big\{\frac{\log(p\vee T)}{Tl_T}\Big\}^{\frac{1}{2}}\nn
	\eenr
	with probability at least $1-2\exp\big[-\{(c_u^2/2)-3\}\log(p\vee T)\big].$ Additionally, let $u_T\ge 0,$ be any sequence and $c_u>0$ any constant, then uniformly over $j=1,...,p,$ we have,
	\benr
	\sup_{\substack{\tau\in\cG(u_T,0);\\\lfloor T\tau\rfloor \ge c_{u1}Tl_T}}\frac{1}{\lfloor T\tau\rfloor} \Big\|\sum_{t=\lfloor T\tau^0\rfloor+1}^{\lfloor T\tau\rfloor}\eta^{0T}_{(j)}z_{t,-j}z_{t,-j}^T\Big\|_{\iny}&\le& c_{u2}(\si^2\vee\phi)\|\eta^0_{(j)}\|_2\max\Big\{\frac{\log (p\vee T)}{Tl_T},\,\,\frac{u_T}{l_T}\Big\},\nn
	\eenr
	with probability $1-2\exp\big\{-c_{u3}\log (p\vee T)\big\},$ with $c_{u2}=(1+48c_u)/c_{u1},$ $c_{u3}=\{(c_u\wedge c_u^2)/2\}-3.$
\end{lem}

\begin{proof}[Proof of Lemma \ref{lem:bounds.for.nuis.thm}]
	We begin with proving Part (i). Using Lemma \ref{lem:subez} we have that $\vep_{tj}z_{t,-j,k}\sim {\rm subE}(\la_1),$ with $\la_1=48\si^2\surd (1+\nu^2).$ For any $\tau\in(0,1)$ satisfying $\lfloor T\tau\rfloor\ge c_{u1}Tl_T,$ applying Lemma \ref{lem:subetail} we have for $d>0,$
	\benr
	pr\Big(\Big|\sum_{t=1}^{\lfloor T\tau\rfloor}\vep_{tj}z_{t,-j,k}\Big|>d\lfloor T\tau\rfloor\Big)\le 2\exp\Big\{-\frac{\lfloor T\tau\rfloor}{2}\Big(\frac{d^2}{\la_1^2}\wedge\frac{d}{\la_1}\Big)\Big\}.\nn
	\eenr	
	Choose $d=c_{u}\la_1\surd\{\log (p\vee T)\big/\lfloor T\tau\rfloor\},$ and recall that by choice we have $\lfloor T\tau\rfloor\ge c_{u1}Tl_T,$ and from Condition A$'$(i) we have $\log(p\vee T)\le c_{u1}Tl_T.$ Thus, $d/\la_1\le 1,$ and consequently $(d^2/\la_1^2)\le (d/\la_1).$ Using these relations the above probability bound yields,
	\benr
	\frac{1}{\lfloor T\tau\rfloor}\Big|\sum_{t=1}^{\lfloor T\tau\rfloor}\vep_{tj}z_{t,-j,k}\Big|\le (c_u/\surd c_{u1})\la_1 \Big\{\frac{\log(p\vee T)}{Tl_T}\Big\}^{\frac{1}{2}}\nn
	\eenr
	with probability at least $1-2\exp\big\{-(c_u^2/2)\log(p\vee T)\big\}.$ Part (i) now follows by applying a union bound over $k=1,...,(p-1),$ $j=1,...,p$ and over the at most $T$ distinct values of $\lfloor T\tau\rfloor.$
	
	To prove Part (ii), first note that using similar arguments as in Lemma \ref{lem:subez} we have that $\eta^{0T}_{(j)}z_{t,-j}z_{t,-j,k}-E\big(\eta^{0T}_{(j)}z_{t,-j}z_{t,-j,k}\big)\sim {\rm subE}(\la_1),$ with $\la_1= 48\si^2\|\eta^0_{(j)}\|_2.$ For any $\tau\in\cG(u_T,0),$ satisfying $\lfloor T\tau\rfloor\ge c_{u1}Tl_T,$ applying a union bound over $k=1,...,p-1,$ on the Bernstein's inequality (Lemma \ref{lem:subez}) yields the following probability bound,
	\benr\label{eq:22}
	pr\Big\{\Big\|\sum_{t=\lfloor T\tau^0\rfloor+1}^{\lfloor T\tau\rfloor}\big(\eta^{0T}_{(j)}z_{t,-j}z_{t,-j}-\eta^{0T}_{(j)}\D_{-j,-j}\big)\Big\|_{\iny}>d(\lfloor T\tau\rfloor-\lfloor T\tau^0\rfloor)\Big\}\nn\\
	\le 2p\exp\Big\{-\frac{(\lfloor T\tau\rfloor-\lfloor T\tau^0\rfloor)}{2}\Big(\frac{d^2}{\la_1^2}\wedge\frac{d}{\la_1}\Big)\Big\}\hspace{-1cm}
	\eenr
	Now upon choosing,
	\benr
	d=c_u\la_1\max\Big[\Big\{\frac{\log (p\vee T)}{(\lfloor T\tau\rfloor-\lfloor T\tau^0\rfloor)}\Big\}^{\frac{1}{2}},\,\frac{\log(p\vee T)}{(\lfloor T\tau\rfloor-\lfloor T\tau^0\rfloor)}\Big],\nn
	\eenr
	it can be verified that	\footnote{See, Remark \ref{rem:calculation}},
	\benr\label{eq:21}
	&&d\frac{(\lfloor T\tau\rfloor-\lfloor T\tau^0\rfloor)}{\lfloor T\tau\rfloor}\le \frac{c_u}{c_{u1}}\la_1\max\Big\{\frac{\log(p\vee T)}{Tl_T},\,\,\frac{u_T}{l_T}\Big\},\quad {\rm and},\nn\\
	&&\frac{(\lfloor T\tau\rfloor-\lfloor T\tau^0\rfloor)}{2}\Big(\frac{d^2}{\la_1^2}\wedge\frac{d}{\la_1}\Big)=\frac{(c_u\wedge c_u^2)}{2}\log(p\vee T)
	\eenr
	Substituting the relations of (\ref{eq:21}) in the probability bound (\ref{eq:22}) we obtain,
	\benr
	\frac{1}{\lfloor T\tau\rfloor}\Big\|\sum_{t=\lfloor T\tau^0\rfloor+1}^{\lfloor T\tau\rfloor}\big(\eta^{0T}_{(j)}z_{t,-j}z_{t,-j}-\eta^{0T}_{(j)}\D_{-j,-j}\big)\Big\|_{\iny}\le \frac{c_u}{c_{u1}}\la_1\max\Big\{\frac{\log(p\vee T)}{Tl_T},\,\,\frac{u_T}{l_T}\Big\} \nn
	\eenr
	with probability at least $1-2p\exp\big[\{(c_u\wedge c_u^2)/2\}\log (p\vee T)\big].$ Next, using the bounded eigenvalue assumption of Condition B we have that,
	\benr
	\frac{1}{\lfloor T\tau\rfloor}\sum_{t=\lfloor T\tau^0\rfloor+1}^{\lfloor T\tau\rfloor}\eta^{0T}_{(j)}\D_{-j,-j}\le \|\eta^0_{(j)}\|_2\phi\frac{u_T}{c_{u1}l_T}\nn
	\eenr
	Using this relation in the probability bound now yields,
	\benr
	\frac{1}{\lfloor T\tau\rfloor}\Big\|\sum_{t=\lfloor T\tau^0\rfloor+1}^{\lfloor T\tau\rfloor}\eta^{0T}_{(j)}z_{t,-j}z_{t,-j}\Big\|_{\iny}\le c_{u2}\phi\|\eta^0_{(j)}\|_2\frac{u_T}{l_T}+c_{u3}\si^2\|\eta^0_{(j)}\|_2\max\Big\{\frac{\log(p\vee T)}{Tl_T},\,\,\frac{u_T}{l_T}\Big\}, \nn
	\eenr
	with probability at least $1-2p\exp\big[\{(c_u\wedge c_u^2)/2\}\log (p\vee T)\big],$ where $c_{u2}=1/c_{u1},$ and $c_{u3}=48c_u/c_{u1}.$  Uniformity over $\tau$ can be obtained by a union bound over the at most $T$ values of $\lfloor T\tau\rfloor,$ and similarly over $j=1,...,p,$ by using another union bound. This completes the proof of the lemma.
\end{proof}

\bc$\rule{3.5in}{0.1mm}$\ec

\begin{rem}\label{rem:calculation} Consider,
	\benr\label{eq:24}
	d=c_u\la_1\max\Big[\Big\{\frac{\log (p\vee T)}{(\lfloor T\tau\rfloor-\lfloor T\tau^0\rfloor)}\Big\}^{\frac{1}{2}},\,\frac{\log(p\vee T)}{(\lfloor T\tau\rfloor-\lfloor T\tau^0\rfloor)}\Big],
	\eenr	
	observe that when $\log (p\vee T)\big/(\lfloor T\tau\rfloor-\lfloor T\tau^0\rfloor)\ge 1,$ then the maximum of the two terms in the expression (\ref{eq:24}) is $\log (p\vee T)\big/(\lfloor T\tau\rfloor-\lfloor T\tau^0\rfloor).$ In this case,
	\benr\label{eq:25}
	\Big(\frac{d^2}{\la_1^2}\wedge\frac{d}{\la_1}\Big) = (c_u^2\wedge c_u) \frac{\log(p\vee T)}{(\lfloor T\tau\rfloor-\lfloor T\tau^0\rfloor)}.
	\eenr
	In the case where $\log (p\vee T)\big/(\lfloor T\tau\rfloor-\lfloor T\tau^0\rfloor)< 1,$ the maximum in the expression $(\ref{eq:24})$ becomes $\surd\{\log (p\vee T)\big/(\lfloor T\tau\rfloor-\lfloor T\tau^0\rfloor)\},$ however the minimum in the expression (\ref{eq:25}) remains the same.
\end{rem}

\bc$\rule{3.5in}{0.1mm}$\ec

\begin{lem}\label{lem:nearoptimalcross.check} Suppose Condition B holds and let  $\vep_{tj}$ be as defined in (\ref{def:epsilons}). Let $T\ge\log(p\vee T)$ and $\log(p\vee T)\le Tv_T\le Tu_T$ be any non-negative sequences. Then for any $c_{u}>0,$ we have,
	\benr
	\sup_{\substack{\tau\in\cG(u_T,v_T)\\\tau\ge \tau^0}}\frac{1}{T}\Big\|\sum_{t=\lfloor T\tau^0\rfloor+1}^{\lfloor T\tau\rfloor}\vep_{tj}z_{t,-j}^T\Big\|_{\iny}\le 48\surd(2c_u)\si^2 \surd{(1+\nu^2)} \Big(\frac{u_T\log (p\vee T)}{T}\Big)^{\frac{1}{2}},\nn
	\eenr
	with probability at least $1-2\exp\big\{-(c_{u1}-3)\log(p\vee T)\big\},$ with $c_{u1}=c_u\wedge\surd(c_u/2).$
\end{lem}

\begin{proof}[Proof of Lemma \ref{lem:nearoptimalcross.check}] The proof of this result is very similar to that of Lemma \ref{lem:nearoptimalcross}, the difference being utilization of the additional assumption $Tv_T\ge \log(p\vee T),$ in order to obtain this sharper bound. Proceeding as in (\ref{eq:3}) we have,
	\benr
	pr\Big(\Big|\sum_{t=\lfloor T\tau^0\rfloor+1}^{\lfloor T\tau\rfloor}\vep_{tj}z_{t,-j,k}\Big|>d\big(\lfloor T\tau\rfloor-\lfloor T\tau^0\rfloor\big)\Big)\le 2\exp\Big\{-\frac{\big(\lfloor T\tau\rfloor-\lfloor T\tau^0\rfloor\big)}{2}\Big(\frac{d^2}{\la_1^2}\wedge \frac{d}{\la_1}\Big)\Big\},\nn
	\eenr	
	where $\la_1=48\si^2\surd(1+\nu^2).$ Choose $d=\la_1\{2c_{u}\log (p\vee T)\big/(\lfloor T\tau\rfloor-\lfloor T\tau^0\rfloor)\}^{1/2},$ then,
	\benr
	\big(\lfloor T\tau\rfloor-\lfloor T\tau^0\rfloor\big)\frac{d^2}{2\la_1^2}&=&c_{u}\log (p\vee T),\quad{\rm and},\nn\\
	\big(\lfloor T\tau\rfloor-\lfloor T\tau^0\rfloor\big)\frac{d}{2\la_1}&=&\surd(c_{u}/2)\{\log(p\vee T)(\lfloor T\tau\rfloor-\lfloor T\tau^0\rfloor)\}^{1/2}\ge\surd(c_{u}/2) \log(p\vee T) ,\nn
	\eenr
	where we used $\big(\lfloor T\tau\rfloor-\lfloor T\tau^0\rfloor\big)\ge Tv_T\ge \log(p\vee T).$ Substituting back in the probability bound yields,
	\benr
	\frac{1}{T}\Big|\sum_{t=\lfloor T\tau^0\rfloor+1}^{\lfloor T\tau\rfloor}\vep_{tj}z_{t,-j,k}\Big|\le \la_1\Big\{\frac{2c_u u_T\log(p\vee T)}{T}\Big\}^{1/2},\nn
	\eenr
	with probability $1-2\exp\{-c_{u1} \log (p\vee T)\},$ with $c_{u1}=c_u\wedge \surd(c_u/2).$ Finally applying a union bound over $j=1,...,p,$ $k=1,...,p-1$ and over at most $T$ values of $\lfloor T\tau\rfloor$ for $\tau\in\cG(u_T,v_T),$ yields the statement of the lemma.
\end{proof}

\bc$\rule{3.5in}{0.1mm}$\ec

\begin{lem}\label{lem:nearoptimal.sqterm} Let $\Phi(\cdotp,\cdotp)$ be as defined in (\ref{def:Phi}) and suppose Condition B holds and $T\ge \log(p\vee T).$ Additionally, let $u_T,$ $v_T$ be non-negative sequences satisfying $\log (p\vee T)\le Tv_T\le Tu_T.$ Then for any constant $c_u>0,$ we have,
	\benr
	&(i)&\inf_{\substack{\tau\in\cG(u_T,v_T);\\ \tau\ge\tau^0}}\Phi(\eta^0,\eta^0)\ge v_T\ka\xi_{2,2}^2-16\surd(2c_u)\si^2\xi_{2,2}^2\Big(\frac{u_T\log(p\vee T)}{T}\Big)^{\frac{1}{2}},\nn\\
	&(ii)&\sup_{\substack{\tau\in\cG(u_T,v_T);\\ \tau\ge\tau^0}}\Phi(\eta^0,\eta^0)\le u_T\phi\xi_{2,2}^2+16\surd(2c_u)\si^2\xi_{2,2}^2\Big(\frac{u_T\log(p\vee T)}{T}\Big)^{\frac{1}{2}}\nn
	\eenr
	with probability at least $1-2\exp\{-(c_{u1}-1)\log(p\vee T)\},$ where $c_{u1}=c_u\wedge \surd(c_u/2).$
\end{lem}

\begin{proof}[Proof of Lemma \ref{lem:nearoptimal.sqterm}] Note that  $\sum_{j=1}^p\big(\|z_{t,-j}^T\eta^0_{(j)}\|_2^2-E\|z_{t,-j}^T\eta^0_{(j)}\|_2^2\big)\sim {\rm subE}(\la),$ where $\la=16\si^2\xi_{2,2}^2.$ For any fixed $\tau \in\cG(u_T,v_T),$ applying the Bernstein's inequality (Lemma \ref{lem:subetail}) we obtain,
	\benr
	pr\Big\{\Big|\sum_{t=\lfloor T\tau^0\rfloor+1}^{\lfloor T\tau\rfloor}\sum_{j=1}^p\big(\|z_{t,-j}^T\eta^0_{(j)}\|_2^2-E\|z_{t,-j}^T\eta^0_{(j)}\|_2^2\big)\Big|\ge d(\lfloor T\tau\rfloor-\lfloor T\tau^0\rfloor)\Big\}\hspace{0.5in}\nn\\
	\le 2\exp\Big\{-\frac{(\lfloor T\tau\rfloor-\lfloor T\tau^0\rfloor)}{2}\Big(\frac{d^2}{\la^2}\wedge \frac{d}{\la}\Big)\Big\}\nn
	\eenr
	Choose $d=\la\{2c_u\log(p\vee T)\big/(\lfloor T\tau\rfloor-\lfloor T\tau^0\rfloor)\}^{1/2}$ and observe that,
	\benr
	(\lfloor T\tau\rfloor-\lfloor T\tau^0\rfloor)\frac{d^2}{2\la^2}&=&c_u\log(p\vee T)\nn\\
	(\lfloor T\tau\rfloor-\lfloor T\tau^0\rfloor)\frac{d}{2\la}&=& \surd(c_u/2)\{Tv_T\log(p\vee T)\}^{1/2}\ge \surd(c_u/2)\log(p\vee T)\nn
	\eenr
	where the inequality follows from the assumption $Tv_T\ge \log(p\vee T).$ A substitution back in the above probability bound yields,
	\benr\label{eq:30}
	\frac{1}{T}\Big|\sum_{t=\lfloor T\tau^0\rfloor+1}^{\lfloor T\tau\rfloor}\sum_{j=1}^p\big(\|z_{t,-j}^T\eta^0_{(j)}\|_2^2-E\|z_{t,-j}^T\eta^0_{(j)}\|_2^2\big)\Big|\nn\\
	\le \surd(2c_u)\la \Big\{\frac{u_T\log(p\vee T)}{T}\Big\}^{\frac{1}{2}}\hspace{-1cm}
	\eenr	
	with probability at least $1-2\exp\big(-c_{u1}\log(p\vee T)\big),$ $c_{u1}=c_u\wedge \surd(c_u/2).$ Applying a union bound over at most $T$ distinct values of $\lfloor T\tau\rfloor,$ yields the bound (\ref{eq:30}) uniformly over $\tau.$ The statements of this lemma are now a direct consequence.
\end{proof}	

\bc$\rule{3.5in}{0.1mm}$\ec

\begin{lem}\label{lem:etabounds.check} Let $\Phi(\cdotp,\cdotp)$ be as defined in (\ref{def:Phi}) and suppose Condition B holds and $T\ge \log(p\vee T).$ Let $\check\mu_{(j)}$ and $\check\g_{(j)},$ $j=1,...,p$ be Step 1 edge estimates of Algorithm 1, and $u_T, v_T$ be any non-negative sequences satisfying $\log (p\vee T)\le Tv_T\le Tu_T.$ Then,
	\benr
	&(i)&\inf_{\substack{\tau\in\cG(u_T,v_T);\\ \tau\ge\tau^0}}\Phi(\eta^0,\eta^0)\ge v_T\ka\xi_{2,2}^2-c_u\si^2\xi_{2,2}^2\Big(\frac{u_T\log(p\vee T)}{T}\Big)^{\frac{1}{2}},\nn\\
	&(ii)&\sup_{\substack{\tau\in\cG(u_T,v_T);\\\tau\ge\tau^0}} \Phi(\check\eta-\eta^0,\check\eta-\eta^0)\le c_u(\si^2\vee\phi) u_T\Big(s\sum_{j=1}^p\|\check\eta_{(j)}-\eta^0_{(j)}\|_2^2\Big)\nn
	\eenr
	with probability $1-o(1).$ Furthermore, when $u_T\ge c_u\si^4\log (p\vee T)\big/T\phi^2,$ we have,
	\benr
	&(iii)&\sup_{\substack{\tau\in\cG(u_T,v_T);\\\tau\ge\tau^0}}\Phi(\eta^0,\eta^0)\le 2u_T\phi\xi_{2,2}^2,\nn\\ &(iv)&\sup_{\substack{\tau\in\cG(u_T,v_T);\\\tau\ge\tau^0}}\big|\Phi(\check\eta-\eta^0,\eta^0)\big|\le  c_u(\si^2\vee \phi) u_T \xi_{2,2}\Big\{s\sum_{j=1}^p\|\check\eta_{(j)}-\eta^0_{(j)}\|_2^2\Big\}^{\frac{1}{2}},\nn
	\eenr
	with probability at least $1-o(1).$
\end{lem}

\begin{proof}[Proof of Lemma \ref{lem:etabounds.check}] Part (i) and Part (iii) are a direct consequence of Lemma \ref{lem:nearoptimal.sqterm}. To prove Part (ii), first note from Theorem \ref{thm:est.nuisance.para} we have that $\check\mu_{(j)}-\mu^0_{(j)}\in\cA_{1j},$ and $\check\g_{(j)}-\g^0_{(j)}\in\cA_{2j},$ $j=1,...,p,$ with probability at least $1-o(1).$ It can be verified that this property yields $\|\check\eta_{(j)}-\eta^0_{(j)}\|_1\le c_u \surd s\|\check\eta_{(j)}-\eta^0_{(j)}\|_2.$ \big(see, e.g. (\ref{eq:9})\big). Now applying Part (ii) of \ref{lem:WUURE} yields,
	\benr
	\sup_{\substack{\tau\in\cG(u_T,v_T);\\\tau\ge\tau^0}} \Phi(\check\eta-\eta^0,\check\eta-\eta^0)\le c_u(\si^2\vee\phi) u_T\Big(s\sum_{j=1}^p\|\check\eta_{(j)}-\eta^0_{(j)}\|_2^2\Big)\nn
	\eenr
	with probability at least $1-o(1).$ Part (iv) follows by an application of the Cauchy-Schwartz inequality together with the bounds of Part (ii) and Part (iii) (see, (\ref{eq:5})). This completes the proof of this lemma.
\end{proof}	

\bc$\rule{3.5in}{0.1mm}$\ec

\begin{lem}\label{lem:term123.check} Suppose Condition B holds and $T\ge \log(p\vee T).$ Let $\check\mu_{(j)},$ $\check\g_{(j)},$ $j=1,...,p$ be Step 1 estimates of Algorithm 1, and assume $u_T, v_T$ satisfy $\log (p\vee T)\le Tv_T\le Tu_T.$ Then,
	\benr
	(i)\,\,\inf_{\substack{\tau\in\cG(u_T,v_T);\\\tau\ge\tau^0}}\frac{1}{T}\sum_{t=\lfloor T\tau^0\rfloor+1}^{\lfloor T\tau\rfloor}\sum_{j=1}^p \|\check\eta_{(j)}^Tz_{t,-j}\|_2^2\ge\hspace{3.3in}\nn\\ \ka\xi_{2,2}^2\Big[v_T-\frac{c_{u}\si^2}{\ka}\Big\{\frac{u_T\log(p\vee T)}{T}\Big\}^{\frac{1}{2}}- c_u(\si^2\vee \phi)\frac{u_T}{\ka\xi_{2,2}}\Big(s\sum_{j=1}^p\|\h\eta_{(j)}-\eta^0_{(j)}\|_2^2\Big)^{\frac{1}{2}}\Big]\hspace{0.3in}\nn\\
	(ii)\,\,\sup_{\substack{\tau\in\cG(u_T,v_T);\\\tau\ge\tau^0}}\frac{1}{T}\Big|\sum_{t=\lfloor T\tau^0\rfloor+1}^{\lfloor T\tau\rfloor}\sum_{j=1}^p(\check\g_{(j)}-\g^0_{(j)})^Tz_{t,-j}z_{t,-j}^T\check\eta_{(j)}\Big|\le\hspace{2.2in}\nn\\
	c_u(\si^2\vee \phi)\xi_{2,2}u_T\Big\{s\sum_{j=1}^p\|\check\g_{(j)}-\g^0_{(j)}\|_2^2\Big\}^{\frac{1}{2}}\Big[1+
	\frac{1}{\xi_{2,2}}\Big\{s\sum_{j=1}^p\|\check\eta_{(j)}-\eta^0_{(j)}\|_2^2\Big\}^{\frac{1}{2}}\Big]\hspace{0.5in}\nn\\
	(iii)\sup_{\substack{\tau\in\cG(u_T,v_T);\\\tau\ge\tau^0}}\frac{1}{T}\Big|\sum_{t=\lfloor T\tau^0\rfloor+1}^{\lfloor T\tau\rfloor}\sum_{j=1}^p\vep_{tj}z_{t,-j}^T\check\eta_{(j)}\Big|\le\hspace{3.2in}\nn\\
	c_u\surd(1+\nu^2)\si^2\xi_{2,1}\Big(\frac{u_T\log(p\vee T)}{T}\Big)^{\frac{1}{2}}+ c_{u}\surd(1+\nu^2)\si^2\Big(\frac{u_T\log(p\vee T)}{T}\Big)^{\frac{1}{2}}\sum_{j=1}^p\|\check\eta_{(j)}-\eta^0_{(j)}\|_1,\nn
	\eenr
	each with probability at least $1-o(1).$
\end{lem}

\begin{proof}[Proof of Lemma \ref{lem:term123.check}]
	Let $\Phi(\cdotp,\cdotp)$ be as defined in (\ref{def:Phi}). Then note that $\Phi(\check\eta,\check\eta)=\Phi(\eta^0,\eta^0)+2\Phi(\check\eta-\eta^0,\eta^0)+\Phi(\check\eta-\eta^0,\check\eta-\eta^0).$ Using this relation together with the bounds of Part (i) and Part (iv) of Lemma \ref{lem:etabounds.check} we obtain,
	\benr
	\inf_{\substack{\tau\in\cG(u_T,v_T);\\\tau\ge\tau^0}}\Phi(\check\eta,\check\eta)&\ge& \inf_{\substack{\tau\in\cG(u_T,v_T);\\\tau\ge\tau^0}}\Phi(\eta^0,\eta^0)-2\sup_{\substack{\tau\in\cG(u_T,v_T);\\\tau\ge\tau^0}}|\Phi(\check\eta-\eta^0,\eta^0)|\nn\\
	&\ge& v_T\ka\xi_{2,2}^2- c_u\si^2\xi_{2,2}^2\Big\{\frac{u_T\log(p\vee T)}{T}\Big\}^{\frac{1}{2}}\nn\\
	&&- c_u(\si^2\vee\phi)u_T \xi_{2,2}\Big(s\sum_{j=1}^p\|\check\eta_{(j)}-\eta^0_{(j)}\|_2^2\Big)^{\frac{1}{2}}\nn
	\eenr
	with probability at least $1-o(1).$ To prove Part (ii), note that using identical arguments as in the proof of Lemma \ref{lem:etabounds.check} it can be shown that,
	\benr
	&&\sup_{\substack{\tau\in\cG(u_T,v_T);\\\tau\ge\tau^0}}\Phi(\check\g-\g^0,\check\g-\g^0)\le   c_u(\si^2\vee \phi)u_Ts\sum_{j=1}^p\|\check\g_{(j)}-\g^0_{(j)}\|_2^2,\nn\\
	&&\sup_{\substack{\tau\in\cG(u_T,v_T);\\\tau\ge\tau^0}}\big|\Phi(\check\g-\g^0,\eta^0)\big|\le c_u(\si^2\vee \phi) u_T \xi_{2,2}\Big\{s\sum_{j=1}^p\|\check\g_{(j)}-\g^0_{(j)}\|_2^2\Big\}^{\frac{1}{2}},\nn
	\eenr	
	with probability at least $1-o(1).$ The above inequalities and the relation $\Phi\big(\check\g-\g^0,\check\eta\big)\le \big|\Phi(\check\g-\g^0,\check\eta-\eta^0)\big|+\big|\Phi(\check\g-\g^0,\eta^0)\big|,$ together with applications of the Cauchy-Schwartz inequality yields,
	\benr
	\sup_{\substack{\tau\in\cG(u_T,v_T);\\\tau\ge\tau^0}} \big|\Phi\big(\check\g-\g^0,\check\eta\big)\big|\le
	c_u(\si^2\vee \phi) u_T\Big(s\sum_{j=1}^p\|\check\g_{(j)}-\g^0_{(j)}\|_2^2\Big)^{\frac{1}{2}}\Big(s\sum_{j=1}^p\|\check\eta_{(j)}-\eta^0_{(j)}\|_2^2\Big)^{\frac{1}{2}}\nn\\
	+ c_u(\si^2\vee \phi) u_T \xi_{2,2}\Big\{s\sum_{j=1}^p\|\check\g_{(j)}-\g^0_{(j)}\|_2^2\Big\}^{\frac{1}{2}}\hspace{1in}\nn\\
	\le  c_u(\si^2\vee \phi)\xi_{2,2}u_T\Big\{s\sum_{j=1}^p\|\check\g_{(j)}-\g^0_{(j)}\|_2^2\Big\}^{\frac{1}{2}}\Big[1+
	\frac{1}{\xi_{2,2}}\Big\{s\sum_{j=1}^p\|\check\eta_{(j)}-\eta^0_{(j)}\|_2^2\Big\}^{\frac{1}{2}}\Big]\nn
	\eenr	
	with probability at least $1-o(1).$ To prove Part (iii), note that,
	\benr
	\sup_{\substack{\tau\in\cG(u_T,v_T);\\\tau\ge\tau^0}}\frac{1}{T}\Big|\sum_{t=\lfloor T\tau^0\rfloor+1}^{\lfloor T\tau\rfloor}\sum_{j=1}^p\vep_{tj}z_{t,-j}^T\check\eta_{(j)}\Big|&\le &\sup_{\substack{\tau\in\cG(u_T,v_T);\\\tau\ge\tau^0}}\frac{1}{T}\Big|\sum_{t=\lfloor T\tau^0\rfloor+1}^{\lfloor T\tau\rfloor}\sum_{j=1}^p\vep_{tj}z_{t,-j}^T\eta^0_{(j)}\Big|\nn\\
	&&+\sup_{\substack{\tau\in\cG(u_T,v_T);\\\tau\ge\tau^0}}\frac{1}{T}\Big|\sum_{t=\lfloor T\tau^0\rfloor+1}^{\lfloor T\tau\rfloor}\sum_{j=1}^p\vep_{tj}z_{t,-j}^T(\check\eta_{(j)}-\eta^0_{(j)})\Big|\nn\\
	&:=& R1+R2.\nn
	\eenr 	
	Now using Lemma \ref{lem:nearoptimalcross.check}, we have
	\benr
	&R1&\le c_{u}\surd(1+\nu^2)\si^2\xi_{2,1}\Big(\frac{u_T\log(p\vee T)}{T}\Big)^{\frac{1}{2}},\,\,{\rm and}\nn\\
	&R2&\le c_{u}\surd(1+\nu^2)\si^2\Big(\frac{u_T\log(p\vee T)}{T}\Big)^{\frac{1}{2}}\sum_{j=1}^p\|\check\eta_{(j)}-\eta^0_{(j)}\|_1
	\eenr
	with probability at least $1-o(1).$ Part (iv) now follows by combining bounds for terms $R1$ and $R2.$
\end{proof}

\bc$\rule{3.5in}{0.1mm}$\ec
\section{Uniform (over $\tau$) Restricted Eigenvalue Condition}
\begin{lem}\label{lem:UURElem1} Let $z_t\in\R^p,$ $t=1,...,n$ be independent ${\rm subG}(\si)$ r.v's and $\la=16\si^2.$ Additionally, for any $s\ge 1,$ let $\cK_p(s)=\{\delta\in\R^{p};\,\,\|\delta\|_1\le 1,\,\|\delta\|_0\le s\}.$ Then for non-negative $0\le v_T\le u_T,$ and any $d_1>0,$ we have $T\ge 2,$
	\benr
	pr\left[\sup_{\substack{\tau\in\cG(u_T,v_T);\\\tau\ge\tau^0}}\sup_{\delta\in \cK_p(2s)}\frac{1}{T}\Big|\sum_{t=\lfloor T\tau^0\rfloor+1}^{\lfloor T\tau\rfloor}\big\{\|z_t^T\delta\|_2^2-E\|z_t^T\delta\|_2^2\big\}\Big|\ge d_1u_T \right]\le\hspace{1in}\nn\\
	2\exp\Big\{-\frac{Tv_T}{2}\Big(\frac{d_1^2}{\la^2}\wedge \frac{d_1}{\la}\Big)+3s\log(p\vee T)\Big\}\nn
	\eenr	
\end{lem}

\begin{proof}[Proof of Lemma \ref{lem:UURElem1}] Consider any fixed $\delta\in\R^p,$ with $\|\delta\|_2\le 1,$ then from Lemma \ref{lem:sqsubGsubE} we have $\|z_t^T\delta\|_2^2-E\|z_t^T\delta\|_2^2\sim{\rm subE}(\la),$ with $\la=16\si^2.$ Now, for any fixed $\tau\in\cG(u_T,v_T),$ $\tau\ge\tau^0$ applying Lemma \ref{lem:subetail} (Bernstein's inequality) we have,
	\benr	
	pr\Big(\Big|\sum_{t=\lfloor T\tau^0\rfloor+1}^{\lfloor T\tau\rfloor}\|z_t^T\delta\|_2^2-E\|z_t^T\delta\|_2^2\Big|>d(\lfloor T\tau\rfloor-\lfloor T\tau^0\rfloor)\Big)\nn\\
	\le 2\exp\Big\{-\frac{(\lfloor T\tau\rfloor-\lfloor T\tau^0\rfloor)}{2}\Big(\frac{d^2}{\la^2}\wedge \frac{d}{\la}\Big)\Big\}\hspace{-1cm}\nn
	\eenr	
	Choose $d=d_1Tu_T/(\lfloor T\tau\rfloor-\lfloor T\tau^0\rfloor)$ and observe that by definition of the set $\cG(u_T,v_T),$ we have $Tv_T\le (\lfloor T\tau\rfloor-\lfloor T\tau^0\rfloor)\le Tu_T,$ this in turn yields $d_1\le d,$ and consequently,
	\benr\label{eq:6}
	pr\Big(\frac{1}{T}\Big|\sum_{t=\lfloor T\tau^0\rfloor+1}^{\lfloor T\tau\rfloor}\|z_t^T\delta\|_2^2-E\|z_t^T\delta\|_2^2\Big|\ge d_1u_T\Big)\nn\\
	\le 2\exp\Big\{-\frac{Tv_T}{2}\Big(\frac{d_1^2}{\la^2}\wedge \frac{d_1}{\la}\Big)\Big\}\hspace{-1cm}
	\eenr
	Using the inequality (\ref{eq:6}) and a covering number argument, it can be shown that (see, Lemma 15 of the supplementary materials of \cite{loh2012}) for any $s\ge 1,$
	\benr
	pr\Big(\sup_{\delta\in\cK_p(2s)}\frac{1}{T}\Big|\sum_{t=\lfloor T\tau^0\rfloor+1}^{\lfloor T\tau\rfloor}\|z_t^T\delta\|_2^2-E\|z_t^T\delta\|_2^2\Big|\ge d_1u_T\Big)\nn\\
	\le 2\exp\Big\{-\frac{Tv_T}{2}\Big(\frac{d_1^2}{\la^2}\wedge \frac{d_1}{\la}\Big)+2s\log (p\vee T)\Big\}.\hspace{-1cm}\nn
	\eenr
	Finally, uniformity over the set $\cG(u_T,v_T)$ can be obtained by applying a union bound over the at most $T$ distinct values of $\lfloor T\tau\rfloor$ for $\tau\in\cG(u_T,v_T),$ thus yielding the statement of this lemma.
\end{proof}

\bc$\rule{3.5in}{0.1mm}$\ec

\begin{lem}\label{lem:WUURE} Suppose Condition B holds and let $0\le v_T\le u_T$ be any non-negative sequences. Then for all $\delta_{(j)}\in\R^{p-1},$ $j=1,...,p,$ and $T\ge 2,$ we have,
	\benr
	(i)\,\,\sup_{\substack{\tau\in\cG(u_T,v_T);\\\tau\ge\tau^0}}\frac{1}{T}\sum_{t=\lfloor T\tau^0\rfloor+1}^{\lfloor T\tau\rfloor}\sum_{j=1}^p \delta_{(j)}^T z_{t,-j}z_{t,-j}^T\delta_{(j)}\le\hspace{1.5in}\nn\\
	 c_u (\si^2\vee\phi) u_T\log(p\vee T)\Big(\sum_{j=1}^p\|\delta_{(j)}\|_2^2+\sum_{j=1}^p\|\delta_{(j)}\|_1^2\Big)\hspace{-1cm}\nn
	\eenr
	with probability at least $1-2\exp\big\{-\log (p\vee T)\big\}.$ Additionally assuming that $T\ge \log(p\vee T)$ and $v_T$ satisfies $Tv_T\ge \log(p\vee T),$ then for all $\delta_{(j)}\in\R^{p-1},$ $j=1,...,p,$
	\benr
	(ii) \,\,\sup_{\substack{\tau\in\cG(u_T,v_T);\\\tau\ge\tau^0}}\frac{1}{T}\sum_{t=\lfloor T\tau^0\rfloor+1}^{\lfloor T\tau\rfloor}\sum_{j=1}^p \delta_{(j)}^T z_{t,-j}z_{t,-j}^T\delta_{(j)}\le\hspace{1.5in}\nn\\
	 c_u (\si^2\vee\phi) u_T\Big(\sum_{j=1}^p\|\delta_{(j)}\|_2^2+\sum_{j=1}^p\|\delta_{(j)}\|_1^2\Big)\hspace{1cm}\nn
	\eenr
	with probability at least $1-2\exp\big\{-\log (p\vee T)\big\}.$
\end{lem}

\begin{proof}[Proof of Lemma \ref{lem:WUURE}]  w.l.o.g. assume $v_T\ge (1/T)$ \big(see, Lemma \ref{lem:nearoptimalcross}\big). Now for any $s\ge 1,$ consider any non-negative $u_T,$ any $\delta_{(j)}\in\cK_{p-1}(2s),$ $j=1,...,p.$ Then for any $d_1>0,$ applying a union bound to the result of Lemma \ref{lem:UURElem1} over the components $j=1,...,p$ we obtain,
	\benr\label{eq:7}
	\qquad\sup_{\substack{\tau\in\cG(u_T,v_T);\\\tau\ge\tau^0}}\sup_{\substack{\delta_{(j)}\in\cK(2s);\\j=1,...,p}}\frac{1}{T}\Big|\sum_{t=\lfloor T\tau^0\rfloor+1}^{\lfloor T\tau\rfloor}\|z_{t,-j}^T\delta_{(j)}\|_2^2-E\|z_{t,-j}^T\delta_{(j)}\|_2^2\Big|\le d_1u_T
	\eenr
	with probability at least $1-2\exp\Big\{-\frac{Tv_T}{2}\Big(\frac{d_1^2}{\la^2}\wedge \frac{d_1}{\la}\Big)+4s\log (p\vee T)\Big\}.$ It can be shown that the bound (\ref{eq:7}) in turn implies that (see, Lemma 12 of supplement of \cite{loh2012}), for all $\tau\in \cG(u_T,v_T),$ and for all $\delta_{(j)}\in\R^{p-1},$ $j=1,...,p,$
	\benr
	\frac{1}{T}\Big|\sum_{t=\lfloor T\tau^0\rfloor+1}^{\lfloor T\tau\rfloor}\sum_{j=1}^p\|z_{t,-j}^T\delta_{(j)}\|_2^2-E\|z_{t,-j}^T\delta_{(j)}\|_2^2\Big|\le 27d_1u_T(\sum_{j=1}^p\|\delta_{(j)}\|_2^2+(1/s)\sum_{j=1}^p\|\delta_{(j)}\|_1^2)\nn
	\eenr
	with probability at least $1-2\exp\Big\{-\frac{Tv_T}{2}\Big(\frac{d_1^2}{\la^2}\wedge \frac{d_1}{\la}\Big)+4s\log (p\vee T)\Big\}.$ Now choose $d_1=10\la\log(p\vee T),$ and note that $\frac{Tv_T}{2}\Big(\frac{d_1^2}{\la^2}\wedge \frac{d_1}{\la}\Big)\ge 5\log(p\vee T).$
	This follows since $Tv_T\ge 1,$ and that $d_1/\la\ge 1.$ A substitution back in the probability bound yields,
	\benr
	\frac{1}{T}\Big|\sum_{t=\lfloor T\tau^0\rfloor+1}^{\lfloor T\tau\rfloor}\sum_{j=1}^p\|z_{t,-j}^T\delta_{(j)}\|_2^2-E\|z_{t,-j}^T\delta_{(j)}\|_2^2\Big|\hspace{2cm}\nn\\
	\le 270 \la u_T\log(p\vee T)\Big\{\sum_{j=1}^p\|\delta_{(j)}\|_2^2+\frac{1}{s}\sum_{j=1}^p\|\delta_{(j)}\|_1^2\Big\}\hspace{-1cm},\nn
	\eenr
	with probability at least $1-2\exp\big\{-5\log(p\vee T)+4s\log (p\vee T)\big\}.$ The statement of Part (i) follows by setting $s= 1.$ The proof of Part (ii) is quite analogous. This can be obtained by proceeding as earlier with (\ref{eq:7}) above, and additionally utilizing $Tv_T\ge \log(p\vee T),$ and setting $d_1=10\la,$ instead of the choice made for Part (i). This completes the proof of the result.
\end{proof}
\bc$\rule{3.5in}{0.1mm}$\ec

\begin{lem}\label{lem:lower.RE.ordinary} Suppose Condition A$'$ and B hold, then for $i=1,2,$
	\benr
	\min_{j=1,...,p;}\inf_{\substack{\tau\in(0,1);\\\tau\ge c_{u1}l_T}}\inf_{\substack{\delta\in\cA_{ij};\\\|\delta\|_2=1}} \frac{1}{\lfloor T\tau\rfloor} \sum_{t=1}^{\lfloor T\tau\rfloor}\delta^Tz_{t,-j}z_{t,-j}^T\delta \ge \frac{\ka}{2}.\nn
	\eenr
	with probability at least $1-2\exp\{-c_u\log(p\vee T)\},$ for some $c_u>0$ and for $T$ sufficiently large.
\end{lem}
Lemma \ref{lem:lower.RE.ordinary} is a nearly direct extension of the usual restricted eigenvalue condition. Its proof is analogous to those available in the literature, for e.g., Corollary 1 of \cite{loh2012}. In comparison to the typical restricted eigenvalue condition, Lemma \ref{lem:lower.RE.ordinary} has  additional uniformity over $\tau,$ $i$ and $j,$ which can be obtained by simply using additional union bounds.
\bc$\rule{3.5in}{0.1mm}$\ec


\section{Auxiliary results}\label{sec:aux}

In the following Definition's \ref{def:subg}, \ref{def:sube}, and Lemma's \ref{lem:tailb}-\ref{lem:subetail}, we provide basic properties of subgaussian and subexponential distributions. These are largely reproduced from \cite{vershynin2019high} and \cite{rigollet201518}. Theorem \ref{thm:kolmogorov} and \ref{thm:argmax} below reproduce Kolmogorov's inequality and the argmax theorem. Lemma \ref{lem:condnumberbound} provides an upper bound for the $\ell_2$ norm of the parameter vectors defined in Section \ref{sec:intro}.

\begin{Def}\label{def:subg} {\bf Sub-gaussian r.v.}: A random variable $X\in\R$ is said to be sub-gaussian with parameter $\si>0$ \big(denoted by $X\sim{\rm subG(\si)}$\big) if $E(X)=0$ and its moment generating function
	\benr
	E(\e^{tX})\le \e^{t^2\si^2/2},\qquad \forall\,\, t\in\R
	\eenr
	Furthermore, a random vector $X\in\R^p$ is said to be sub-gaussian with parameter $\si,$ if the inner products $\langle X, v\rangle\sim {\rm subG}(\si)$ for any $v\in\R^p$ with $\|v\|_2 = 1.$
\end{Def}

\begin{Def}\label{def:sube} {\bf Sub-exponential r.v.}: A random variable $X\in\R$ is said to be sub-exponential with parameter $\si>0$ \big(denoted by $X\sim{\rm subE(\si)}$\big) if $E(X)=0$ and its moment generating function
	\benr
	E(\e^{tX})\le \e^{t^2\si^2/2},\qquad \forall\,\, |t|\le \frac{1}{\si}\nn
	\eenr
\end{Def}

\begin{lem}\label{lem:tailb}[Tail bounds] (i) If $X\sim {\rm subG}(\si),$ then,
	\benr
	pr(|X|\ge \la)\le 2\exp(-\la^2/2\si^2).\nn
	\eenr
	(ii) If $X\sim {\rm subE}(\si),$ then
	\benr
	pr(|X|\ge \la)\le 2\exp\Big\{-\frac{1}{2}\Big(\frac{\la^2}{\si^2}\wedge\frac{\la}{\si}\Big) \Big\}.\nn
	\eenr
\end{lem}

\begin{proof}[Proof of Lemma \ref{lem:tailb}] This proof is a simple application of the Markov inequality. For any $t>0,$
	\benr
	pr(X\ge \la)=pr(tX\ge t\la)\le \frac{E\e^{tX}}{\e^{t\la}}=\e^{-t\la+t^2\si^2/2}.\nn
	\eenr
	Minimizing over $t>0,$ yields the choice $t^*=\la/\si^2,$ and substituting in the above bound ,we obtain,
	\benr
	pr(X\ge \la)\le \inf_{t>0}\e^{-t\la+t^2\si^2/2}=e^{-\la^2/2\si^2}.\nn
	\eenr
	Repeating the same for $P(X\le -\la)$ yields part (i) of the lemma. To prove Part (ii),   repeat the above argument with $t\in(0,1/\si],$ to obtain,
	\benr
	pr(X\ge \la)=pr(tX\ge t\la)\le\e^{-t\la+t^2\si^2/2}.
	\eenr
	As in the subgaussian case, to obtain the tightest bound one needs to find $t^*$ that minimizes $-t\la+t^2\si^2/2,$ with the additional constraint for this subexponential case that $t\in(0,1/\si].$ We know that the unconstrained minimum occurs at $t^*=\la/\si^2>0.$ Now consider two cases:
	\begin{enumerate}[itemindent=0mm]
		\item If $t^*<(0,1/\si] \Leftrightarrow\la\le \si$ then the unconstrained minimum is same as the constrained minimum, and substituting this value yields the same tail behavior as the subgaussian case.
		\item If $t^*>(1/\si)\Leftrightarrow\la>\si,$ then note that $-t\la+t^2\si^2/2$ is decreasing in $t,$ in the interval $(0,(1/\si)],$ thus the minimum occurs at the boundary $t=1/\si.$ Substituting in the tail bound we obtain for this case,
		\benr
		pr(X\ge \la)\le\e^{-t\la+t^2\si^2/2}= \exp\{-(\la/\si)+(1/2)\}\le\exp{(-\la/2\si)},\nn
		\eenr
		where the final inequality follows since $\la>\si.$
	\end{enumerate}
	Part (ii) of the lemma is obtained by combining the results of the above two cases.
\end{proof}

\bc$\rule{3.5in}{0.1mm}$\ec

\begin{lem}[Moment bounds]\label{lem:momentprop} (i) If $X\sim {\rm subG}(\si),$ then
	\benr
	E|X|^k\le 3k\si^{k} k^{k/2}, \qquad k\ge 1.\nn
	\eenr
	(ii) If $X\sim {\rm subE}(\si),$ then
	\benr
	E|X|^k\le 4\si^k k^k, \qquad k> 0.\nn
	\eenr	
\end{lem}

\begin{proof}[Proof of Lemma \ref{lem:momentprop}]
	Consider $X\sim {\rm subG}(\si),$ and w.l.o.g assume that $\si=1$ (else define $X^*=X/\si$). Using the integrated tail probability expectation formula, we have for any $k> 0,$
	\benr
	E|X|^k&=&\int_{0}^{\iny}pr(|X|^k>t)dt=\int_{0}^{\iny}pr(|X|>t^{1/k})dt\nn\\
	&\le& 2\int_0^\iny\exp\Big(-\frac{t^{2/k}}{2}\Big)dt\nn\\
	&=& 2^{k/2}k\int_0^\iny\e^{-u}u^{k/2-1}du,\qquad u=\frac{t^{2/k}}{2}\nn\\
	&=&  2^{k/2}k \Gamma(k/2)\nn
	\eenr
	Here the first inequality follows from the tail bound Lemma \ref{lem:tailb}. Now, for $x\ge 1/2,$ we have the inequality $\G(x)\le 3x^x,$ thus for $k\ge 1$ we have,  $\G(k/2)\le 3(k/2)^{(k/2)}.$ A substitution back in the moment bound yields desired bound of Part (i).
	
	To prove the moment bound of Part (ii). As before, w.l.o.g. assume  $\si=1.$ Consider the inequality,
	\benr
	|x|^k\le k^k(e^x+e^{-x})\nn
	\eenr
	which is valid for all $x\in\R$ and $k>0.$ Substitute x=X and take expectation to get,
	\benr
	E|X|^k\le k^k(E\e^X+E\e^{-X}).\nn
	\eenr
	Since in this case $\si=1,$ from the mgf condition, at $t=\pm 1$ we have, $E\e^X\le\e^{1/2}\le 2,$ and $E\e^{-X}\le 2.$ Thus for any $k>0,$
	\benr
	E|X|^k\le 4k^k\nn
	\eenr
	This yields the desired moment bound of Part (ii).
\end{proof}

\bc$\rule{3.5in}{0.1mm}$\ec

\begin{lem} \label{lem:lcsubG} Assume that $X\sim{\rm subG}(\si),$ and that $\al\in\R,$ then $\alpha X\sim{\rm subG}(|\alpha|\si).$ Moreover if $X_1\sim{\rm subG}(\si_1)$ and $X_2\sim{\rm subG}(\si_2),$ then $X_1+X_2\sim {\rm subG}(\si_1+\si_2).$
\end{lem}

\begin{proof}[Proof of Lemma \ref{lem:lcsubG}] The first part follows directly from the inequality $E(\e^{t\al X})\le \exp(t^2\al^2\si^2/2).$ To prove Part (ii) use the H\"older's inequality to obtain,
	\benr
	E(\e^{t(X_1+X_2)})&=&E(\e^{tX_1}\e^{tX_2})\le \{E(\e^{tX_1p})\}^{\frac{1}{p}}\{E(\e^{tX_2q})\}^{\frac{1}{q}}\nn\\
	&\le& \e^{\frac{t^2}{2}\si_1^2p^2}\e^{\frac{t^2}{2}\si_2^2q^2}=\e^{\frac{t^2}{2}(p\si_1^2+q\si_2^2)}\nn
	\eenr
	where $p,q\in[1,\iny],$ with $1/p+1/q=1.$ Choose $p^{*}=(\si_2/\si_1)+1,$ $q^*=(\si_1/\si_2)+1$ to obtain $E(\e^{t(X_1+X_2)})\le \exp\big\{\frac{t^2}{2}(\si_1+\si_2)^2\big\}.$ This completes the proof of this lemma.	
\end{proof}

\bc$\rule{3.5in}{0.1mm}$\ec

\begin{lem}\label{lem:lcsubE} Assume that $X\sim{\rm subE(\si)},$ and that $\al\in\R,$ then $\alpha X\sim{\rm subE}(|\alpha|\si).$ Moreover, assume that $X_1\sim{\rm subE(\si_1)}$ and $X_2\sim{\rm subE(\si_2)},$ then $X_1+X_2\sim{\rm subE(\si_1+\si_2)}.$
\end{lem}
The proof of Lemma \ref{lem:lcsubE} is analogous to that of Lemma \ref{lem:lcsubG} and is thus omitted.

\bc$\rule{3.5in}{0.1mm}$\ec

\begin{lem}[Lemma 1.12 of \cite{rigollet201518}]\label{lem:sqsubGsubE} Let $X\sim {\rm subG}(\si)$ then the random variable $Z=X^2-E[X^2]$ is sub-exponential: $Z\sim {\rm subE(16\si^2)}.$
\end{lem}

\bc$\rule{3.5in}{0.1mm}$\ec

The next result is Bernstein's inequality, reproduced from Lemma 1.13 of \cite{rigollet201518}.
\begin{lem}[Bernstein's inequality]\label{lem:subetail} Let $X_1,X_2,...,X_T$ be independent random variables such that $X_t\sim {\rm subE}(\si).$ Then for any $d>0$ we have,
	\benr
	pr(|\bar X|>d)\le 2\exp\Big\{-\frac{T}{2}\Big(\frac{d^2}{\si^2}\wedge \frac{d}{\si}\Big)\Big\}\nn
	\eenr
\end{lem}

\bc$\rule{3.5in}{0.1mm}$\ec

The next result is Kolmogorov's inequality reproduced from \cite{hajek1955generalization}
\begin{thm}[Kolmogorov's inequality]\label{thm:kolmogorov} If $\xi_1,\xi_2,...$ is a sequence of mutually independent random variables with mean values $E(\xi_k)=0$ and finite variance ${\rm var}(\xi_k)=D_k^2$ $(k=1,2,...),$ we have, for any $\vep>0,$
	\benr
	pr\Big(\max_{1\le k\le m}\big|\xi_1+\xi_2+...+\xi_k\big|>\vep\Big)\le \frac{1}{\vep^2}\sum_{k=1}^mD_k^2\nn
	\eenr	
\end{thm}

\bc$\rule{3.5in}{0.1mm}$\ec

The following theorem is the well known `Argmax' theorem reproduced from Theorem 3.2.2 of \cite{vaart1996weak}
\begin{thm}[Argmax Theorem]\label{thm:argmax} Let $\cM_n,\cM$ be stochastic processes indexed by a metric space $H$ such that $\cM_n\Rightarrow\cM$ in $\ell^{\iny}(K)$ for every compact set $K\subseteq H.$ Suppose that almost all sample paths $h\to \cM(h)$ are upper semicontinuous and posses a unique maximum at a (random) point $\h h,$ which as a random map in $H$ is tight. If the sequence $\h h_n$ is uniformly tight and satisfies $\cM_n(\h h_n)\ge \sup_h \cM_n(h)-o_p(1),$ then $\h h_n\Rightarrow \h h$ in $H.$
\end{thm}

\bc$\rule{3.5in}{0.1mm}$\ec

\begin{lem}\label{lem:condnumberbound} Suppose condition B holds, and let $\mu^0_{(j)}$ and $\g^0_{(j)},$ be as defined in (\ref{def:muga}). Then we have,
	\benr
	\max_{1\le j\le p}\Big(\|\mu^0_{(j)}\|_2 \vee \|\g^0_{(j)}\|_2\Big)\le \nu,\nn
	\eenr
\end{lem}

\begin{proof}{\bf of Lemma \ref{lem:condnumberbound}} Let $\Omega=\Sigma^{-1}$ be the precision matrix corresponding to $\Si.$ Then we can write $\Om_{jj}=-(\Si_{jj}-\Si_{j,-j}\mu^0_{(j)})^{-1},$ and $\Om_{-j,j}=-\Om_{jj}\mu^0_{(j)},$ for each $j=1,...,p,$ (see, e.g., \cite{yuan2010high}). We also have that $1\big/\phi\le \max_{j}|\Om_{jj}|\le 1\big/\ka.$ Now note that the $\ell_2$ norm of the rows (or columns) of $\Om$ are bounded above,
	i.e., $\|\Om_{j\cdot}\|_2=\|\Om e_j\|_2\le 1/\ka.$ This finally implies that
	\benr\label{eq:2}
	\|\mu^0_{(j)}\|_2=\|-\Om_{-j,j}\big/\Om_{jj}\|_2\le \|\Om_{j\cdot}\|_2\big/ |\Om_{jj}|\le \frac{\phi}{\ka}=\nu
	\eenr
	Since the r.h.s. in (\ref{eq:2}) is free of $j,$ this implies that $\max_{j}\|\mu^0_{(j)}\|\le \nu.$ Identical arguments can be used to show that $\max_{j}\|\g^0_{(j)}\|\le \nu.$ These two statements together imply the statement of the lemma.
\end{proof}

\bc$\rule{3.5in}{0.1mm}$\ec

\section{Further computational details and additional numerical results}\label{sec:add.numerical}

\vspace{5mm}
\noi{\bf Computation of asymptotic variances and negative drifts}: Here we discuss the computation of the asymptotic variance parameters $\si_1^{*2},\bar\si_1^2,\si_2^{*2},\bar\si_2^2,$ and the negative drift parameters $\si_1^2,\si_2^2$ and $\psi_{\iny}.$ Estimation of these quantities is necessary for the implementation of confidence intervals for $\tau^0$ using the result of Theorem \ref{thm:limitingdist} and Theorem \ref{thm:wc.non.vanishing}. 

We begin by recalling that Step 2 of Algorithm 1 or Algorithm 2 yields the availablility of $\ell_1$ regularized estimates $\h\mu_{(j)}$ and $\h\g_{(j)},$ $j=1,...,p$ that satisfy the error bounds (\ref{eq:optimalmeans}). In order to alleviate finite sample regularization biases we undertake a first supplemental step of refitting these coefficient estimates as ordinary least squares on the set of estimated non-zero indices computed as follows. Let $\h S_{1j}=\{k\,;\,\h\mu_{(j)k\ne 0}$ and $\h S_{2j}=\{k\,;\,\h\g_{(j)k\ne 0}\},$ $j=1,...,p.$ Then define the refitted least squares versions as, $\tilde\mu_{(j)}=(\tilde\mu^T_{(j)\h S_{1j}},0^T_{S_{1j}^c})^T,$ and  $\tilde\g_{(j)}=(\tilde\g^T_{(j)\h S_{2j}},0^T_{S_{2j}^c})^T,$ where,
\benr
\tilde\mu_{(j)\h S_{1j}}=\argmin_{\substack{\mu_{(j)}\in \R^{|\h S_{1j}|}}}\frac{1}{\lfloor T\tilde\tau\rfloor}\sum_{t=1}^{\lfloor T\tilde\tau\rfloor}\big(z_{tj}-\big(z_{t,-j}\big)_{\h S_{1j}}^T\mu_{(j)}\big)^2,\quad j=1,...,p\nn
\eenr
and symmetrically define $\tilde\g_{(j)\h S_{2j}},$ $j=1,...,p.$  It is known from the literature that refitted estimates preserve the rate of convergence of the regularized version while reducing finite sample biases, see, e.g. \cite{belloni2011square} and \cite{belloni2017pivotal}. The jump sizes $\xi_{2,2}$ and $\psi$ are then estimated using these refitted parameters, i.e., let $\tilde\eta_{(j)}=\tilde\mu_{(j)}-\tilde\g_{(j)},$ then 
$\tilde\xi_{2,2}=\big(\sum_{j=1}^p\|\tilde\eta_{(j)}\|_2^2\big)^{1/2}$ and $\tilde\psi=\tilde\xi_{2,2}/\surd{p}.$ 

Next we consider the drift parameters $\si_1^2$ and $\si_2^2$ defined in Condition D. Note the finite sample representation of these parameters is, $\xi_{2,2}^{-2}\sum_{j=1}^p\eta^{0T}_{(j)}\Si_{-j,-j}\eta^0_{(j)}$ and $\xi_{2,2}^{-2}\sum_{j=1}^p\eta^{0T}_{(j)}\D_{-j,-j}\eta^0_{(j)},$ respectively. A plug in version is computed by utilizing the above described $\tilde\xi_{2,2}$ and $\tilde\eta_{(j)},$ $j=1,...,p.$ The covariances in the above calculation are chosen as the sample covariances $\tilde\Si$ and $\tilde\D$ computed on the binary partition of the data from the estimated change point $\tilde\tau.$ We note that since we are not interested in the estimation of the covariances themselves but instead the quardratic form described above, thus utilizing the sample covariances is effectively identical to utilizing refitted covariances on the adjacency matrix defined by the jump parameters $\tilde\eta_{(j)},$ in turn making this shortcut valid despite potential high dimensionality. 

Next consider the asymptotic variances $\si_1^{*2},\bar\si_1^2,\si_2^{*2},\bar\si_2^2,$ defined in Condition D and Condition B$'$. A plug in estimate of these quantities is infeasible since no closed form expressions are available for these variances. Instead we approximate them numerically as follows. Recall from Condition D the finite sample representation of $\si_1^{*2}$ defined as ${\rm var}\big[\xi_{2,2}^{-1}p^{-1/2}\sum_{j=1}^p\vep_{tj}z_{t,-j}^T\eta^0_{(j)}\big].$ In order to approximate this expression, define, 
\benr
\tilde\vep_{tj}=\begin{cases}z_{tj}-z_{t,-j}^T\tilde\mu_{(j)}, & t=1,...,\lfloor T\tilde\tau\rfloor\\
	z_{tj}-z_{t,-j}^T\tilde\g_{(j)}, & t=\lfloor T\tilde\tau\rfloor+1,...,T,	
\end{cases}\nn
\eenr
Then one can obtain $\lfloor T\tilde\tau\rfloor$ predicted realizations as $\tilde\z_t^*=\tilde\xi_{2,2}^{-1}p^{-1/2}\sum_{j=1}^p\tilde\vep_{tj}z_{t,-j}^T\tilde\eta_{(j)},$ $t=1,...,\lfloor T\tilde\tau\rfloor.$ The parameter $\si_1^{*2}$ is then estimated as the sample variance of these realizations, 
\benr
\tilde\si_1^{*2}=\frac{1}{\lfloor T\tilde\tau\rfloor}\sum_{j=1}^{\lfloor T\tilde\tau\rfloor}\big(\tilde\z_t^{*}-{\bar{\tilde\z}}_t^*\big)^2.\nn
\eenr 
The parameter $\si_2^{*2}$ is approximated similarly as the sample variance of predicted realizations $\tilde\z_t^*$ from the post binary partition $\lfloor T\tilde\tau\rfloor+1,...,T.$ The approach to estimation of the variance parameters $\bar\si_1^2$ and $\bar\si_2^2$ is conceptually similar as above. Recall the representation of these parameters from Condition B$'$: ${\rm var}\big[p^{-1}\sum_{j=1}^p\big\{2\vep_{tj}z_{t,-j}^T\eta^0_{(j)}-\eta^{0T}_{(j)}z_{t,-j}z_{t,-j}^T\eta^0_{(j)}\big\}\big].$ One can obtain predicted realizations of this distribution using the previously estimated parameters and approximate $\bar\si_1^{2}$ as corresponding sample variances of these realizations. Similar calculation can be utilized for the approximation of $\bar\si_2^2.$

\vspace{5mm}
\noi{\bf Empirically fitting distribution law $\cL$}: This subsection illustrates an example of the process that can be utilized to empirically fit a distribution $\cL$ of Condition $B',$ in order to implement the result of Theorem \ref{thm:wc.non.vanishing} thereby allowing construction of a confidence interval for the change point parameter in the non-vanishing jump size regime. First note that the distribution under question is that of the sequence $p^{-1}\sum_{j=1}^p\big\{2\vep_{tj}z_{t,-j}^T\eta^0_{(j)}-\eta^{0T}_{(j)}z_{t,-j}z_{t,-j}^T\eta^0_{(j)}\big\}$ in the limit. As described in the previous subsection, the available data $z$ and plug in estimates of underlying parameters allow one to obtain predicted realizations from this distribution. Figure \ref{fig:increment.alg3} provides an example of the centered and scaled distribution of these realizations when the data generating process is Gaussian. 

Two key observations at this stage are as follows. First, given the sub-gaussian assumption of Condition B, the distribution under investigation must be sub-exponential. This observation allows considerable reduction of the potential distributions to be tested to a sub-class of well known sub-expoential distributions. Next, note that the second part of the sequence under consideration is a quadratic form, thus it induces a skewness in the distribution with the underlying skewness diminishing with a decreasing jump size\footnote{Recall from (\ref{eq:finite.var}) that the variance of the quadratic form is upper bounded by $O(\psi^4),$ whereas the first symmetric part if $O(\psi^2).$}. Since this quadratic form appears with a negative sign in the distribution of interest, thus the skewness appears through a larger left tail. 

In view of the above two observations, we consider a negative centered and scaled chi-square distribution as an empirical fit. The negative sign switches the right skew of a chi-square to a left skew, moreover, an increasing degrees of freedom parameter of this chi-square allows one to fit a distribution from complete left skew all the way to perfect symmetry. Specifically, we utlize Algorithm 3 to empirically fit a distribution $\cL,$ where the degrees of freedom of this chi-square distribution is chosen so as to maximize the p-value of the Kolmogorov-Smirnov goodness of fit test, i.e., so as to provide the best fitting chi-square to the distribution of interest.

\vspace{-2mm}
\begin{figure}[H]
	\noi\rule{\textwidth}{0.5pt}
	
	\vspace{-2mm}
	\flushleft {\bf Algorithm 3:} Empirically fitting a centered and sclaed $\chi^2_k$ to the distribution law $\cL.$
	
	\vspace{-2mm}
	\noi\rule{\textwidth}{0.5pt}
	
	\vspace{-1.5mm}
	\flushleft{\bf Step 1:} Obtain refitted estimates $\tilde\mu_{(j)},$ $\tilde\g_{(j)}$ and $\tilde\eta_{(j)},$ $j=1,...,p$ and obtain predicted realizations,
	\benr
	\tilde\z_t=p^{-1}\sum_{j=1}^p\big\{2\tilde\vep_{tj}z_{t,-j}^T\tilde\eta_{(j)}-\tilde\eta^{T}_{(j)}z_{t,-j}z_{t,-j}^T\tilde\eta_{(j)}\big\}, & t=1,...,T.
	\eenr

	\vspace{-1.5mm}
	\flushleft{\bf Step 2:} Piecewise center and scale the predicted realizations $\tilde\z_t,$ i.e.
	\benr
	\tilde\z_t^*=\begin{cases}
		(\tilde\z_t-\bar\z_{t1})/sd_{1}, & t=1,...,\lfloor T\tilde\tau\rfloor\\
		(\tilde\z_t-\bar\z_{t2})/sd_{2}, &  t=\lfloor T\tilde\tau\rfloor+1,...,T.
	\end{cases}\nn
	\eenr
	Here $\bar\z_{t1},\bar\z_{t2}$ and $sd_{1},sd_{2}$ are the piecewise sample means and standard deviations, respectively. 
	
	\vspace{-1.5mm}
	\flushleft{\bf Step 3:} Consider a negative centered and scaled $\chi^2_k$ distribution with $k$ degrees of freedom, i.e., $X=-(\chi^2_{k}-k)/\surd{(2k)}$ and utilize the Kolmogorv-Smirnov (K-S) goodness of fit test to check for the empirical fit between $X$ and the realizations $\tilde\z_t^*,$ $t=1,...,T.$ 
	
	\vspace{-1.5mm}
	\flushleft{\bf Step 4:} Repeat Step 3 on a grid of values for the degrees of freedom $k\in\{1,2,3....\}$ and choose $k$ as the maximizing value of the p-value of the K-S	test. 
	
	\vspace{-1.5mm}
	\noi\rule{\textwidth}{0.5pt}
\end{figure}

An illustration of the fitted distribution using Algorithm 3 is provided in Figure \ref{fig:increment.alg3}. The remainder of this section provides additional numerical results of the simulated experiments discussed in Section \ref{sec:nuisance} of the main article.

\begin{figure}[]
	\centering
	\begin{minipage}[l]{0.35\textwidth}
		\includegraphics[width=0.8\textwidth]{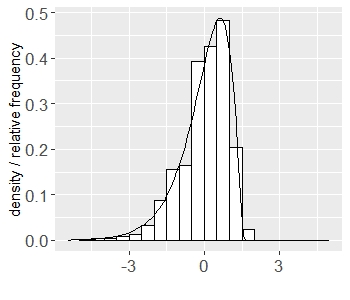}
	\end{minipage}
	\hspace{3mm}
	\begin{minipage}[r]{0.35\textwidth}
		\includegraphics[width=0.8\textwidth]{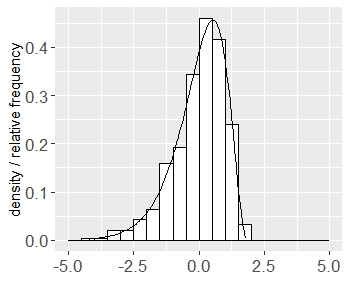}
	\end{minipage}
	\begin{minipage}[l]{0.35\textwidth}
		\includegraphics[width=0.8\textwidth]{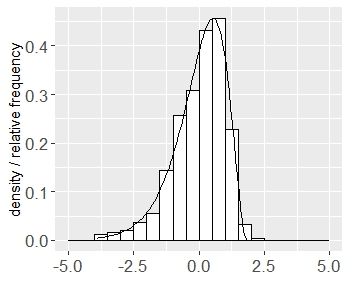}
	\end{minipage}
	\hspace{3mm}
	\begin{minipage}[r]{0.35\textwidth}
		\includegraphics[width=0.8\textwidth]{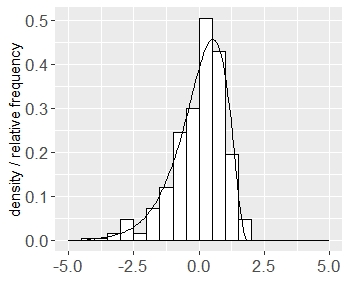}
	\end{minipage}
	\caption{\footnotesize{Empirically fitting a negative centered and scaled $\chi^2_k$ distribution to predicted realizations $\tilde\z_t^^*$ (see, Step 2 of Algorithm 3) for identification of distribution $\cL.$ Figure illustrates histograms of $\tilde\z_t^*,$ along with superimposed densities of fitted negative centered and scaled $\chi^2_k$ distributions with degrees of freedom (df) identified via K-S goodness of fit test. 
			{\it Top panels:} Cases $p=25,$ and $p=50,$ (fitted distributions with df=$5,7,$ respectively. Corresponding p-values of K-S test: $0.90,$ $0.95,$ respectively.) {\it Bottom panels:} Cases $p=150$ and $p=250.$ (fitted distributions with df=$7,8,$ respectively. Corresponding p-values of K-S test: $0.99,$ $0.82,$ respectively.). Data generating process same as that of Section \ref{sec:numerical}, all underlying parameters estimated as described in Appendix \ref{sec:add.numerical}.}}
	\label{fig:increment.alg3}
\end{figure}

\begin{table}[H]
	\centering
	\resizebox{1\textwidth}{!}{
		\begin{tabular}{cclllllll}
			\toprule
			\multicolumn{3}{c}{\multirow{2}{*}{$\tau^0=0.6$}} & \multicolumn{6}{c}{Coverage (Av. margin of error)}                                                                                                                                                \\ \cmidrule{4-9} 
			\multicolumn{3}{c}{}                              & \multicolumn{3}{c}{Non-vanishing}                                                               & \multicolumn{3}{c}{Vanishing}                                                                   \\ \midrule
			$n$   & $p$   & \multicolumn{1}{c}{bias (rmse)}   & \multicolumn{1}{c}{$\al=0.1$} & \multicolumn{1}{c}{$\al=0.05$} & \multicolumn{1}{c}{$\al=0.01$} & \multicolumn{1}{c}{$\al=0.1$} & \multicolumn{1}{c}{$\al=0.05$} & \multicolumn{1}{c}{$\al=0.01$} \\ \midrule
			300   & 25    & 0.018 (0.279)                    & 0.94 (0.006)                 & 0.95 (0.19)                   & 0.99 (0.91)                   & 0.94 (0.187)                 & 0.94 (0.278)                  & 0.94 (0.53)                    \\
			300   & 50    & 0.022 (0.265)                    & 0.97 (0)                     & 0.97 (0.008)                  & 0.99 (0.456)                  & 0.98 (0.091)                 & 0.98 (0.138)                  & 0.98 (0.265)                  \\
			300   & 150   & 0.050 (0.688)                    & 0.98 (0)                     & 0.98 (0.002)                  & 0.98 (0.024)                  & 0.98 (0.03)                  & 0.98 (0.047)                  & 0.98 (0.091)                   \\
			300   & 250   & 6.768 (13.31)                    & 0.65 (0.008)                 & 0.65 (0.19)                   & 0.66 (0.362)                  & 0.65 (0.034)                 & 0.65 (0.053)                  & 0.65 (0.104)                  \\ 
			\midrule
			400   & 25    & 0.030 (0.349)                    & 0.94 (0.002)                 & 0.95 (0.182)                  & 0.98 (0.944)                  & 0.95 (0.193)                 & 0.95 (0.289)                  & 0.95 (0.552)                  \\
			400   & 50    & 0.012 (0.155)                    & 0.98 (0)                     & 0.98 (0.004)                  & 0.99 (0.434)                  & 0.98 (0.092)                 & 0.98 (0.139)                  & 0.98 (0.268)                  \\
			400   & 150   & 0.006 (0.077)                    & 0.99 (0)                     & 0.99 (0)                      & 0.99 (0.006)                  & 0.99 (0.028)                 & 0.99 (0.043)                  & 0.99 (0.084)                  \\
			400   & 250   & 0.480 (3.865)                    & 0.97 (0)                     & 0.97 (0.012)                  & 0.98 (0.03)                   & 0.97 (0.019)                 & 0.97 (0.03)                   & 0.97 (0.058)                  \\ 
			\midrule
			500   & 25    & 0.008 (0.261)                    & 0.94 (0)                     & 0.95 (0.162)                  & 0.99 (0.958)                  & 0.95 (0.195)                 & 0.95 (0.292)                  & 0.95 (0.558)                  \\
			500   & 50    & 0.002 (0.118)                    & 0.98 (0)                     & 0.98 (0.002)                  & 0.99 (0.536)                  & 0.98 (0.093)                 & 0.98 (0.141)                  & 0.98 (0.272)                  \\
			500   & 150   & 0 (0)                            & 1 (0)                        & 1 (0)                         & 1 (0.01)                      & 1 (0.027)                    & 1 (0.042)                     & 1 (0.082)                      \\
			500   & 250   & 0 (0)                            & 1 (0)                        & 1 (0)                         & 1 (0)                         & 1 (0.016)                    & 1 (0.025)                     & 1 (0.049)                      \\ 
			\bottomrule
	\end{tabular}}
	\caption{\footnotesize{Summary of monte-carlo simulation results at $\tau^0=0.60$ based on 500 replications. Bias, rmse and av.margin of error rounded to three decimals, coverage rounded to two decimals.}}
	\label{tab:t03}
\end{table}

\begin{table}[H]
	\centering
	\resizebox{1\textwidth}{!}{
		\begin{tabular}{cclllllll}
			\toprule
			\multicolumn{3}{c}{\multirow{2}{*}{$\tau^0=0.8$}} & \multicolumn{6}{c}{Coverage (Av. margin of error)}                                                                                                                                                \\ \cmidrule{4-9} 
			\multicolumn{3}{c}{}                              & \multicolumn{3}{c}{Non-vanishing}                                                               & \multicolumn{3}{c}{Vanishing}                                                                   \\ \midrule
			$n$   & $p$   & \multicolumn{1}{c}{bias (rmse)}   & \multicolumn{1}{c}{$\al=0.1$} & \multicolumn{1}{c}{$\al=0.05$} & \multicolumn{1}{c}{$\al=0.01$} & \multicolumn{1}{c}{$\al=0.1$} & \multicolumn{1}{c}{$\al=0.05$} & \multicolumn{1}{c}{$\al=0.01$} \\ \midrule
			300   & 25    & 0.088 (0.400)                    & 0.90 (0.018)                 & 0.93 (0.198)                  & 0.97 (0.882)                  & 0.90 (0.185)                 & 0.90 (0.281)                  & 0.90 (0.543)                  \\
			300   & 50    & 0.098 (0.417)                    & 0.93 (0)                     & 0.93 (0.020)                  & 0.96 (0.430)                  & 0.93 (0.095)                 & 0.93 (0.147)                  & 0.93 (0.285)                  \\
			300   & 150   & 7.342 (9.891)                    & 0.35 (0.008)                 & 0.35 (0.314)                  & 0.36 (0.708)                  & 0.35 (0.055)                 & 0.35 (0.083)                  & 0.35 (0.166)                  \\
			300   & 250   & 11.82 (14.71)                    & 0.16 (0.014)                 & 0.16 (0.282)                  & 0.16 (0.83)                   & 0.16 (0.039)                 & 0.16 (0.059)                  & 0.16 (0.118)                  \\ 
			\midrule
			400   & 25    & 0.034 (0.332)                    & 0.92 (0.008)                 & 0.95 (0.212)                  & 0.99 (0.92)                   & 0.93 (0.187)                 & 0.93 (0.284)                  & 0.93 (0.549)                  \\
			400   & 50    & 0.058 (0.279)                    & 0.93 (0)                     & 0.93 (0.020)                  & 0.96 (0.438)                  & 0.94 (0.094)                 & 0.94 (0.145)                  & 0.94 (0.281)                  \\
			400   & 150   & 2.664 (6.291)                    & 0.71 (0.002)                 & 0.71 (0.070)                  & 0.72 (0.238)                  & 0.71 (0.042)                 & 0.71 (0.064)                  & 0.71 (0.126)                  \\
			400   & 250   & 16.87 (17.84)                    & 0.07 (0.006)                 & 0.07 (0.518)                  & 0.07 (1.038)                  & 0.07 (0.047)                 & 0.07 (0.074)                  & 0.07 (0.147)                  \\ 
			\midrule
			500   & 25    & 0.052 (0.379)                    & 0.91 (0.004)                 & 0.93 (0.200)                  & 0.98 (0.940)                  & 0.91 (0.192)                 & 0.91 (0.290)                  & 0.91 (0.560)                   \\
			500   & 50    & 0.028 (0.179)                    & 0.97 (0)                     & 0.97 (0.004)                  & 0.99 (0.536)                  & 0.97 (0.094)                 & 0.97 (0.143)                  & 0.97 (0.278)                  \\
			500   & 150   & 0.146 (1.017)                    & 0.93 (0)                     & 0.93 (0.002)                  & 0.94 (0.066)                  & 0.93 (0.032)                 & 0.93 (0.050)                   & 0.93 (0.097)                  \\
			500   & 250   & 13.23 (16.96)                    & 0.30 (0)                     & 0.30 (0.338)                  & 0.30 (0. 718)                  & 0.30 (0.042)                 & 0.30 (0.066)                  & 0.30 (0.131)                  \\ 
			\bottomrule
	\end{tabular}}
	\caption{\footnotesize{Summary of monte-carlo simulation results at $\tau^0=0.80$ based on 500 replications. Bias, rmse and av.margin of error rounded to three decimals, coverage rounded to two decimals.}}
	\label{tab:t04}
\end{table}

\end{document}